\documentclass[reqno]{amsart}
\usepackage[margin=1in]{geometry}

\usepackage{eurosym}
\usepackage{amsmath}
\usepackage{amssymb,amsthm}
\usepackage{amsfonts,amsrefs}
\usepackage{url,hyperref}
\usepackage{scalerel}
\usepackage{graphicx,mathtools,caption,subcaption}
\usepackage{float}

\usepackage{xcolor}

\usepackage[export]{adjustbox}

\theoremstyle{plain} 
\newtheorem{theorem}{Theorem}[section]
\newtheorem{corollary}[theorem]{Corollary}
\newtheorem{lemma}[theorem]{Lemma}

\theoremstyle{definition}

\newtheorem{remark}[theorem]{Remark}

\usepackage{chngcntr}
\counterwithin{figure}{section}

\usepackage[T3,OT1]{fontenc}
\DeclareSymbolFont{tipa}{T3}{cmr}{m}{n}
\DeclareMathAccent{\invbreve}{\mathalpha}{tipa}{16}

\makeatletter
\newsavebox{\@brx}
\newcommand{\llangle}[1][]{\savebox{\@brx}{\(\m@th{#1\langle}\)}%
  \mathopen{\copy\@brx\mkern2mu\kern-0.9\wd\@brx\usebox{\@brx}}}
\newcommand{\rrangle}[1][]{\savebox{\@brx}{\(\m@th{#1\rangle}\)}%
  \mathclose{\copy\@brx\mkern2mu\kern-0.9\wd\@brx\usebox{\@brx}}}
\makeatother

\usepackage{enumerate}

\title{KBSM of lens spaces $L(p,2)$ and $L(4k,2k+1)$}

\author{Mieczyslaw K. Dabkowski}
\address{Department of Mathematical Sciences, The University of Texas at Dallas, Richardson, TX 75080}
\email{mdab@utdallas.edu}

\author{Cheyu Wu}
\address{Department of Mathematical Sciences, The University of Texas at Dallas, Richardson, TX 75080}
\email{cheyu.wu@utdallas.edu}

\begin{document}

\begin{abstract}
J. Hoste and J. H. Przytycki computed the Kauffman bracket skein module (KBSM) of lens spaces in their papers published in 1993 and 1995. Using a basis for the KBSM of a fibered torus, we construct new bases for the KBSMs of two families of lens spaces: $L(p,2)$ and $L(4k,2k+1)$ with $k\neq 0$. For KBSM of $L(0,1) = {\bf S}^{2}\times S^{1}$, we find a new generating set that yields its decomposition into a direct sum of cyclic modules.
\end{abstract}

\maketitle

\section{Introduction}
\label{s:intro}

The Kauffman bracket skein module\footnote{Skein modules were introduced by J.~H.~Przytycki \cite{Prz1991} in 1987, and independently by V.~G.~Turaev \cite{Tur1990} in 1988. The skein module based on the Kauffman bracket skein relation (see \cite{KLH1987}) is called the Kauffman bracket skein module.} (KBSM) of lens spaces was computed in \cite{HP1993} and \cite{HP1995}, with a new proof given for the special cases of $L(p,1)$ and $L(0,1)$ in \cite{M2011b}. This paper builds on the results of \cite{DW2025} to construct a new basis for the KBSM of two families of lens spaces: $L(p,2)$ and $L(4k,2k+1)$, where $k \in \mathbb{Z}$ and $k \neq 0$. For KBSM of $L(0,1)$ we construct a new generating set which leads to its natural decomposition into a direct sum of cyclic modules. 

\medskip

A framed link in an oriented $3$-manifold $M$ is a disjoint union of smoothly embedded circles, each equipped with a non-zero normal vector field. We fix an invertible element $A$ of a commutative ring $R$ with identity, and let $R\mathcal{L}^{fr}$ be the free $R$-module with basis $\mathcal{L}^{fr}$, where $\mathcal{L}^{fr}$ is the set of ambient isotopy classes of framed links in $M$ (including the empty set as a framed link). Let $S_{2,\infty}$ be the submodule of $R\mathcal{L}^{fr}$ generated by all $R$-linear combinations:
\begin{equation*}
L_{+} - AL_{0} - A^{-1}L_{\infty} \quad \text{and} \quad L \sqcup T_{1} + (A^{-2}+A^{2})L,
\end{equation*}
where framed links $L_{+},\, L_{0},\, L_{\infty}$ are identical outside of a $3$-ball and differ inside of it as on the left of Figure~\ref{fig:skeinrelation}; $L\sqcup T_{1}$ on the right of Figure~\ref{fig:skeinrelation} is the disjoint union of $L$ and the trivial framed knot $T_{1}$ (i.e., $T_{1}$ is in a $3$-ball disjoint from $L$). The \emph{Kauffman bracket skein module} of $M$ is defined as the quotient module of $R\mathcal{L}^{fr}$ by $S_{2,\infty}$, i.e.,
\begin{equation*}
\mathcal{S}_{2,\infty}(M;R,A) = R\mathcal{L}^{fr}/S_{2,\infty}.
\end{equation*}

\begin{figure}[ht]
\centering
\includegraphics[scale=0.6]{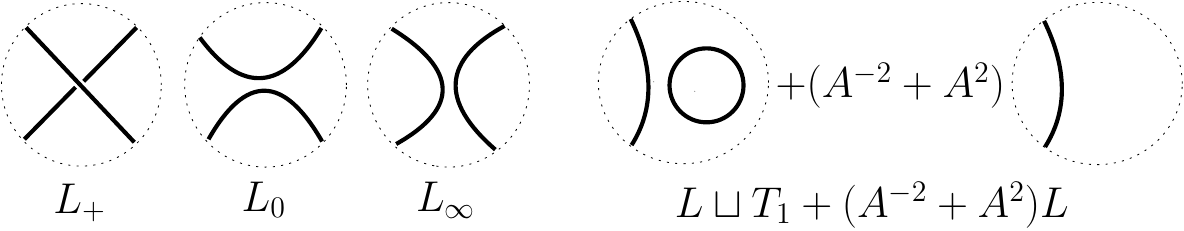}
\caption{Skein triple $L_{+}$, $L_{0}$, $L_{\infty}$ and $L\sqcup T_{1} + (A^{-2}+A^{2})L$}
\label{fig:skeinrelation}
\end{figure}

We organize this paper as follows. In Section~\ref{s:ambient_isotopy_In_Lens_Spaces}, we introduce a model for lens spaces that will be used throughout the paper. This model enables a representation of framed links and their ambient isotopy using arrow diagrams, and the arrow moves on ${\bf S}^{2}$ with two marked points (see Theorem~\ref{thm:AmbientIsotopiesInMBeta1Beta2}). In Section~\ref{s:Summary}, we provide a brief summary of the results of \cite{DW2025} that are relevant to this paper. In Section~\ref{s:LensSpaceLbeta2}, we construct a new basis for the KBSM of $L(\beta,2)$, where $\beta$ is an odd integer. In Section~\ref{s:CaseWithTwoFibers_Nu0DifferntThenMinus_1}, we find a new basis for the KBSM of $L(4k,2k+1)$, where $k \neq 0$. Finally, in Section~\ref{s:CaseWithTwoFibers_NuEqualMinus_1}, we construct a new generating set for the KBSM of $L(0,1) = {\bf S}^{2} \times S^{1}$.

\section{Ambient isotopy of framed links in \texorpdfstring{$M_{2}(\beta_{1})$ and $M_{2}(\beta_{1},\beta_{2})$}{M\unichar{8322}(\unichar{0946}\unichar{8321}) and M\unichar{8322}(\unichar{0946}\unichar{8321}, \unichar{0946}\unichar{8322})}}
\label{s:ambient_isotopy_In_Lens_Spaces}

Let $M(0;(\alpha_{1},\beta_{1}),(\alpha_{2},\beta_{2}))$ be a $3$-manifold obtained by $(\alpha_{i},\beta_{i})$-Dehn filling of boundary tori of a product ${\bf A}^{2} \times S^{1}$ of an annulus ${\bf A}^{2}$ and a circle $S^{1}$ along the curves $(\alpha_{i},\beta_{i})$, where $\alpha_{i} > 0$, $\gcd(\alpha_{i}, \beta_{i}) = 1$ for $i = 1,2$. In this paper, we consider two special cases: 
\begin{equation*}
M_{2}(\beta_{1}) = M(0;(2,\beta_{1}),(1,0))\,\, \text{and}\,\,M_{2}(\beta_{1}, \beta_{2}) = M(0;(2, \beta_{1}), (2, \beta_{2}))
\end{equation*}
From \cite{JN1983} (see Theorem~4.4), we know that for $p = \alpha_{1} \beta_{2} + \alpha_{2} \beta_{1}$ and $q = s \alpha_{1} + r \beta_{1}$, where $s \alpha_{2} - r \beta_{2} = 1$,
\begin{equation*}
M(0;(\alpha_{1},\beta_{1}),(\alpha_{2},\beta_{2})) \cong L(p,q).
\end{equation*}
For $\alpha_{i} = 2$ and $\nu_{i} = \lfloor \frac{\beta_{i}}{2}\rfloor$, $i = 1,2$, if $\nu_{0} = \nu_{1} + \nu_{2}$, then by Theorem~4.2 of \cite{JN1983},
\begin{equation*}
M_{2}(\beta_{1},\beta_{2}) \simeq L(4k,2k+1),
\end{equation*}
where $k = \nu_{0} + 1$. Thus, in the special case of $\nu_{0} = -1$, $M_{2}(\beta_{1},\beta_{2}) \simeq L(0,1) = {\bf S}^{2}\times S^{1}$.
\medskip

We define \emph{framed link} and \emph{generic framed link} in $M_{2}(\beta_{1})$ or $M_{2}(\beta_{1}, \beta_{2})$ as in \cite{DW2025}, and observe that generic framed links in $M_{2}(\beta_{1})$ or $M_{2}(\beta_{1}, \beta_{2})$ can be represented using arrow diagrams in ${\bf S}^{2}$ with two marked points $\beta_{1}$ and $\beta_{2}$ correspond to singular fibers. In this paper, we represent generic framed links on a $2$-disk ${\bf D}^{2}$ centered at $\beta_{1}$, with its boundary identified with the second marked point $\beta_{2}$. We will denote this disk by $\hat{{\bf S}}^{2}$ (see Figure~\ref{fig:S2withtwoconepoints}).

\begin{figure}[H]
\centering
\includegraphics[scale=0.7]{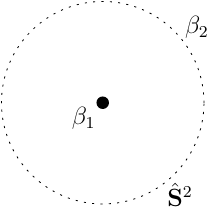}
\caption{Disk $\hat{\bf S}^{2}$ with marked points $\beta_{1}$ and $\beta_{2}$}
\label{fig:S2withtwoconepoints}
\end{figure}

It follows from Corollary~6.3 of \cite{Hud1969}, that every ambient isotopy of links (framed links) in $M_{2}(\beta_{1})$ or $M_{2}(\beta_{1},\beta_{2})$ are compositions of \emph{moves} either in a normal cylinder $N$ inside ${\bf A}^{2}\times S^{1}$ or a $2$-handle $H$ attached along $(2,\beta_{i})$-curves in its boundary called $2$-\emph{handle slides}. A move in $N$ corresponds to one of $\Omega_{1}-\Omega_{5}$-moves (see Figure~\ref{fig:ArrowMoves}) on $\hat{\bf S}^{2}$. Furthermore, it follows from Lemma~2.1 of \cite{DW2025} that a $2$-handle slide corresponds to an $S_{\beta_{i}}$-move on $\hat{{\bf S}}^{2}$ (see Figure~\ref{fig:SBeta1Beta2MovesOnhatS2}). When $\beta_{2} = 0$, $S_{\beta_{2}}$-move on $\hat{{\bf S}}^{2}$ is shown in Figure~\ref{fig:OmegaInfinityOnhatS} and we will denote it by $\Omega_{\infty}$.

\begin{figure}[ht]
\centering
\includegraphics[scale=0.8]{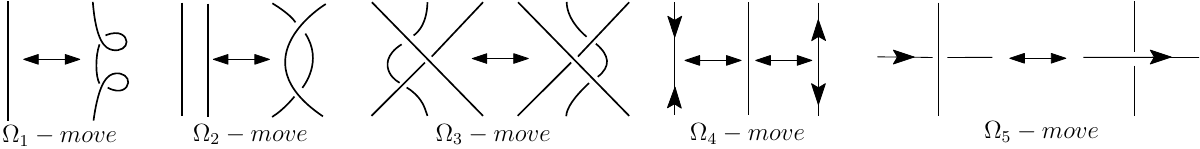}
\caption{Arrow moves $\Omega_{1}-\Omega_{5}$ on ${\bf A}^{2}$}
\label{fig:ArrowMoves}
\end{figure}

\begin{figure}[ht]
\centering
\includegraphics[scale=0.75]{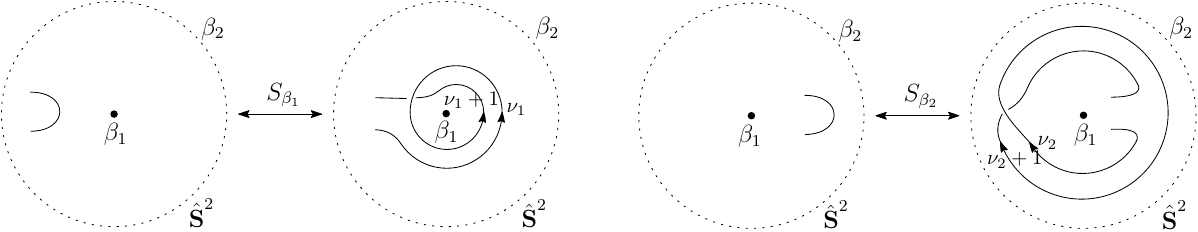}
\caption{$S_{\beta_{1}}$ and $S_{\beta_{2}}$-moves on $\hat{{\bf S}}^{2}$}
\label{fig:SBeta1Beta2MovesOnhatS2}
\end{figure}

\begin{figure}[ht]
\centering
\includegraphics[scale=0.6]{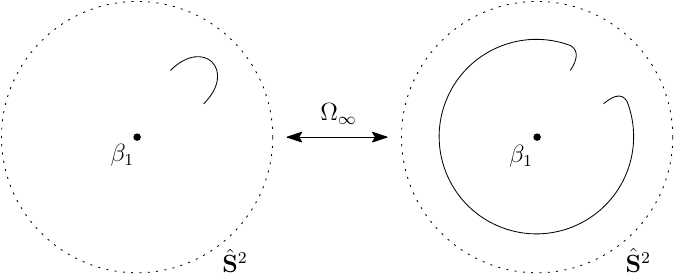}
\caption{$\Omega_{\infty}$-move on $\hat{\bf S}^{2}$}
\label{fig:OmegaInfinityOnhatS}
\end{figure}

\begin{theorem}
\label{thm:AmbientIsotopiesInMBeta1Beta2}
Let $L_{1}$ and $L_{2}$ be generic links either in $M_{2}(\beta_{1})$ or $M_{2}(\beta_{1},\beta_{2})$.
\begin{itemize}
\item[(i)] $L_{1}$ and $L_{2}$ are ambient isotopic in $M_{2}(\beta_{1})$ if and only if their arrow diagrams differ on $\hat{\bf S}^{2}$ by a finite sequence of $\Omega_{1}-\Omega_{5}$, $S_{\beta_{1}}$, and $\Omega_{\infty}$-moves.
\item[(ii)] $L_{1}$ and $L_{2}$ are ambient isotopic in $M_{2}(\beta_{1},\beta_{2})$ if and only if their arrow diagrams differ on $\hat{\bf S}^{2}$ by a sequence of $\Omega_{1}-\Omega_{5}$ and $S_{\beta_{i}}$-moves, $i = 1,2$.
\end{itemize}
\end{theorem}

\section{Preliminaries}
\label{s:Summary}

We begin this section with a brief summary of the relevant results of \cite{DW2025}. Let ${\bf D}^{2}$ be a $2$-disk, ${\bf A}^{2}$ be an annulus, and ${\bf D}^{2}_{\beta_{1}}$ be a $2$-disk with marked point $\beta_{1}$. Arrow diagrams in ${\bf D}^{2}$, ${\bf A}^{2}$, and ${\bf D}^{2}_{\beta_{1}}$ can naturally be regarded as arrow diagrams in $\hat{\bf S}^{2}$. Therefore, the curves $t_{m}$, $\lambda$, $\lambda^{n}$, $t_{m,n}$, $x_{m}$, and $(x_{m})^{n}$ introduced in \cite{DW2025} can also be viewed as the curves in $\hat{\bf S}^{2}$ shown in Figure~\ref{fig:CurvesLambdaXnOnS2}.
\begin{figure}[ht]
\centering
\includegraphics[scale=0.8]{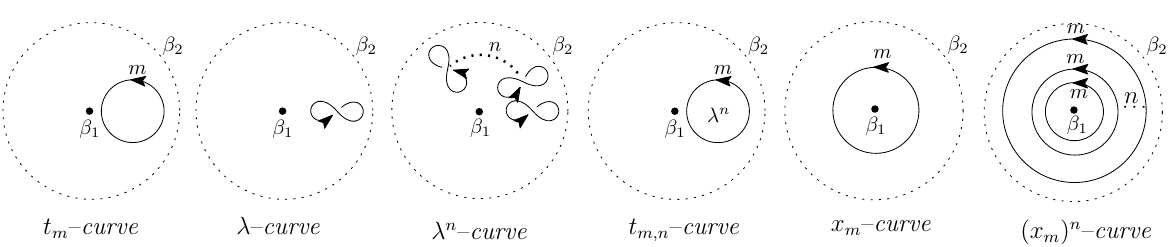}
\caption{Curves $t_{m}$, $\lambda$, $\lambda^{n}$, $t_{m,n}$ $x_{m}$, and $(x_{m})^n$ on $\hat{\bf S}^{2}$, $m\in \mathbb Z$, $n\geq 0$}
\label{fig:CurvesLambdaXnOnS2}
\end{figure}

We set $R = \mathbb{Z}[A^{\pm 1}]$ for the remainder of this paper. In \cite{DW2025}, we introduced three families of polynomials $\{P_{m}\}_{m \in \mathbb{Z}}$, $\{Q_{m}\}_{m \in \mathbb{Z}}$, and $\{ P_{m,k}\mid m \in \mathbb{Z},\, k \geq 0\}$. The first one (see \cite{DW2025}, p.5) is determined by the relation\footnote{This is a modified version of the relation defining $\{P_{m}\}_{m \in \mathbb{Z}}$ introduced in \cite{MD2009}.}
\begin{equation*} 
P_{m} - A \lambda P_{m-1} + A^{2} P_{m-2} = 0,
\end{equation*}
with $P_{0} = -A^{2}-A^{-2}$, $P_{1} = -A^{3} \lambda$. The second one (see Definition~3.3 of \cite{DW2025}), is determined by relation
\begin{equation*}
Q_{0} = 0, \quad Q_{1} = 1, \quad \text{and} \quad Q_{m+2} = \lambda Q_{m+1} - Q_{m}
\end{equation*}
for $m \geq 0$, and $Q_{m} = -Q_{-m}$ for $m < 0$. We note that for $m > 0$, the degree of $Q_{m}$ is $\deg (Q_{m}) = m-1$ and its leading coefficient is $1$. Moreover, as we showed in Lemma~3.4 of \cite{DW2025},
\begin{equation}
\label{eqn:rel_Pn}
P_{m} = -A^{m+2} Q_{m+1} + A^{m-2} Q_{m-1}
\end{equation}
for any $m \in \mathbb{Z}$. The third family\footnote{This is also a modified version of family $\{P_{m,k}\mid m \in \mathbb{Z},\, k \geq 0\}$ introduced in \cite{MD2009}.} is defined by $P_{m,0} = P_{m}$ and for $k \geq 1$,
\begin{equation*}
P_{m,k} = A P_{m+1,k-1} + A^{-1} P_{m-1,k-1}.
\end{equation*}

Let $\mathcal{D}({\hat{\bf S}^{2}})$ be the set of all equivalence classes of arrow diagrams (including empty arrow diagram) modulo $\Omega_{1}-\Omega_{5}$, $S_{\beta_{1}}$, and $\Omega_{\infty}$-moves, or $\Omega_{1}-\Omega_{5}$, $S_{\beta_{1}}$, and $S_{\beta_{2}}$-moves (this will be clear from the context). We denote by $R\mathcal{D}({\hat{\bf S}^{2}})$ the free $R$-module with basis $\mathcal{D}({\hat{\bf S}^{2}})$ and let $S_{2,\infty}(\hat{\bf S}^{2})$ be its free $R$-submodule generated by all $R$-linear combinations:
\begin{equation*}
D_{+}-AD_{0}-A^{-1}D_{\infty}\,\,\text{and}\,\,D\sqcup T_{1} + (A^{2}+A^{-2})D,    
\end{equation*}
where $D_{+}$, $D_{0}$, $D_{\infty}$, and $D\sqcup T_{1}$ are arrow diagrams in Figure~\ref{fig:SkeinTripleOfDiagrams}.

\begin{figure}[ht]
\centering
\includegraphics[scale=0.6]{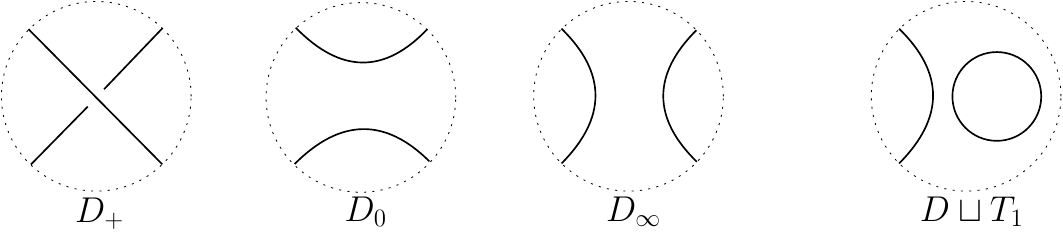}
\caption{Skein triple $D_{+}$, $D_{0}$, $D_{\infty}$ and disjoint union $D\sqcup T_{1}$}
\label{fig:SkeinTripleOfDiagrams}
\end{figure}

Therefore, we can define two corresponding quotient modules $S\mathcal{D}_{\nu_{1}}$ and $S\mathcal{D}_{\nu_{1},\nu_{2}}$ of $R\mathcal{D}({\hat{\bf S}^{2}})$ by $S_{2,\infty}(\hat{\bf S}^{2})$. We show that the first determines the KBSM of $M_{2}(\beta_{1})$ and the second one gives the KBSM of $M_{2}(\beta_{1},\beta_{2})$.

\medskip

An arrow diagram $D$ in $\hat{\bf S}^{2}$ contained in a $2$-disk ${\bf D}^{2}$ can be expressed in $S\mathcal{D}_{\nu_{1},\nu_{2}}$ (or $S\mathcal{D}_{\nu_{1}}$) as a $R$-linear combination of $\lambda^{k}$ ($k\geq 0$) using a modified version of the bracket $\langle\cdot \rangle_{r}$ (also denoted by $\langle\cdot \rangle_{r}$ in \cite{DW2025}) defined in \cite{MD2009} (see Definition~3.5). It follows from Proposition~3.7 of \cite{MD2009} that $\langle D \rangle_{r} = \langle D' \rangle_{r}$, whenever arrow diagrams $D$ and $D'$ are related by a finite sequence of $\Omega_{1}-\Omega_{5}$-moves on ${\bf D}^{2}$. Furthermore, as noted in \cite{DW2025}, $\langle t_{m} \rangle_{r} = P_{m}$ and $\langle t_{m,n} \rangle_{r} = P_{m,n}$.

Given an arrow diagram $D$ in $\hat{\bf S}^{2}$, we define $\langle D \rangle$ and $\llangle D \rrangle$ analogously to those defined for an arrow diagram in ${\bf A}^{2}$ (or ${\bf D}^{2}_{\beta}$) in \cite{DW2025}.

\begin{figure}[H]
\centering
\includegraphics[scale=0.7]{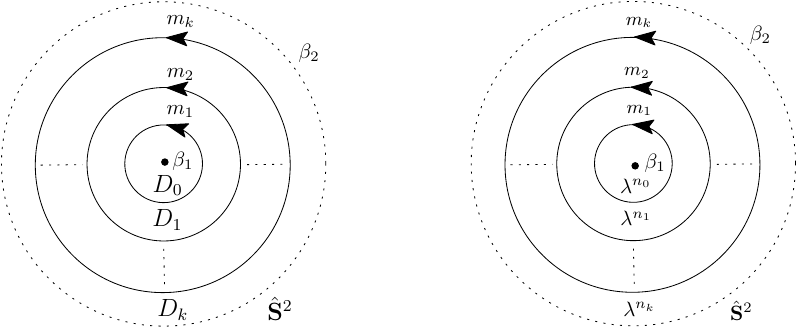}
\caption{Arrow diagram $D$ in $\hat{\bf S}^{2}$ without crossings and $\lambda^{n_{0}}x_{m_{1}}\lambda^{n_{1}} \cdots \lambda^{n_{k-1}}x_{m_{k}}\lambda^{n_{k}}$}
\label{fig:ArrowDiagramOnhatS2NoCrAndGamma}
\end{figure}

Let
\begin{equation*}
\Gamma = \{\lambda^{n_{0}}x_{m_{1}}\lambda^{n_{1}} \cdots \lambda^{n_{k-1}}x_{m_{k}}\lambda^{n_{k}} \mid n_{i} \geq 0, \ m_{i} \in \mathbb{Z}, \ \text{and} \ k \geq 0 \},
\end{equation*}
where $\lambda^{n_{0}}x_{m_{1}}\lambda^{n_{1}} \cdots \lambda^{n_{k-1}}x_{m_{k}}\lambda^{n_{k}}$ is an arrow diagram on the right of Figure~\ref{fig:ArrowDiagramOnhatS2NoCrAndGamma}. For an arrow diagram without crossings $D = D_{0}x_{m_{1}}D_{1}\ldots D_{k-1}x_{m_{k}}D_{k}$ in $\hat{\bf S}^{2}$ (see left of Figure~\ref{fig:ArrowDiagramOnhatS2NoCrAndGamma}) we define $\llangle D \rrangle_{\Gamma}$ as in \cite{DW2025}. Let
\begin{equation*}
\Sigma^{\prime}_{\nu_{1}} = \{\lambda^{n}, x_{\nu_{1}}\lambda^{n} \mid n \geq 0\}\subset \Gamma, \,\, \nu_{1} = \lfloor \frac{\beta_{1}}{2} \rfloor,
\end{equation*}
and, for each $w\in \Gamma$, we define $\llangle w \rrangle_{\Sigma^{\prime}_{\nu_{1}}}$ as in \cite{DW2025}. As we showed (see Theorem~4.9 of \cite{DW2025}), the KBSM of $(\beta,2)$-fibered torus $S\mathcal{D}({\bf D}^{2}_{\beta_{1}})$ is a free $R$-module with the basis $\Sigma^{\prime}_{\nu_{1}}$. 
In this paper, we will use the following properties of $\llangle \cdot \rrangle_{\Sigma^{\prime}_{\nu_{1}}}$.

\begin{lemma}[Lemma~4.3, \cite{DW2025}]
\label{lem:annulus_for_any_kn_bracket_2_beta0}
For any $w_{1}x_{m}w_{2} \in \Gamma$ with $m \in \mathbb{Z}$ and $k \in \mathbb{Z}$:
\begin{equation}
\label{eqn:rel_xm_to_xkQ}
\llangle w_{1}x_{m}w_{2} \rrangle_{\Sigma^{\prime}_{\nu_{1}}} = -A^{m-k}\llangle w_{1}x_{k}Q_{m-k-1}w_{2} \rrangle_{\Sigma^{\prime}_{\nu_{1}}} + A^{m-k-1}\llangle w_{1}x_{k+1}Q_{m-k}w_{2} \rrangle_{\Sigma^{\prime}_{\nu_{1}}},
\end{equation}
and
\begin{equation}
\label{eqn:rel_xm_to_Qxk}
\llangle w_{1}x_{m}w_{2} \rrangle_{\Sigma^{\prime}_{\nu_{1}}} = -A^{k-m}\llangle w_{1}Q_{m-k-1}x_{k}w_{2} \rrangle_{\Sigma^{\prime}_{\nu_{1}}} + A^{k-m+1}\llangle w_{1}Q_{m-k}x_{k+1}w_{2} \rrangle_{\Sigma^{\prime}_{\nu_{1}}}.
\end{equation}
\end{lemma}

\begin{lemma}[Lemma~4.4, \cite{DW2025}]
\label{lem:annulus_for_any_kn_bracket_2_beta}
Let $\Delta_{t}^{+},\Delta_{t}^{-},\Delta_{x}^{+},\Delta_{x}^{-}$ be finite subsets of $R \times \Gamma \times \Gamma \times \mathbb{Z}$, and define
\begin{equation*}
\Theta_{t}^{+}(k,n) = \sum_{(r,w_{1},w_{2},v) \in \Delta_{t}^{+}} r\llangle w_{1}P_{n+v,k}w_{2} \rrangle_{\Sigma^{\prime}_{\nu_{1}}}, \quad
\Theta_{t}^{-}(k,n) = \sum_{(r,w_{1},w_{2},v) \in \Delta_{t}^{-}} r\llangle w_{1}P_{-n+v}\lambda^{k}w_{2} \rrangle_{\Sigma^{\prime}_{\nu_{1}}},
\end{equation*}
\begin{equation*}
\Theta_{x}^{+}(k,n) = \sum_{(r,w_{1},w_{2},v) \in \Delta_{x}^{+}} r\llangle w_{1}\lambda^{k}x_{n+v}w_{2} \rrangle_{\Sigma^{\prime}_{\nu_{1}}}, \quad
\Theta_{x}^{-}(k,n) = \sum_{(r,w_{1},w_{2},v) \in \Delta_{x}^{-}} r\llangle w_{1}x_{-n+v}\lambda^{k}w_{2} \rrangle_{\Sigma^{\prime}_{\nu_{1}}},
\end{equation*}
and 
\begin{equation*}
\Theta_{t,x}(k,n) = \Theta_{t}^{+}(k,n) + \Theta_{t}^{-}(k,n) + \Theta_{x}^{+}(k,n)+ \Theta_{x}^{-}(k,n).
\end{equation*}
If either \textup{(1)} $\Theta_{t,x}(0,n) = 0$ for all $n \in \mathbb{Z}$ or \textup{(2)} $\Theta_{t,x}(k,n_{0}) = \Theta_{t,x}(k,n_{0}+1) = 0$ for all $k \geq 0$ and a fixed $n_{0} \in \mathbb{Z}$, then $\Theta_{t,x}(k,n) = 0$ for any $k \geq 0$ and $n \in \mathbb{Z}$.
\end{lemma}

For an arrow diagram $D$ in $\hat{\bf S}^{2}$ we also define as in \cite{DW2025},
\begin{equation*}
\phi_{\beta_{1}}(D) = \llangle \llangle \llangle D \rrangle \rrangle_{\Gamma} \rrangle_{\Sigma^{\prime}_{\nu_{1}}}
\end{equation*} 
and we note that by Lemma~4.2 and Lemma~4.8 of \cite{DW2025},
\begin{equation}
\label{eqn:Compositionh4h2h1}
\phi_{\beta_{1}}(D - D') = 0
\end{equation}
for any arrow diagrams $D,D'$ on $\hat{\bf S }^{2}$, which differ by $\Omega_{1}-\Omega_{5}$ and $S_{\beta_{1}}$-moves.

\medskip

Let $\{F_{m}\}_{m \in \mathbb{Z}}$ and $\{R_{m}\}_{m \in \mathbb{Z}}$ be families of polynomials in $R[\lambda]$ defined by
\begin{equation*}
F_{m} = A^{-m} Q_{m+1} + A^{-m+2} Q_{m} \quad \text{and} \quad R_{m} = A^{-1} P_{m-1} - A^{-2} P_{m}.
\end{equation*}

\begin{remark}
\label{rem:DegAndCoeffOf_F} 
One checks that $\deg (F_{m}) = \max\{m, -m-1\}$, the leading coefficient of $F_{m}$ is $A^{-m}$ if $m \geq 0$ and $-A^{-m+2}$ otherwise, and
\begin{equation}
\label{eqn:rel_F_to_P}
P_{m} = -A^{-2}F_{-m} + A^{-1}F_{-m-1}.
\end{equation}
One also verifies that $\deg (R_{m}) = \max\{m, 1-m\}$, the leading coefficient of $R_{m}$ is $A^{m}$ if $m \geq 1$ and $-A^{m-4}$ otherwise.
\end{remark}

\begin{lemma} 
\label{lem:kbsm_xm_xvxm}
In $S\mathcal{D}({\bf D}^{2}_{\beta_{1}})$, for all $m \in \mathbb{Z}$ and $w_{x} \in \Gamma$:
\begin{equation}
\label{eqn:kbsm_xm}
x_{m}w_{x} = x_{\nu_{1}} F_{\nu_{1}-m} w_{x}
\end{equation}
and
\begin{equation}
\label{eqn:kbsm_xvxm}
x_{\nu_{1}} x_{m} w_{x} = R_{m-\nu_{1}} w_{x}.
\end{equation}
\end{lemma}

\begin{figure}[ht]
\centering
\includegraphics[scale=0.79]{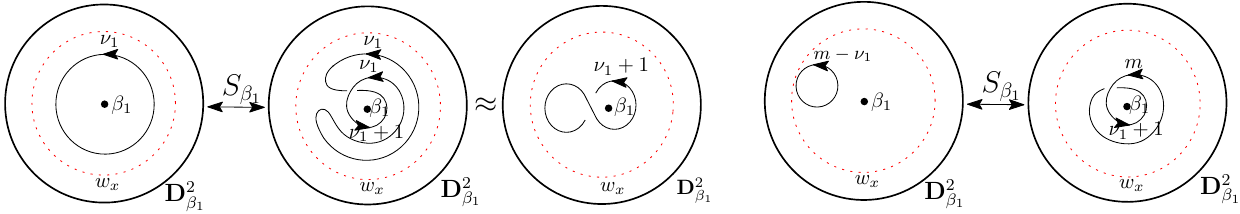}
\caption{$S_{\beta_{1}}$-moves on ${\bf D}^{2}_{\beta_{1}}$ for $x_{\nu_{1}}w_{x}$ and $t_{m-\nu_{1}}w_{x}$ curves}
\label{fig:SBeta_1_Moves_On_X_LensSpece}
\end{figure}

\begin{proof}
Since curves on the left of Figure~\ref{fig:SBeta_1_Moves_On_X_LensSpece} are related by $S_{\beta_{1}}$-move on ${\bf D}^{2}_{\beta_{1}}$, after applying Kauffman bracket skein relations, in $S\mathcal{D}({\bf D}^{2}_{\beta_{1}})$:
\begin{equation*}
x_{\nu_{1}} w_{x} = A x_{\nu_{1}+1} w_{x} + A^{-1} x_{\nu_{1}+1} t_{0} w_{x} = -A^{-3} x_{\nu_{1}+1} w_{x}
\end{equation*}
or equivalently,
\begin{equation}
\label{eqn:pf_kbsm_xm_xvxm}
x_{\nu_{1}+1} w_{x} = -A^{3} x_{\nu_{1}} w_{x}.
\end{equation}
Since \eqref{eqn:rel_xm_to_xkQ} holds for $\llangle \cdot \rrangle_{\Sigma^{\prime}_{\nu_{1}}}$, it is also true in $S\mathcal{D}({\bf D}^{2}_{\beta_{1}})$. Therefore,
\begin{eqnarray*}
x_{m} w_{x} &=& -A^{m-\nu_{1}} x_{\nu_{1}} Q_{m-\nu_{1}-1} w_{x} + A^{m-\nu_{1}-1} x_{\nu_{1}+1} Q_{m-\nu_{1}} w_{x} \\
&=& -A^{m-\nu_{1}} x_{\nu_{1}} Q_{m-\nu_{1}-1} w_{x} - A^{m-\nu_{1}+2} x_{\nu_{1}} Q_{m-\nu_{1}} w_{x} \\
&=& x_{\nu_{1}} F_{\nu_{1}-m} w_{x},
\end{eqnarray*}
where the second equality is due to \eqref{eqn:pf_kbsm_xm_xvxm}. 

The curves on the right of Figure~\ref{fig:SBeta_1_Moves_On_X_LensSpece} are related by $S_{\beta_{1}}$-move on ${\bf D}^{2}_{\beta_{1}}$. Therefore, after applying Kauffman bracket skein relation, in $S\mathcal{D}({\bf D}^{2}_{\beta_{1}})$:
\begin{equation*}
t_{m-\nu_{1}}w_{x} = A t_{m-\nu_{1}-1}w_{x} + A^{-1} x_{\nu_{1}+1} x_{m}w_{x} = A t_{m-\nu_{1}-1}w_{x} - A^{2} x_{\nu_{1}} x_{m}w_{x},
\end{equation*}
where the last equality is due to \eqref{eqn:pf_kbsm_xm_xvxm}. Since in $S\mathcal{D}({\bf D}^{2}_{\beta_{1}})$, $t_{m}w_{x} = P_{m}w_{x}$ for any $m$, using the definition of $R_{m}$, we see that equation \eqref{eqn:kbsm_xvxm} follows.
\end{proof}

\begin{remark}
\label{rem:Prop_Of_Bracket_Sigma}
We note that the statement of Lemma~\ref{lem:kbsm_xm_xvxm} also holds for $S\mathcal{D}_{\nu_{1}}$ and $S\mathcal{D}_{\nu_{1},\nu_{2}}$ in place of $S\mathcal{D}({\bf D}^{2}_{\beta_{1}})$. Furthermore, it follows from Lemma~\ref{lem:kbsm_xm_xvxm} and \eqref{eqn:Compositionh4h2h1} that
\begin{equation}
\label{eqn:kbsm_xm_Bracket_Sigma}
\llangle x_{m} w_{x} \rrangle_{\Sigma^{\prime}_{\nu_{1}}} = \llangle x_{\nu_{1}} F_{\nu_{1}-m} w_{x} \rrangle_{\Sigma^{\prime}_{\nu_{1}}}
\end{equation}
and
\begin{equation}
\label{eqn:kbsm_xvxm_Bracket_Sigma}
\llangle x_{\nu_{1}} x_{m} w_{x} \rrangle_{\Sigma^{\prime}_{\nu_{1}}} = \llangle R_{m-\nu_{1}} w_{x} \rrangle_{\Sigma^{\prime}_{\nu_{1}}}.
\end{equation}
\end{remark}

\section{Lens spaces \texorpdfstring{$L(\beta_{1},2)$}
{L(\unichar{0946}\unichar{8321}, 2)}}
\label{s:LensSpaceLbeta2}
As we noted in Section~\ref{s:ambient_isotopy_In_Lens_Spaces}, we can represent links in $M_{2}(\beta_{1}) = L(\beta_{1},2)$ by arrow diagrams in $\hat{\bf S}^{2}$ and, by Theorem~\ref{thm:AmbientIsotopiesInMBeta1Beta2}, their ambient isotopies by a finite sequence of $\Omega_{1}-\Omega_{5}$ (see Figure~\ref{fig:ArrowMoves}), $S_{\beta_{1}}$, and $\Omega_{\infty}$ moves (see Figure~\ref{fig:SBeta2Omega_inf_Moves}).
\begin{figure}[H]
\centering
\includegraphics[scale=0.7]{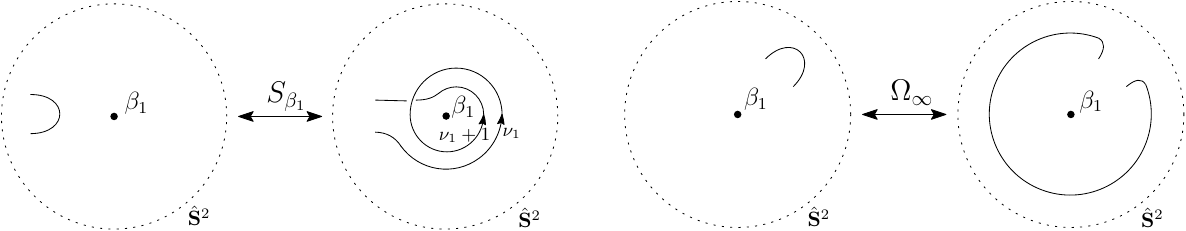}
\caption{$S_{\beta_{1}}$ and $\Omega_{\infty}$-moves on $\hat{\bf S}^{2}$}
\label{fig:SBeta2Omega_inf_Moves}
\end{figure}

Let $\kappa = \max\{\nu_{1}+1, -\nu_{1}\}$ and
\begin{equation*}
\Lambda_{\nu_{1}} = \{\lambda^{n} \mid 0\leq n\leq \kappa -1 \} \subset \Sigma^{\prime}_{\nu_{1}}.    
\end{equation*}
In this section, we show that:
\begin{equation*}
S\mathcal{D}_{\nu_{1}} \cong R\Lambda_{\nu_{1}}.
\end{equation*}

\begin{lemma}
\label{lem:rel_F_and_P}
In $S\mathcal{D}_{\nu_{1}}$, for all $m \in \mathbb{Z}$,
\begin{equation*}
x_{\nu_{1}}F_{\nu_{1}-m} = t_{-m}.
\end{equation*}
\end{lemma}

\begin{proof}
Arrow diagrams on the left and the right of Figure~\ref{fig:Omega_inf_On_X} are related by $\Omega_{\infty}$-move, so by \eqref{eqn:kbsm_xm} in $S\mathcal{D}_{\nu_{1}}$
\begin{equation*}
t_{-m} = x_{m} = x_{\nu_{1}} F_{\nu_{1}-m}.
\end{equation*}
\end{proof}

\begin{figure}[H]
\centering
\includegraphics[scale=0.6]{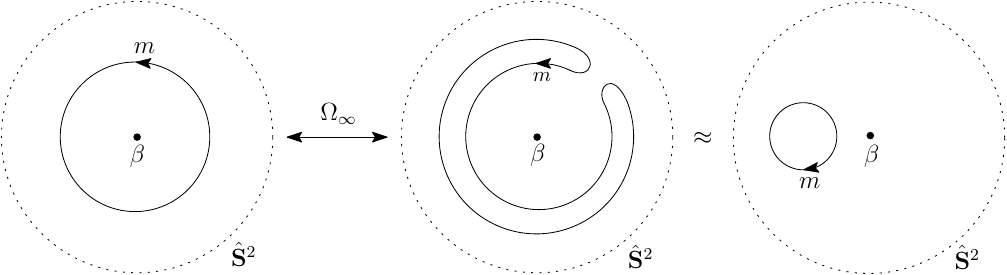}
\caption{$\Omega_{\infty}$-move on $x_{m}$-curve}
\label{fig:Omega_inf_On_X}
\end{figure}

Using Lemma~\ref{lem:rel_F_and_P}, we define a bracket $\langle \cdot \rangle_{\star}$ for $w \in R\Sigma^{\prime}_{\nu_{1}}$ as follows:
\begin{enumerate}[(a)]
\item for $w = \sum_{w' \in S} r_{w'}w'$, $S$ is a finite subset of $\Sigma^{\prime}_{\nu_{1}}$ with at least two elements and $r_{w'}\in R$, let 
\begin{equation*}
\langle w \rangle_{\star} = \sum_{w'\in S} r_{w'}\langle w' \rangle_{\star},
\end{equation*}
\item If $\nu_{1} \geq 0$, let
    \begin{enumerate}[(b1)]
    \item if $w = \lambda^{n}$ and $n < \nu_{1} +1$, then $\langle w \rangle_{\star} = w$,
    \item if $w = \lambda^{n}$, $n \geq \nu_{1}+1$ then 
    \begin{equation*}
     \langle w \rangle_{\star} = \langle \lambda^{n} + A^{n+2}P_{-n} \rangle_{\star} - A^{n+2} \langle x_{\nu_{1}}F_{\nu_{1}-n} \rangle_{\star};   
    \end{equation*}
    \item if $w = x_{\nu_{1}}\lambda^{n}$, then 
    \begin{equation*}
    \langle w \rangle_{\star} = \langle x_{\nu_{1}}(\lambda^{n} - A^{n}F_{n}) \rangle_{\star} + A^{n}\langle P_{n-\nu_{1}} \rangle_{\star};    
    \end{equation*}
    \end{enumerate}
\item If $\nu_{1} \leq -1$, let
    \begin{enumerate}[(c1)]
    \item if $w = \lambda^{n}$ and $n < -\nu_{1}$, then $\langle w \rangle_{\star} = w$,
    \item if $w = \lambda^{n}$, $n \geq -\nu_{1}$ then 
    \begin{equation*}
    \langle w \rangle_{\star} = \langle \lambda^{n} + A^{-n-2}P_{n} \rangle_{\star} - A^{-n-2} \langle x_{\nu_{1}}F_{\nu_{1}+n} \rangle_{\star};   
    \end{equation*}
    \item if $w = x_{\nu_{1}}\lambda^{n}$, then 
    \begin{equation*}
    \langle w \rangle_{\star} = \langle x_{\nu_{1}}(\lambda^{n} + A^{-n-3}F_{-n-1}) \rangle_{\star} - A^{-n-3}\langle P_{-n-1-\nu_{1}} \rangle_{\star}.   
    \end{equation*}
    \end{enumerate}
\end{enumerate}

Let $p(\lambda) \in R[\lambda]$, for $x_{\nu_{1}}p(\lambda)\in R\Sigma^{\prime}_{\nu_{1}}$, define
\begin{equation*}
\deg_{\lambda}(x_{\nu_{1}}p(\lambda)) = \deg(p(\lambda)).    
\end{equation*}
\begin{lemma}
\label{lem:lens_space_reduce}
For every $w \in \Sigma^{\prime}_{\nu_{1}}$, 
\begin{equation*}
\langle w \rangle_{\star}\in R\Lambda_{\nu_{1}}.    
\end{equation*}
\end{lemma}

\begin{proof}
Let $w = (x_{\nu_{1}})^{\varepsilon}\lambda^{n}$. Assume that $\nu_{1} \geq 0$, $\varepsilon = 0$, and $n > \nu_{1}$, then 
\begin{equation*}
\deg(\lambda^{n} + A^{n+2}P_{-n}) \leq n-1,    
\end{equation*}
hence using b2) in the definition of $\langle \cdot \rangle_{\star}$, we see that $\langle \lambda^{n} \rangle_{\star}$ can be expressed as an $R$-linear combination of $\langle \lambda^{j} \rangle_{\star}$, with $j = 0,1,\ldots,n-1$ and $\langle x_{\nu_{1}}\lambda^{k} \rangle_{\star}$ with $0 \leq k \leq n-1-\nu_{1}$. Since 
\begin{equation*}
\deg_{\lambda}(x_{\nu_{1}}(\lambda^{k} - A^{k}F_{k})) \leq k-1 
\end{equation*}
and when $k = 0$ this term vanishes, applying the b3) inductively allows us to express $\langle x_{\nu_{1}}\lambda^{k} \rangle_{\star}$ as an $R$-linear combination of $\langle \lambda^{j} \rangle_{\star}$ with $0 \leq j \leq |k-\nu_{1}| \leq n-1$. Therefore, $\langle \lambda^{n} \rangle_{\star}$ is an $R$-linear combination of $\langle \lambda^{j} \rangle_{\star}$, where $0\leq j \leq n-1$. Consequently, $\langle \lambda^{n} \rangle_{\star}\in R\Lambda_{\nu_{1}}$, by induction on $n$.

For $\nu_{1} \geq 0$, $\varepsilon = 1$, and $n \geq 0$, since 
\begin{equation*}
\deg_{\lambda}(x_{\nu_{1}}(\lambda^{n} - A^{n}F_{n})) \leq n-1    
\end{equation*}
and this term vanishes when $n = 0$, applying the b3) inductively allows us to express $\langle x_{\nu_{1}}\lambda^{n} \rangle_{\star}$ as $R$-linear combination of $\langle \lambda^{j} \rangle_{\star}$ with $0 \leq j \leq |n-\nu_{1}|$. Since as we showed $\langle \lambda^{j} \rangle_{\star}\in R\Lambda_{\nu_{1}}$, it follows that $\langle x_{\nu_{1}}\lambda^{n} \rangle_{\star}\in R\Lambda_{\nu_{1}}$ by induction on $n$.

Assume that $\nu_{1} \leq -1$, $\varepsilon = 0$, and $n > \kappa - 1 = -\nu_{1} - 1$. Then 
\begin{equation*}
\deg_{\lambda}(\lambda^{n} + A^{-n-2}P_{n}) \leq n-1,   
\end{equation*}
and using c2) in the definition of $\langle \cdot \rangle_{\star}$, $\langle \lambda^{n} \rangle_{\star}$ is an $R$-linear combinations of $\langle \lambda^{j} \rangle_{\star}$, where $0\leq j \leq n-1$ and $\langle x_{\nu_{1}}\lambda^{k} \rangle_{\star}$ with $0 \leq k \leq n + \nu_{1}$. Since 
\begin{equation*}
\deg_{\lambda}(x_{\nu_{1}}(\lambda^{k} + A^{-k-3}F_{-k-1})) \leq k-1
\end{equation*}
and this term vanishes when $k = 0$, applying c3) inductively allows us to express $\langle x_{\nu_{1}}\lambda^{k} \rangle_{\star}$ as an $R$-linear combination of $\langle \lambda^{j} \rangle_{\star}$ with $0 \leq j \leq |k + 1 + \nu_{1}| \leq n-1$. Consequently, $\langle \lambda^{n} \rangle_{\star} \in R\Lambda_{\nu_{1}}$ by induction on $n$.

For $\nu_{1} \leq -1$, $\varepsilon = 1$, and $n \geq 0$, since 
\begin{equation*}
\deg_{\lambda}(x_{\nu_{1}}(\lambda^{n} + A^{-n-3}F_{-n-1})) \leq n-1
\end{equation*}
and this term vanishes when $n = 0$, applying c3) inductively allows us to express $\langle x_{\nu_{1}}\lambda^{n} \rangle_{\star}$ as an $R$-linear combination of $\langle \lambda^{j} \rangle_{\star}$ with $0 \leq j \leq |n+1+\nu_{1}|$. Since $\langle \lambda^{j} \rangle_{\star}\in R\Lambda_{\nu_{1}}$ it follows that $\langle x_{\nu_{1}}\lambda^{n} \rangle_{\star}\in R\Lambda_{\nu_{1}}$ by induction on $n$.
\end{proof}

Since $\Lambda_{\nu_{1}}\subset \Sigma^{\prime}_{\nu_{1}}$, $R\Lambda_{\nu_{1}}$ is a free submodule of $R\Sigma^{\prime}_{\nu_{1}}$. For $w\in R\Gamma$ define
\begin{equation*}
\llangle w \rrangle_{\star} = \langle \llangle w \rrangle_{\Sigma^{\prime}_{\nu_{1}}} \rangle_{\star}.    
\end{equation*}

\begin{lemma}
\label{lem:lens_space}
For all $\varepsilon \in \{0,1\}$, $n_{1},n_{2}\geq 0$, and $m\in \mathbb{Z}$,
\begin{equation*}
\llangle (x_{\nu_{1}})^{\varepsilon} \lambda^{n_{1}}x_{m}\lambda^{n_{2}} - (x_{\nu_{1}})^{\varepsilon} \lambda^{n_{1}}P_{-m,n_{2}}\rrangle_{\star} = 0.
\end{equation*}
\end{lemma}

\begin{proof}
By Lemma~\ref{lem:annulus_for_any_kn_bracket_2_beta}, it suffices to show that $\llangle (x_{\nu_{1}})^{\varepsilon} \lambda^{n_{1}}x_{m}\lambda^{n_{2}} \rrangle_{\star} = \llangle (x_{\nu_{1}})^{\varepsilon} \lambda^{n_{1}}P_{-m,n_{2}} \rrangle_{\star}$ when $n_{1} = n_{2} = 0$ and $m = 0, -1$. For $\varepsilon = 0$ and $m \in \mathbb{Z}$, by \eqref{eqn:kbsm_xm_Bracket_Sigma} and the definition of $\langle \cdot \rangle_{\star}$,
\begin{equation*}
\llangle x_{m} \rrangle_{\star} = \llangle x_{\nu_{1}} F_{-m+\nu_{1}} \rrangle_{\star} = \llangle P_{-m} \rrangle_{\star}.
\end{equation*}
When $\varepsilon = 1$ and $m = 0$, by \eqref{eqn:kbsm_xvxm_Bracket_Sigma} and the definition of $\langle \cdot \rangle_{\star}$, 
\begin{eqnarray*}
\llangle x_{\nu_{1}} x_{0} \rrangle_{\star} 
&=& \llangle A^{-1} P_{-\nu_{1}-1} - A^{-2} P_{-\nu_{1}} \rrangle_{\star} = \llangle x_{\nu_{1}}(A^{-1} F_{-1} - A^{-2} F_{0}) \rrangle_{\star} \\
&=& \llangle x_{\nu_{1}}(-A^{2} - A^{-2}) \rrangle_{\star} = \llangle x_{\nu_{1}} P_{0} \rrangle_{\star}.
\end{eqnarray*}
Finally, for $\varepsilon = 1$ and $m = -1$, by \eqref{eqn:kbsm_xvxm_Bracket_Sigma} and the definition of $\langle \cdot \rangle_{\star}$, 
\begin{eqnarray*}
\llangle x_{\nu_{1}} x_{-1} \rrangle_{\star} 
&=& \llangle A^{-1} P_{-\nu_{1}-2} - A^{-2} P_{-\nu_{1}-1} \rrangle_{\star} = \llangle x_{\nu_{1}}(A^{-1} F_{-2} - A^{-2} F_{-1}) \rrangle_{\star} \\
&=& \llangle x_{\nu_{1}}(-A^{3}\lambda - A + A) \rrangle_{\star} = \llangle x_{\nu_{1}} P_{1} \rrangle_{\star}.
\end{eqnarray*}
We showed that
\begin{equation*}
\llangle (x_{\nu_{1}})^{\varepsilon} x_{m} \rrangle_{\star} = \llangle (x_{\nu_{1}})^{\varepsilon} P_{-m} \rrangle_{\star},    
\end{equation*}
for $\varepsilon\in \{0,1\}$ and $m\in \{0,-1\}$, which completes our proof.
\end{proof}

\begin{theorem}
\label{thm:lens_space}
KBSM of $M_{2}(\beta_{1}) = L(\beta_{1},2)$ is a free $R$-module with basis consisting of equivalence classes of generic framed links in $M_{2}(\beta_{1})$ with their arrow diagrams in $\Lambda_{\nu_{1}}$, i.e.,
\begin{equation*}
S_{2,\infty}(L(\beta_{1},2);R,A)\cong R\Lambda_{\nu_{1}}.    
\end{equation*}
\end{theorem}

\begin{proof}
For an arrow diagram $D$ on $\hat{\bf S}^{2}$, define
\begin{equation*}
\psi_{\nu_{1}}(D) = \langle \phi_{\beta_{1}}(D) \rangle_{\star}.
\end{equation*}
If arrow diagrams $D,D'$ on $\hat{\bf S}^{2}$ are related by $\Omega_{1}-\Omega_{5}$ and $S_{\beta_{1}}$-moves then, as we noted in Section~\ref{s:Summary},
\begin{equation*}
\psi_{\nu_{1}}(D-D') = \langle \phi_{\beta_{1}}(D - D') \rangle_{\star} = 0.    
\end{equation*}
Assume that arrow diagrams $D,D'$ on $\hat{\bf S}^{2}$ are related by $\Omega_{\infty}$-move. Let $\mathcal{K}(D)$ and $\mathcal{K}(D')$ be sets of all Kauffman states of $D$ and $D'$ respectively. Since $D$ and $D'$ have the same crossings inside ${\bf D}^{2}_{\beta_{1}} = \hat{\bf S}^{2}\smallsetminus {\bf D}^{2}_{\infty}$, there is a natural bijection between $\mathcal{K}(D)$ and $\mathcal{K}(D')$ which assigns to $s\in \mathcal{K}(D)$ the state $s'\in \mathcal{K}(D')$ with exactly the same markers for each crossing of $D'$. 
\begin{figure}[H]
\centering
\includegraphics[scale=0.7]{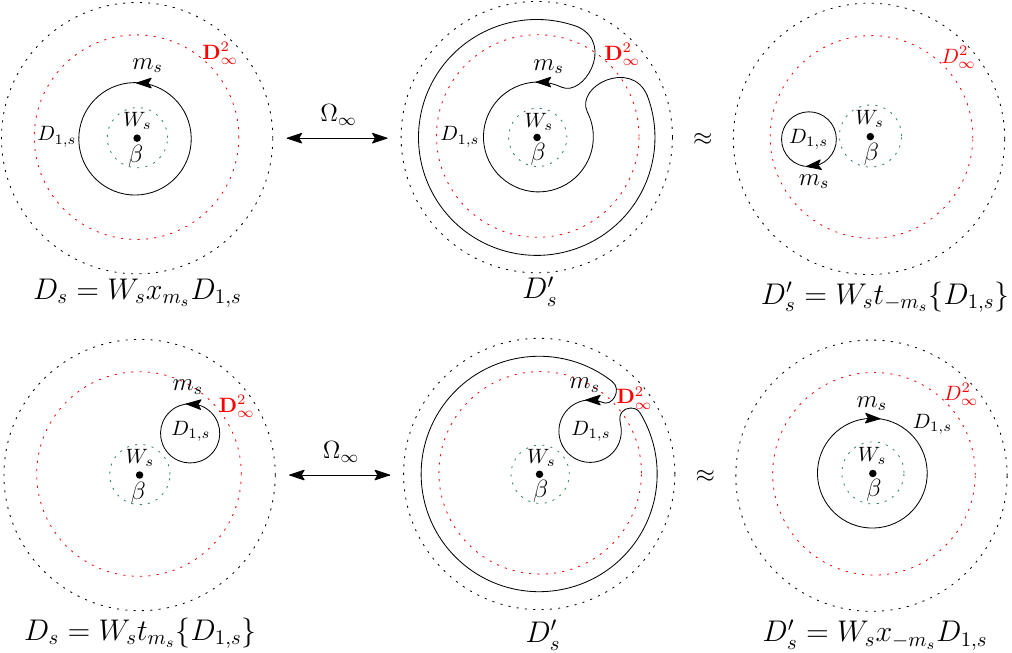}
\caption{Arrow diagrams $D_{1}'$ and $D_{2}'$ in $\hat{\bf S}^{2}$ related by $\Omega_{\infty}$-move}
\label{fig:Omega_inf_On_X_T_proof}
\end{figure}
Furthermore, arrow diagrams $D_{s}$ and $D'_{s}$ corresponding to $s\in \mathcal{K}(D) = \mathcal{K}_{a}(D) \cup \mathcal{K}_{b}(D)$ have one of two forms shown in Figure~\ref{fig:Omega_inf_On_X_T_proof}:
\begin{enumerate}[a)]
\item if $s \in \mathcal{K}_{a}(D)$ then $D_{s} = W_{s}x_{m_{s}}D_{1,s}$ and $D'_{s} = W_{s}t_{-m_{s}}\{D_{1,s}\}$ , or
\item if $s \in \mathcal{K}_{b}(D)$ then $D_{s} = W_{s}t_{m_{s}}\{D_{1,s}\}$ and $D'_{s} = W_{s}x_{-m_{s}}D_{1,s}$.
\end{enumerate}
Consequently,
\begin{equation*}
\llangle D - D' \rrangle =\sum_{s\in \mathcal{K}_{a}(D)} A^{p(s)-n(s)}\langle D_{s} - D'_{s}\rangle + \sum_{s\in \mathcal{K}_{b}(D)} A^{p(s)-n(s)} \langle D_{s} - D'_{s}\rangle.
\end{equation*}
Since
\begin{eqnarray*}
\llangle \langle D_{s} - D'_{s}\rangle \rrangle_{\Gamma}  &=& \llangle \llangle W_{s} \rrangle \rrangle_{\Gamma}(x_{m_{s}}\langle D_{1,s}\rangle_{r}- \langle t_{-m_{s}}\{\langle D_{1,s}\rangle_{r}\}\rangle_{r})\,\, \text{for}\, s\in \mathcal{K}_{a}(D),\,\, \text{and}\\
\llangle \langle D_{s} - D'_{s}\rangle \rrangle_{\Gamma}  &=& \llangle \llangle W_{s} \rrangle \rrangle_{\Gamma}(\langle t_{m_{s}}\{\langle D_{1,s}\rangle_{r}\}\rangle_{r} - x_{-m_{s}}\langle D_{1,s}\rangle_{r})\,\, \text{for}\, s\in \mathcal{K}_{b}(D),
\end{eqnarray*}
where
\begin{equation*}
\langle D_{1,s} \rangle_{r} = \sum_{i=0}^{n_{s}}r_{s,i}^{(1)}\lambda^{i}, \ \langle t_{-m_{s}}\{\langle D_{1,s}\rangle_{r}\}\rangle_{r} = \sum_{i=0}^{n_{s}}r_{s,i}^{(1)}P_{-m_{s},i} \  \text{and} \ \llangle \llangle W_{s} \rrangle \rrangle_{\Gamma} = \sum_{j = 0}^{k_{s}}r_{s,j}^{(2)}w_{j}(s).   
\end{equation*}
Therefore,
\begin{eqnarray*}
\llangle \llangle \langle D_{s} - D'_{s}\rangle \rrangle_{\Gamma} \rrangle_{\Sigma^{\prime}_{\nu_{1}}}  &=& \sum_{j=0}^{k_{s}}\sum_{i=0}^{n_{s}}r_{s,i}^{(1)}r_{s,j}^{(2)}\llangle w_{j}(s)(x_{m_{s}}\lambda^{i} - P_{-m_{s},i})\rrangle_{\Sigma^{\prime}_{\nu_{1}}}\,\, \text{for}\,\, s\in \mathcal{K}_{a}(D),\,\text{and}\\
\llangle \llangle \langle D_{s} - D'_{s}\rangle \rrangle_{\Gamma} \rrangle_{\Sigma^{\prime}_{\nu_{1}}}  &=& \sum_{j=0}^{k_{s}}\sum_{i=0}^{n_{s}}r_{s,i}^{(1)}r_{s,j}^{(2)} \llangle w_{j}(s)(P_{m_{s},i} - x_{-m_{s}}\lambda^{i})\rrangle_{\Sigma^{\prime}_{\nu_{1}}}\,\, \text{for}\,\, s\in \mathcal{K}_{b}(D),
\end{eqnarray*}
and furthermore, for $s\in \mathcal{K}_{a}(D)$ 
\begin{eqnarray*}
\llangle w_{j}(s)(x_{m_{s}}\lambda^{i} - P_{-m_{s},i})\rrangle_{\Sigma^{\prime}_{\nu_{1}}}& =& \llangle \llangle w_{j}(s)\rrangle_{\Sigma^{\prime}_{\nu_{1}}} (x_{m_{s}}\lambda^{i} - P_{-m_{s},i})\rrangle_{\Sigma^{\prime}_{\nu_{1}}}\\
&=& \sum_{\varepsilon \in \{0,1\}}\sum_{k=0}^{l_{s,j}} r_{s,j,\varepsilon,k}^{(3)}\llangle (x_{\nu_{1}})^{\varepsilon}\lambda^{k}(x_{m_{s}}\lambda^{i} - P_{-m_{s},i})\rrangle_{\Sigma^{\prime}_{\nu_{1}}},
\end{eqnarray*}
and for $s\in \mathcal{K}_{b}(D)$
\begin{eqnarray*}
\llangle w_{j}(s)(P_{m_{s},i} - x_{-m_{s}}\lambda^{i})\rrangle_{\Sigma^{\prime}_{\nu_{1}}}& =& \llangle \llangle w_{j}(s)\rrangle_{\Sigma^{\prime}_{\nu_{1}}} (P_{m_{s},i} - x_{-m_{s}}\lambda^{i})\rrangle_{\Sigma^{\prime}_{\nu_{1}}} \\
&=& \sum_{\varepsilon \in \{0,1\}}\sum_{k=0}^{l_{s,j}} r_{s,j,\varepsilon,k}^{(3)}\llangle (x_{\nu_{1}})^{\varepsilon}\lambda^{k}(P_{m_{s},i} - x_{-m_{s}}\lambda^{i})\rrangle_{\Sigma^{\prime}_{\nu_{1}}},
\end{eqnarray*}
where
\begin{equation*}
\llangle w_{j}(s)\rrangle_{\Sigma^{\prime}_{\nu_{1}}} = \sum_{\varepsilon \in \{0,1\}}\sum_{k=0}^{l_{s,j}} r_{s,j,\varepsilon,k}^{(3)}(x_{\nu_{1}})^{\varepsilon}\lambda^{k}.
\end{equation*}
Consequently, for arrow diagrams $D,D'$ on $\hat{\bf S}^{2}$ which differ by $\Omega_{\infty}$-move $\psi_{\nu_{1}}(D-D') = 0$ if and only if for all $\varepsilon \in \{0,1\}$, $k\geq 0$ and $m\in \mathbb{Z}$,
\begin{equation*}
\llangle (x_{\nu_{1}})^{\varepsilon}\lambda^{k}x_{m}\lambda^{i} - (x_{\nu_{1}})^{\varepsilon}\lambda^{k}P_{-m,i}\rrangle_{\star} = 0,
\end{equation*}
which we proved in Lemma~\ref{lem:lens_space}. It follows that 
$\psi_{\nu_{1}}$ is well-defined map on equivalence classes of arrow diagrams in $\hat{\bf S}^{2}$, modulo $\Omega_{1}-\Omega_{5}$, $S_{\beta_{1}}$, and $\Omega_{\infty}$-moves, which also extends to a surjective\footnote{Surjectivity of $\psi_{\nu_{1}}$ is clear since $\Lambda_{\nu_{1}}\subset \mathcal{D}(\hat{\bf S}^{2})$.} homomorphism of free $R$-modules
$\psi_{\nu_{1}}: R\mathcal{D}(\hat{\bf S}^{2}) \to R\Lambda_{\nu_{1}}$.
Let
\begin{equation*}
\varphi: R\Lambda_{\nu_{1}} \to  S\mathcal{D}_{\nu_{1}},\, \varphi(\lambda^{j}) = [\lambda^{j}],\,\,0\leq j \leq \kappa -1. 
\end{equation*}
Let $D$ be an arrow diagram in $\hat{\bf S}^{2}$ and $w = \psi_{\nu_{1}}(D)$. Then $\varphi(w) = [w] = [D]$ and consequently $\varphi$ is surjective. 

Furthermore, as it is easy to see, for a skein triple $D_{+}$, $D_{0}$, $D_{\infty}$ of arrow diagrams in $\hat{\bf S}^{2}$, and an arrow diagram $D$ in $\hat{\bf S}^{2}$,
\begin{equation*}
\psi_{\nu_{1}}(D_{+} - AD_{0} - A^{-1}D_{\infty}) = 0 \quad \text{and} \quad \psi_{\nu_{1}}(D\sqcup T_{1} +(A^{-2}+A^{2})D) = 0.
\end{equation*}
Therefore, $\psi_{\nu_{1}}$ descends to a surjective homomorphism of $R$-modules
\begin{equation*}
\hat{\psi}_{\nu_{1}}: S\mathcal{D}_{\nu_{1}} \to R\Lambda_{\nu_{1}},
\end{equation*}
which to a generator $D$ assigns $\psi_{\nu_{1}}(D)$. To show that $\varphi$ is also injective, we simply check that $\hat{\psi}_{\nu_{1}} \circ \varphi = Id$. It follows that $\varphi$ and $\hat{\psi}_{\nu_{1}}$ are isomorphisms of $R$-modules.

By Theorem~\ref{thm:AmbientIsotopiesInMBeta1Beta2}(i), there is a bijection between ambient isotopy classes of framed links in $M_{2}(\beta_{1})$ and equivalence classes of arrow diagrams in $\hat{\bf S}^{2}$ modulo $\Omega_{1}-\Omega_{5}$, $S_{\beta_{1}}$, and $\Omega_{\infty}$-moves. Therefore,
\begin{equation*}
S_{2,\infty}(M_{2}(\beta_{1});R,A) \cong S\mathcal{D}_{\nu_{1}}\underset{\hat{\psi}_{\nu_{1}}}{\cong} R\Lambda_{\nu_{1}},
\end{equation*}
which completes our proof.
\end{proof}

\section{Lens spaces \texorpdfstring{$L(4k,2k+1)$}
{L(4k, 2k+1)}}
\label{s:CaseWithTwoFibers}

As we noted in Section~\ref{s:ambient_isotopy_In_Lens_Spaces}, generic framed links in $M_{2}(\beta_{1},\beta_{2})$ can be represented by arrow diagrams in $\hat{\bf S}^{2}$ and, by Theorem~\ref{thm:AmbientIsotopiesInMBeta1Beta2}, such links are ambient isotopic if and only if their arrow diagrams are related by $\Omega_{1}-\Omega_{5}$, $S_{\beta_{1}}$, and $S_{\beta_{2}}$-moves on $\hat{\bf S}^{2}$ (see Figure~\ref{fig:SBeta1Beta2MovesOnhatS2}).

\begin{lemma}
\label{lem:rel_F_and_FR_TwoFibers}
In $S\mathcal{D}_{\nu_{1},\nu_{2}}$, for all $m \in \mathbb{Z}$,
\begin{equation*}
-A^{-3}F_{m}x_{-\nu_{2}-1} = F_{m}x_{-\nu_{2}} = x_{\nu_{1}}F_{\nu_{0}-m}
\end{equation*}
and
\begin{equation*}
-A^{-3}x_{\nu_{1}}F_{m}x_{-\nu_{2}-1} = x_{\nu_{1}}F_{m}x_{-\nu_{2}} = R_{m-\nu_{0}}.
\end{equation*}
\end{lemma}

\begin{figure}[ht]
\centering
\includegraphics[scale=0.8]{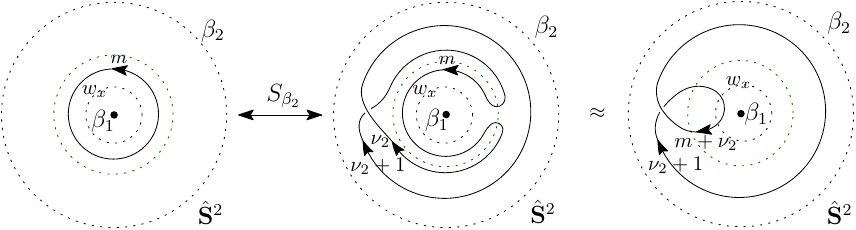}
\caption{$S_{\beta_{2}}$-move on arrow diagram $w_{x}x_{m}$}
\label{fig:SBeta2MovesOnXCurveTwoFibers}
\end{figure}

\begin{proof}
Arrow diagrams on the left and the right of Figure~\ref{fig:SBeta2MovesOnXCurveTwoFibers} differ by an $S_{\beta_{2}}$-move on $\hat{\bf S}^{2}$ hence in $S\mathcal{D}_{\nu_{1},\nu_{2}}$, where $w_{x} \in \Gamma$,
\begin{equation*}
w_{x}x_{m} = Aw_{x}x_{m-1} + A^{-1}w_{x}P_{-\nu_{2}-m}x_{-\nu_{2}-1}.  
\end{equation*}
Consequently, for $m = -\nu_{2}$, 
\begin{equation*}
w_{x}x_{-\nu_{2}} = Aw_{x}x_{-\nu_{2}-1} + A^{-1}w_{x}P_{0}x_{-\nu_{2}-1} = -A^{-3}w_{x}x_{-\nu_{2}-1}.  
\end{equation*}
Therefore,
\begin{equation}
\label{eqn:RelSBetaMoveTwoFiberCase}
w_{x}x_{-\nu_{2}-1} = - A^{3}w_{x}x_{-\nu_{2}}.   
\end{equation}
Furthermore, using \eqref{eqn:rel_xm_to_Qxk} and \eqref{eqn:RelSBetaMoveTwoFiberCase} with $k = \nu_{2} + m$, we see that
\begin{eqnarray}
w_{x}x_{m} &=& A^{-\nu_{2}-m} w_{x}Q_{\nu_{2}+m+1}x_{-\nu_{2}} - A^{-\nu_{2}-m-1} w_{x}Q_{\nu_{2}+m}x_{-\nu_{2}-1} \notag\\
&=& A^{-\nu_{2}-m} w_{x}Q_{\nu_{2}+m+1}x_{-\nu_{2}} + A^{-\nu_{2}-m+2} w_{x}Q_{\nu_{2}+m}x_{-\nu_{2}} = w_{x}F_{\nu_{2}+m}x_{-\nu_{2}}. \label{eqn:RelSBetaMoveTwoFiberCase2}
\end{eqnarray}
Since $x_{m-\nu_{2}} = x_{\nu_{1}}F_{\nu_{0}-m}$ by \eqref{eqn:kbsm_xm}, using the above identities \eqref{eqn:RelSBetaMoveTwoFiberCase} and \eqref{eqn:RelSBetaMoveTwoFiberCase2}, it follows that 
\begin{equation*}
-A^{-3}F_{m}x_{-\nu_{2}-1} = F_{m}x_{-\nu_{2}} = x_{m-\nu_{2}} = x_{\nu_{1}}F_{\nu_{0}-m}.
\end{equation*}
Finally, applying \eqref{eqn:RelSBetaMoveTwoFiberCase}, \eqref{eqn:RelSBetaMoveTwoFiberCase2}, and \eqref{eqn:kbsm_xvxm}, we also see that
\begin{equation*}
-A^{-3}x_{\nu_{1}}F_{m}x_{-\nu_{2}-1} = x_{\nu_{1}}F_{m}x_{-\nu_{2}} = x_{\nu_{1}}x_{m-\nu_{2}} = R_{m-\nu_{0}} 
\end{equation*}
which completes our proof.
\end{proof}

\begin{lemma}
\label{lem:IdentityTwoFiberCase}
In $S\mathcal{D}({\bf D}^{2}_{\beta_{1}})$, for all $m,n \in \mathbb{Z}$ and $k\geq 0$,
\begin{eqnarray}
x_{m}x_{n} &=& A^{-2k}x_{m+k}x_{n-k} + \sum_{i=0}^{k-1}A^{-2i}(P_{n-m-2-2i}-A^{-2}P_{n-m-2i}), \label{eqn:IdentityTwoFiberCase1} \\
x_{m}x_{n} &=& A^{2k}x_{m-k}x_{n+k} + \sum_{i=0}^{k-1}A^{2i}(P_{n-m+2+2i}-A^{2}P_{n-m+2i}). \label{eqn:IdentityTwoFiberCase2}
\end{eqnarray}
\end{lemma}

\begin{proof}
Arrow diagrams on the left and the right of Figure~\ref{fig:PassingArrowsForXcurves} are related by an $\Omega_{5}$-move on ${\bf D}^{2}_{\beta_{1}}$, so after applying Kauffman bracket skein relation to these diagrams gives in $S\mathcal{D}_{\nu_{1},\nu_{2}}$ ,
\begin{equation*}
AP_{n-m-1} + A^{-1}x_{m+1}x_{n} = Ax_{m}x_{n+1} + A^{-1}P_{n+1-m}
\end{equation*}
and hence
\begin{eqnarray*}
x_{m}x_{n+1} &=& A^{-2}x_{m+1}x_{n} + P_{n-m-1} - A^{-2}P_{n+1-m}\,\,\text{and} \\
x_{m+1}x_{n} &=& A^{2} x_{m}x_{n+1} + P_{n+1-m} - A^{2} P_{n-m-1}.
\end{eqnarray*}
Therefore, identities in the statement of our lemma follow by induction on $k\geq 0$.
\end{proof}

\begin{figure}[H]
\centering
\includegraphics[scale=0.8]{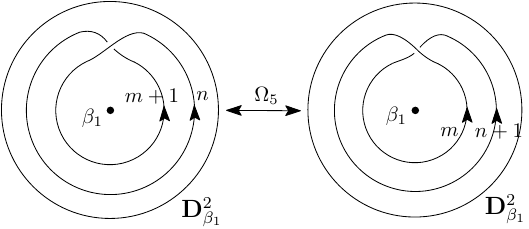}
\caption{Arrow diagrams in ${\bf D}^{2}_{\beta_{1}}$ related by $\Omega_{5}$-move}
\label{fig:PassingArrowsForXcurves}
\end{figure}

We show that, if $\nu_{0} \neq -1$, then KBSM of $M_{2}(\beta_{1},\beta_{2})$ is isomorphic to a free $R$-module $S\mathcal{D}_{\nu_{1},\nu_{2}}$ of rank $2|\nu_{0}+1|+1$, and for $\nu_{0} = -1$, KBSM of $M_{2}(\beta_{1},\beta_{2}) = L(0,1) = {\bf S}^{2}\times S^{1}$ is infinitely generated and it decomposes into a direct sum of cyclic modules. Since case $\nu_{0} \neq -1$ and $\nu_{0} = -1$ require a different approach, we address each in a separate subsection.

\subsection{KBSM of \texorpdfstring{$M_{2}(\beta_{1},\beta_{2})$ with $\nu_{0}\neq -1$}
{M\unichar{8322}(\unichar{0946}\unichar{8321}, \unichar{0946}\unichar{8322}) with \unichar{0957}\unichar{8320} \unichar{8800} -1}}
\label{s:CaseWithTwoFibers_Nu0DifferntThenMinus_1}

In this section we give a new proof of Theorem~4 of \cite{HP1993} for the family of lens spaces $L(4k,2k+1)$, where $k \neq 0$. Theorem~$4$ of \cite{HP1993} gives the rank (i.e., $\lfloor p/2\rfloor + 1$) and a basis for KBSM of $L(p,q)$ over $R$, where $p \geq 1$, $q\in\mathbb{Z}$, and $\gcd(p,q) = 1$. In this paper, using our model $M_{2}(\beta_{1},\beta_{2})$ for $L(4k,2k+1)$, we construct a new basis for its KBSM and develop computational tools which allow us to express any framed link in terms of this basis.\\
\indent Let $\Sigma''_{\nu_{1},\nu_{2}}$ be the subset of $\Sigma^{\prime}_{\nu_{1}}$ defined by
\begin{equation*}
\Sigma''_{\nu_{1},\nu_{2}} = 
\begin{cases}
\{\lambda^{n},x_{\nu_{1}}\lambda^{k} \mid 0 \leq n \leq \nu_{0} + 1,\, 0 \leq k \leq \nu_{0}\}, & \text{if}\ \nu_{0}\geq 0, \\
\{\lambda^{n},x_{\nu_{1}}\lambda^{k} \mid 0 \leq n \leq -\nu_{0} - 1,\, 0 \leq k \leq -\nu_{0} - 2\},    & \text{if}\ \nu_{0} \leq -2.
\end{cases}
\end{equation*}
In this section, we show that
\begin{equation*}
S\mathcal{D}_{\nu_{1},\nu_{2}} \cong R\Sigma''_{\nu_{1},\nu_{2}}.
\end{equation*}

Using Lemma~\ref{lem:rel_F_and_FR_TwoFibers}, we define bracket $\langle w \rangle_{\star\star}$ for $w \in R\Sigma^{\prime}_{\nu_{1}}$ as follows:
\begin{enumerate}[(a)]
\item For $w = \sum_{w' \in S} r_{w'}w'$, $S$ is a finite subset of $\Sigma^{\prime}_{\nu_{1}}$ with at least two elements and $r_{w'}\in R$, let 
\begin{equation*}
\langle w \rangle_{\star\star} = \sum_{w'\in S} r_{w'}\langle w' \rangle_{\star\star},
\end{equation*}

\item If $\nu_{0} \geq 0$, let
\begin{enumerate}[(b1)]
\item if $w \in \Sigma''_{\nu_{1},\nu_{2}}$, then $\langle w \rangle_{\star\star} = w$;
\item if $w = \lambda^{n}$ with $n \geq \nu_{0}+2$, then 
\begin{equation*}
    \langle w \rangle_{\star\star} = \langle \lambda^{n} + A^{n+3} R_{-n+1} \rangle_{\star\star} - A^{n+3} \langle \llangle x_{\nu_{1}} F_{-n+\nu_{0}+1}x_{-\nu_{2}} \rrangle_{\Sigma^{\prime}_{\nu_{1}}} \rangle_{\star\star};   
\end{equation*}
    \item if $w = x_{\nu_{1}}\lambda^{n}$ with $n \geq \nu_{0}+1$, then 
    \begin{equation*}
    \langle w \rangle_{\star\star} = \langle x_{\nu_{1}}(\lambda^{n} - A^{n}F_{n}) \rangle_{\star\star} + A^{n}\langle \llangle F_{\nu_{0}-n} x_{-\nu_{2}} \rrangle_{\Sigma^{\prime}_{\nu_{1}}} \rangle_{\star\star}.   
    \end{equation*}
\end{enumerate}
\item If $\nu_{0} \leq -2$, let
    \begin{enumerate}[(c1)]
    \item if $w \in \Sigma''_{\nu_{1},\nu_{2}}$, then $\langle w \rangle_{\star\star} = w$;
    \item if $w = \lambda^{n}$ with $n \geq -\nu_{0}$, then 
    \begin{equation*}
    \langle w \rangle_{\star\star} = \langle \lambda^{n} - A^{-n} R_{n} \rangle_{\star\star} - A^{-n-3} \langle \llangle x_{\nu_{1}}F_{n+\nu_{0}}x_{-\nu_{2}-1} \rrangle_{\Sigma^{\prime}_{\nu_{1}}} \rangle_{\star\star};
    \end{equation*}
    \item if $w = x_{\nu_{1}}\lambda^{n}$ with $n \geq -\nu_{0}-1$, then
    \begin{equation*}
    \langle w \rangle_{\star\star} = \langle x_{\nu_{1}}(\lambda^{n} + A^{-n-3} F_{-n-1}) \rangle_{\star\star} + A^{-n-6}\langle \llangle F_{n+\nu_{0}+1} x_{-\nu_{2}-1}  \rrangle_{\Sigma^{\prime}_{\nu_{1}}} \rangle_{\star\star}.   
    \end{equation*}
    \end{enumerate}
\end{enumerate}

\begin{lemma}
\label{lem:lens_space_two_fiber_reduce}
For every $w \in \Sigma^{\prime}_{\nu_{1}}$, 
\begin{equation*}
\langle w \rangle_{\star\star}\in R\Sigma''_{\nu_{1},\nu_{2}}.    
\end{equation*}
\end{lemma}

\begin{proof}
Assume that $\nu_{0} \geq 0$ and $w = \lambda^{n}$ with $n \geq \nu_{0}+2$.
Clearly,
\begin{equation*}
\deg(\lambda^{n} + A^{n+3} R_{-n+1}) = n-1    
\end{equation*}
and, by \eqref{eqn:kbsm_xm_Bracket_Sigma}, \eqref{eqn:IdentityTwoFiberCase2}, and \eqref{eqn:kbsm_xvxm_Bracket_Sigma}
\begin{eqnarray*}
\llangle x_{\nu_{1}} F_{-n+\nu_{0}+1}x_{-\nu_{2}} \rrangle_{\Sigma^{\prime}_{\nu_{1}}} 
&=& \llangle x_{(n-\nu_{0}-1) + \nu_{1}}x_{-\nu_{2}} \rrangle_{\Sigma^{\prime}_{\nu_{1}}} \\
&=& A^{2(n-\nu_{0}-1)} \llangle x_{\nu_{1}}x_{-\nu_{2}+(n-\nu_{0}-1)} \rrangle_{\Sigma^{\prime}_{\nu_{1}}} + \sum_{i=0}^{n-\nu_{0}-2} A^{2i} (P_{2i-n+3}- A^{2} P_{2i-n+1}) \\
&=& A^{2(n-\nu_{0}-1)} R_{n-2\nu_{0}-1} + \sum_{i=0}^{n-\nu_{0}-2} A^{2i} (P_{2i-n+3} - A^{2} P_{2i-n+1}).  
\end{eqnarray*}
Moreover, as one may check,
\begin{equation*}
\deg R_{n-2\nu_{0}-1} = \max\{n-2\nu_{0}-1, 2+2\nu_{0}-n\} \leq n-1,\,\deg P_{-n+1} = n-1,\,\deg P_{n-2\nu_{0}-1} = |n-2\nu_{0}-1| \leq n-1.   
\end{equation*}
Therefore, b2) in the definition of $\langle \cdot \rangle_{\star\star}$ allows us to express $\langle \lambda^{n} \rangle_{\star\star}$ as an $R$-linear combination of $\langle \lambda^{k} \rangle_{\star\star}$ with $0\leq k \leq n-1$. It follows that $\langle \lambda^{n} \rangle_{\star\star}\in R\Sigma''_{\nu_{1},\nu_{2}}$ by induction on $n$.

Assume that $\nu_{0} \geq 0$ and $w = x_{\nu_{1}}\lambda^{n}$ with $n \geq \nu_{0}+1$. Clearly,
\begin{equation*}
\deg_{\lambda} (x_{\nu_{1}}(\lambda^{n} - A^{n}F_{n})) = n-1,
\end{equation*}
and applying both, \eqref{eqn:rel_xm_to_Qxk} inductively and then \eqref{eqn:kbsm_xm_Bracket_Sigma}, we see that
\begin{eqnarray*}
\llangle \lambda^{n-\nu_{0}-1}x_{-\nu_{2}} \rrangle_{\Sigma^{\prime}_{\nu_{1}}} &=& \sum_{i=0}^{n-\nu_{0}-1} A^{n-\nu_{0}-1-2i} \binom{n-\nu_{0}-1}{i} \llangle x_{-\nu_{2}+n-\nu_{0}-1-2i} \rrangle_{\Sigma^{\prime}_{\nu_{1}}} \\
&=& \sum_{i=0}^{n-\nu_{0}-1} A^{n-\nu_{0}-1-2i} \binom{n-\nu_{0}-1}{i} x_{\nu_{1}} F_{2\nu_{0}+1-n+2i}.
\end{eqnarray*}
Moreover, 
\begin{equation*}
\deg(F_{2\nu_{0}+1-n}) = \max\{2\nu_{0}+1-n, n-2\nu_{0}-2\} \leq n-1\,\text{and}\, \deg (F_{n-1}) = n-1.  
\end{equation*}
Since $\llangle F_{\nu_{0}-n} x_{-\nu_{2}} \rrangle_{\Sigma^{\prime}_{\nu_{1}}}$ is an $R$-linear combination of $\llangle \lambda^{k} x_{-\nu_{2}} \rrangle_{\Sigma^{\prime}_{\nu_{1}}}$ with $0\leq k \leq n-\nu_{0}-1$, it follows that $\llangle F_{\nu_{0}-n} x_{-\nu_{2}} \rrangle_{\Sigma^{\prime}_{\nu_{1}}}$ is a linear combination of $x_{\nu_{1}}\lambda^{k}$ with $0\leq k\leq n-1$. Therefore, applying b3) in the definition of $\langle \cdot \rangle_{\star\star}$ allows us to represent $\langle x_{\nu_{1}}\lambda^{n} \rangle_{\star\star}$ as an $R$-linear combination of $\langle x_{\nu_{1}}\lambda^{k} \rangle_{\star\star}$ with $0\leq k \leq n-1$. It follows by induction on $n$ that $\langle x_{\nu_{1}}\lambda^{n} \rangle_{\star\star}\in R\Sigma''_{\nu_{1},\nu_{2}}$.

Assume that $\nu_{0} \leq -2$ and let $w = \lambda^{n}$ with $n \geq -\nu_{0}$. Using \eqref{eqn:kbsm_xm_Bracket_Sigma}, \eqref{eqn:IdentityTwoFiberCase1}, and \eqref{eqn:kbsm_xvxm_Bracket_Sigma}, we see that
\begin{eqnarray*}
\llangle x_{\nu_{1}}F_{n+\nu_{0}}x_{-\nu_{2}-1} \rrangle_{\Sigma^{\prime}_{\nu_{1}}}
&=& \llangle x_{\nu_{1}-n-\nu_{0}}x_{-\nu_{2}-1} \rrangle_{\Sigma^{\prime}_{\nu_{1}}} \\
&=& A^{-2(n+\nu_{0})} \llangle x_{\nu_{1}}x_{-\nu_{2}-1-n-\nu_{0}} \rrangle_{\Sigma^{\prime}_{\nu_{1}}} + \sum_{i=0}^{n+\nu_{0}-1} A^{-2i}(P_{n-3-2i} - A^{-2}P_{n-1-2i}) \\
&=& A^{-2(n+\nu_{0})} R_{-n-2\nu_{0}-1} + \sum_{i=0}^{n+\nu_{0}-1} A^{-2i}(P_{n-3-2i} - A^{-2}P_{n-1-2i}). 
\end{eqnarray*}
Furthermore, since 
\begin{equation*}
\deg(R_{-n-2\nu_{0}-1}) = \max\{-n-2\nu_{0}-1,n+2\nu_{0}+2\} \leq n-1,\, \deg(P_{n-1}) = n-1,\,\text{and}\, \deg (P_{-n-2\nu_{0}-1}) \leq n-1,
\end{equation*}
it follows from relation c2) in the definition of $\langle \cdot \rangle_{\star\star}$ that $\langle \lambda^{n} \rangle_{\star\star}$ can be written as an $R$-linear combination of $\langle \lambda^{k} \rangle_{\star\star}$ with $0\leq k \leq n-1$. Thus, $\langle \lambda^{n} \rangle_{\star\star}\in R\Sigma''_{\nu_{1},\nu_{2}}$.

Assume that $\nu_{0} < -1$ and $w = x_{\nu_{1}}\lambda^{n}$, where $n \geq -\nu_{0}-1$. Clearly, 
\begin{equation*}
\deg_{\lambda}(x_{\nu_{1}}(\lambda^{n} + A^{-n-3} F_{-n-1})) = n-1,    
\end{equation*}
and using both, \eqref{eqn:rel_xm_to_Qxk} inductively and then \eqref{eqn:kbsm_xm}, we see that
\begin{eqnarray*}
\llangle \lambda^{n+\nu_{0}+1} x_{-\nu_{2}-1} \rrangle_{\Sigma^{\prime}_{\nu_{1}}} 
&=& \sum_{i=0}^{n+\nu_{0}+1} A^{n+\nu_{0}+1-2i} \binom{n+\nu_{0}+1}{i} \llangle x_{n+\nu_{1}-2i} \rrangle_{\Sigma^{\prime}_{\nu_{1}}} \\
&=& \sum_{i=0}^{n+\nu_{0}+1} A^{n+\nu_{0}+1-2i} \binom{n+\nu_{0}+1}{i} x_{\nu_{1}} F_{2i-n}. 
\end{eqnarray*}
Furthermore, 
\begin{equation*}
\deg (F_{n+2\nu_{0}+2}) = \max\{n+2\nu_{0}+2, -n-2\nu_{0}-3\} \leq n-1\,\text{and}\,\deg (F_{-n}) = n-1.
\end{equation*}
Since $\llangle F_{n+\nu_{0}+1} x_{-\nu_{2}-1} \rrangle_{\Sigma^{\prime}_{\nu_{1}}}$ is a linear combination of $\llangle \lambda^{k} x_{-\nu_{2}-1} \rrangle_{\Sigma^{\prime}_{\nu_{1}}}$ with $0\leq k \leq n+\nu_{0}+1$, it follows that $\llangle F_{n+\nu_{0}+1} x_{-\nu_{2}-1} \rrangle_{\Sigma^{\prime}_{\nu_{1}}}$ is an $R$-linear combination of $x_{\nu_{1}}\lambda^{k}$ with $0\leq k \leq n-1$. Therefore, c3) given in the definition of $\langle \cdot \rangle_{\star\star}$ allows us to write $\langle x_{\nu_{1}}\lambda^{n} \rangle_{\star\star}$ as an $R$-linear combination of $\langle x_{\nu_{1}}\lambda^{k} \rangle_{\star\star}$ with $0\leq k\leq n-1$. Consequently, $\langle x_{\nu_{1}}\lambda^{n} \rangle_{\star\star}\in R\Sigma''_{\nu_{1},\nu_{2}}$ by induction on $n$.
\end{proof}

Since $\Sigma''_{\nu_{1},\nu_{2}}\subset \Sigma^{\prime}_{\nu_{1}}$, $R\Sigma''_{\nu_{1},\nu_{2}}$ is a free $R$-submodule of $R\Sigma^{\prime}_{\nu_{1}}$. For $w\in R\Gamma$ define
\begin{equation*}
\llangle w \rrangle_{\star\star} = \langle \llangle w \rrangle_{\Sigma^{\prime}_{\nu_{1}}} \rangle_{\star\star}.    
\end{equation*}

\begin{remark}
\label{rem:NoteOnGeneratingSetForKBSM}
Using induction on $n\geq 0$ and \eqref{eqn:rel_Pn}, we can show that $\lambda^{n}$ is an $R$-linear combination of polynomials $P_{k}$ with $0\leq k\leq n$. This observation will be used in proofs of Lemma~\ref{lem:Relations_Bracket_Two_Fibers_Case_1a} and Lemma~\ref{lem:Relations_Bracket_Two_Fibers_Case_2a}.
\end{remark}

\begin{lemma}
\label{lem:Relations_Bracket_Two_Fibers_Case_1a}
Let $\nu_{0} \geq 0$, then for any $\varepsilon \in \{0,1\}$ and $n \geq 0$,
\begin{equation}
\label{eqn:lem_Relations_Bracket_Two_Fibers_Case_1a}
\llangle (x_{\nu_{1}})^{\varepsilon}\lambda^{n}x_{-\nu_{2}-1} \rrangle_{\star\star} = -A^{3} \llangle (x_{\nu_{1}})^{\varepsilon}\lambda^{n}x_{-\nu_{2}} \rrangle_{\star\star}.
\end{equation}
\end{lemma}

\begin{figure}[H]
\centering
\includegraphics[scale=0.8]{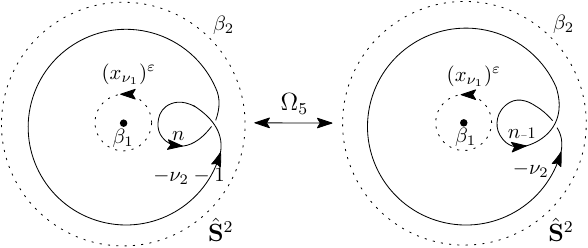}
\caption{Arrow diagrams in $\hat{\bf S}^{2}$ related by $\Omega_{5}$-move}
\label{fig:Omega_5ForTwoFibers}
\end{figure}

\begin{proof}
Assume that $\varepsilon = 0$. Using b3) in the definition of $\langle \cdot \rangle_{\star \star}$ we see that, after using \eqref{eqn:kbsm_xm_Bracket_Sigma} and since $F_{-1} = -A^{3}$,
\begin{equation*}
\llangle x_{-\nu_{2}-1} \rrangle_{\star\star} = \llangle x_{\nu_{1}} F_{\nu_{0}+1} \rrangle_{\star\star} = \llangle F_{-1}x_{-\nu_{2}} \rrangle_{\star\star} = -A^{3} \llangle x_{-\nu_{2}} \rrangle_{\star\star}.
\end{equation*}
Therefore \eqref{eqn:lem_Relations_Bracket_Two_Fibers_Case_1a} holds when $n = 0$.

Using b3) in the definition of $\langle \cdot \rangle_{\star \star}$, we see that 
\begin{equation*}
\llangle x_{\nu_{1}} F_{\nu_{0}+2} \rrangle_{\star\star} = \llangle F_{-2}x_{-\nu_{2}} \rrangle_{\star\star}.
\end{equation*}
By \eqref{eqn:kbsm_xm_Bracket_Sigma} and \eqref{eqn:rel_xm_to_Qxk}
\begin{equation*}
\llangle x_{\nu_{1}} F_{\nu_{0}+2} \rrangle_{\star\star} = \llangle x_{-\nu_{2}-2} \rrangle_{\star\star} = A \llangle \lambda x_{-\nu_{2}-1} \rrangle_{\star\star} - A^{2} \llangle x_{-\nu_{2}} \rrangle_{\star\star},
\end{equation*}
on the other hand, since $F_{-2} = -A^{2}-A^{4}\lambda$,
\begin{equation*}
\llangle F_{-2} x_{-\nu_{2}} \rrangle_{\star\star} = -A^{2} \llangle x_{-\nu_{2}} \rrangle_{\star\star} - A^{4} \llangle \lambda x_{-\nu_{2}} \rrangle_{\star\star},
\end{equation*}
it follows that
\begin{equation*}
A\llangle \lambda x_{-\nu_{2}-1} \rrangle_{\star\star} = -A^{4} \llangle \lambda x_{-\nu_{2}} \rrangle_{\star\star},
\end{equation*}
which proves \eqref{eqn:lem_Relations_Bracket_Two_Fibers_Case_1a} for $n = 1$.

As we noted in Remark~\ref{rem:NoteOnGeneratingSetForKBSM}, $\lambda^{n}$ is $R$-linear combination of $P_{k}$, $0\leq k \leq n$, it suffices to show that
\begin{equation*}
\llangle P_{n}x_{-\nu_{2}-1} \rrangle_{\star\star} = -A^{3} \llangle P_{n}x_{-\nu_{2}} \rrangle_{\star\star}
\end{equation*}
for any $n \geq 2$. Since arrow diagrams $D$ and $D'$ in Figure~\ref{fig:Omega_5ForTwoFibers} are related by $\Omega_{5}$-move, by \eqref{eqn:Compositionh4h2h1}, $\phi_{\beta_{1}}(D) = \phi_{\beta_{1}}(D')$ or
\begin{equation*}
A\llangle x_{-n-\nu_{2}-1} \rrangle_{\Sigma^{\prime}_{\nu_{1}}} + A^{-1}\llangle P_{n}x_{-\nu_{2}-1} \rrangle_{\Sigma^{\prime}_{\nu_{1}}} = A\llangle P_{n-1}x_{-\nu_{2}} \rrangle_{\Sigma^{\prime}_{\nu_{1}}} + A^{-1}\llangle x_{-\nu_{2}-n+1} \rrangle_{\Sigma^{\prime}_{\nu_{1}}}.
\end{equation*}
Thus, by \eqref{eqn:rel_F_to_P}, \eqref{eqn:kbsm_xm_Bracket_Sigma}, and part b3) of the definition of $\langle \cdot \rangle_{\star \star}$,
\begin{eqnarray*}
\llangle P_{n} x_{-\nu_{2}-1} \rrangle_{\star\star} 
&=& A^{2} \llangle P_{n-1} x_{-\nu_{2}} \rrangle_{\star\star} + \llangle x_{-n-\nu_{2}+1} \rrangle_{\star\star} - A^{2} \llangle x_{-n-\nu_{2}-1} \rrangle_{\star\star} \\
&=& A^{2} \llangle (-A^{-2}F_{-n+1} + A^{-1}F_{-n}) x_{-\nu_{2}} \rrangle_{\star\star} + \llangle x_{\nu_{1}}F_{\nu_{0}+n-1} \rrangle_{\star\star} - A^{2} \llangle x_{\nu_{1}}F_{\nu_{0}+n+1} \rrangle_{\star\star} \\
&=& A^{2} \llangle (-A^{-2}F_{-n+1} + A^{-1}F_{-n}) x_{-\nu_{2}} \rrangle_{\star\star} + \llangle F_{-n+1}x_{-\nu_{2}} \rrangle_{\star\star} - A^{2} \llangle F_{-n-1}x_{-\nu_{2}} \rrangle_{\star\star} \\
&=& \llangle (AF_{-n} - A^{2}F_{-n-1}) x_{-\nu_{2}} \rrangle_{\star\star} = -A^{3} \llangle P_{n} x_{-\nu_{2}} \rrangle_{\star\star},
\end{eqnarray*}
which proves \eqref{eqn:lem_Relations_Bracket_Two_Fibers_Case_1a} for $n \geq 2$. 

Assume $\varepsilon = 1$. Using part b2) in the definition of $\langle \cdot \rangle_{\star \star}$, \eqref{eqn:kbsm_xvxm_Bracket_Sigma} and $F_{-1} = -A^{3}$, we see that
\begin{equation*}
\llangle x_{\nu_{1}}x_{-\nu_{2}-1} \rrangle_{\star\star} = \llangle R_{-\nu_{0}-1} \rrangle_{\star\star} = \llangle x_{\nu_{1}}F_{-1}x_{-\nu_{2}} \rrangle_{\star\star} = -A^{3} \llangle x_{\nu_{1}}x_{-\nu_{2}} \rrangle_{\star\star},
\end{equation*}
which proves \eqref{eqn:lem_Relations_Bracket_Two_Fibers_Case_1a} for $n = 0$. By part b2) in the definition of $\langle \cdot \rangle_{\star \star}$ we see that, 
\begin{equation*}
\llangle R_{-\nu_{0}-2} \rrangle_{\star\star} = \llangle x_{\nu_{1}}F_{-2}x_{-\nu_{2}} \rrangle_{\star\star}.
\end{equation*}
By \eqref{eqn:kbsm_xvxm_Bracket_Sigma} and \eqref{eqn:rel_xm_to_Qxk}
\begin{equation*}
\llangle R_{-\nu_{0}-2} \rrangle_{\star\star} = \llangle x_{\nu_{1}}x_{-\nu_{2}-2} \rrangle_{\star\star} = A \llangle x_{\nu_{1}}\lambda x_{-\nu_{2}-1} \rrangle_{\star\star} - A^{2} \llangle x_{\nu_{1}}x_{-\nu_{2}} \rrangle_{\star\star},
\end{equation*}
and, on the other hand, since $F_{-2} = -A^{2}-A^{4}\lambda$,
\begin{equation*}
\llangle x_{\nu_{1}} F_{-2} x_{-\nu_{2}} \rrangle_{\star\star} = -A^{2} \llangle x_{\nu_{1}} x_{-\nu_{2}} \rrangle_{\star\star} - A^{4} \llangle x_{\nu_{1}} \lambda x_{-\nu_{2}} \rrangle_{\star\star},
\end{equation*}
it follows that
\begin{equation*}
A\llangle x_{\nu_{1}} \lambda x_{-\nu_{2}-1} \rrangle_{\star\star} = -A^{4} \llangle x_{\nu_{1}} \lambda x_{-\nu_{2}} \rrangle_{\star\star}.
\end{equation*}
Therefore, \eqref{eqn:lem_Relations_Bracket_Two_Fibers_Case_1a} holds for $n = 1$.

We show that for any $n \geq 2$,
\begin{equation*}
\llangle x_{\nu_{1}} P_{n}x_{-\nu_{2}-1} \rrangle_{\star\star} = -A^{3} \llangle x_{\nu_{1}} P_{n}x_{-\nu_{2}} \rrangle_{\star\star}.
\end{equation*}
Since arrow diagrams $D$ and $D'$ in Figure~\ref{fig:Omega_5ForTwoFibers} are related by $\Omega_{5}$-move, by \eqref{eqn:Compositionh4h2h1}, $\phi_{\beta_{1}}(D) = \phi_{\beta_{1}}(D')$ or
\begin{equation*}
A\llangle x_{\nu_{1}}x_{-n-\nu_{2}-1} \rrangle_{\Sigma^{\prime}_{\nu_{1}}} + A^{-1}\llangle x_{\nu_{1}}P_{n}x_{-\nu_{2}-1} \rrangle_{\Sigma^{\prime}_{\nu_{1}}} = A\llangle x_{\nu_{1}}P_{n-1}x_{-\nu_{2}} \rrangle_{\Sigma^{\prime}_{\nu_{1}}} + A^{-1}\llangle x_{\nu_{1}}x_{-\nu_{2}-n+1} \rrangle_{\Sigma^{\prime}_{\nu_{1}}}.
\end{equation*}
Thus, by \eqref{eqn:rel_F_to_P}, \eqref{eqn:kbsm_xvxm_Bracket_Sigma}, and part b2) in the definition of $\langle \cdot \rangle_{\star \star}$ gives
\begin{eqnarray*}
& &\llangle x_{\nu_{1}} P_{n} x_{-\nu_{2}-1} \rrangle_{\star\star} 
= A^{2} \llangle x_{\nu_{1}} P_{n-1} x_{-\nu_{2}} \rrangle_{\star\star} + \llangle x_{\nu_{1}} x_{-n-\nu_{2}+1} \rrangle_{\star\star} - A^{2} \llangle x_{\nu_{1}} x_{-n-\nu_{2}-1} \rrangle_{\star\star} \\
&=& A^{2} \llangle x_{\nu_{1}} (-A^{-2}F_{-n+1} + A^{-1}F_{-n}) x_{-\nu_{2}} \rrangle_{\star\star} + \llangle R_{-\nu_{0}-n+1} \rrangle_{\star\star} - A^{2} \llangle R_{-\nu_{0}-n-1} \rrangle_{\star\star} \\
&=& A^{2} \llangle x_{\nu_{1}} (-A^{-2}F_{-n+1} + A^{-1}F_{-n}) x_{-\nu_{2}} \rrangle_{\star\star} + \llangle x_{\nu_{1}} F_{-n+1}x_{-\nu_{2}} \rrangle_{\star\star} - A^{2} \llangle x_{\nu_{1}} F_{-n-1}x_{-\nu_{2}} \rrangle_{\star\star} \\
&=& \llangle x_{\nu_{1}} (AF_{-n} - A^{2}F_{-n-1}) x_{-\nu_{2}} \rrangle_{\star\star} = -A^{3} \llangle x_{\nu_{1}} P_{n} x_{-\nu_{2}} \rrangle_{\star\star}.
\end{eqnarray*}
Thus, using Remark~\ref{rem:NoteOnGeneratingSetForKBSM} we see that \eqref{eqn:lem_Relations_Bracket_Two_Fibers_Case_1a} holds for $n \geq 2$. 
\end{proof}
 
\begin{lemma}
\label{lem:Relations_Bracket_Two_Fibers_Case_1}
Let $\nu_{0}\geq 0$, then for all $m \in \mathbb{Z}$,
\begin{equation}
\label{eqn:Bracket_Two_Fibers_Fx_xF_Case_1}
\llangle F_{m}x_{-\nu_{2}} \rrangle_{\star\star} = \llangle x_{\nu_{1}}F_{\nu_{0}-m} \rrangle_{\star\star}
\end{equation}
and
\begin{equation}      
\label{eqn:Bracket_Two_Fibers_xFx_R_Case_1}
\llangle x_{\nu_{1}} F_{m}x_{-\nu_{2}} \rrangle_{\star\star}  = \llangle R_{m-\nu_{0}} \rrangle_{\star\star}.
\end{equation}
\end{lemma}

\begin{figure}[H]
\centering
\includegraphics[scale=0.8]{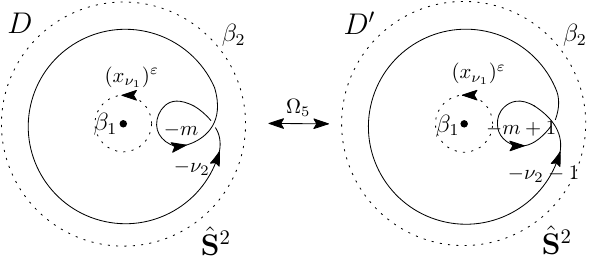}
\caption{Arrow diagrams $D$ and $D'$ related by $\Omega_{5}$-move}
\label{fig:Omega_5LemmaForTwoFibers}
\end{figure}

\begin{proof}
By the definition of $\llangle \cdot \rrangle_{\star\star}$, \eqref{eqn:Bracket_Two_Fibers_Fx_xF_Case_1} and \eqref{eqn:Bracket_Two_Fibers_xFx_R_Case_1} hold for $m \leq -1$. 

Since arrow diagrams $D$ and $D'$ in Figure~\ref{fig:Omega_5LemmaForTwoFibers} are related by $\Omega_{5}$-move, by \eqref{eqn:Compositionh4h2h1}, $\phi_{\beta_{1}}(D) = \phi_{\beta_{1}}(D')$ or
\begin{equation*}
A \llangle P_{-m}x_{-\nu_{2}} \rrangle_{\Sigma^{\prime}_{\nu_{1}}} + A^{-1} \llangle x_{m-\nu_{2}} \rrangle_{\Sigma^{\prime}_{\nu_{1}}} = A \llangle x_{m-\nu_{2}-2} \rrangle_{\Sigma^{\prime}_{\nu_{1}}} + A^{-1} \llangle P_{-m+1}x_{-\nu_{2}-1} \rrangle_{\Sigma^{\prime}_{\nu_{1}}}.
\end{equation*}
Moreover, by \eqref{eqn:rel_F_to_P} and \eqref{eqn:lem_Relations_Bracket_Two_Fibers_Case_1a}, the above equation becomes
\begin{equation*}
A \llangle (A^{-1}F_{m-1} - A^{-2}F_{m})x_{-\nu_{2}} \rrangle_{\star\star} + A^{-1} \llangle x_{m-\nu_{2}} \rrangle_{\star\star} = A \llangle x_{m-\nu_{2}-2} \rrangle_{\star\star} - A^{2} \llangle (A^{-1}F_{m-2} - A^{-2}F_{m-1})x_{-\nu_{2}} \rrangle_{\star\star},
\end{equation*}
which by \eqref{eqn:kbsm_xm_Bracket_Sigma} can be written as
\begin{equation*}
A^{-1} (\llangle x_{\nu_{1}}F_{\nu_{0}-m} \rrangle_{\star\star} - \llangle F_{m}x_{-\nu_{2}} \rrangle_{\star\star}) = A (\llangle x_{\nu_{1}}F_{\nu_{0}-m+2} \rrangle_{\star\star} - \llangle F_{m-2}x_{-\nu_{2}} \rrangle_{\star\star}).
\end{equation*}
Therefore, using induction on $m$ we can see that \eqref{eqn:Bracket_Two_Fibers_Fx_xF_Case_1} holds for all $m \in \mathbb{Z}$.

Since arrow diagrams $D$ and $D'$ in Figure~\ref{fig:Omega_5LemmaForTwoFibers} are related by $\Omega_{5}$-move, by \eqref{eqn:Compositionh4h2h1}, $\phi_{\beta_{1}}(D) = \phi_{\beta_{1}}(D')$ or
\begin{equation*}
A \llangle x_{\nu_{1}}P_{-m}x_{-\nu_{2}} \rrangle_{\Sigma^{\prime}_{\nu_{1}}} + A^{-1} \llangle x_{\nu_{1}}x_{m-\nu_{2}} \rrangle_{\Sigma^{\prime}_{\nu_{1}}} = A \llangle x_{\nu_{1}}x_{m-\nu_{2}-2} \rrangle_{\Sigma^{\prime}_{\nu_{1}}} + A^{-1} \llangle x_{\nu_{1}}P_{-m+1}x_{-\nu_{2}-1} \rrangle_{\Sigma^{\prime}_{\nu_{1}}}.
\end{equation*}
Moreover, by \eqref{eqn:rel_F_to_P} and \eqref{eqn:lem_Relations_Bracket_Two_Fibers_Case_1a}, the above equation becomes
\begin{eqnarray*}
& & A \llangle x_{\nu_{1}}(A^{-1}F_{m-1} - A^{-2}F_{m})x_{-\nu_{2}} \rrangle_{\star\star} + A^{-1} \llangle x_{\nu_{1}}x_{m-\nu_{2}} \rrangle_{\star\star} \\
&=& A \llangle x_{\nu_{1}}x_{m-\nu_{2}-2} \rrangle_{\star\star} - A^{2} \llangle x_{\nu_{1}}(A^{-1}F_{m-2} - A^{-2}F_{m-1})x_{-\nu_{2}} \rrangle_{\star\star},
\end{eqnarray*}
which by \eqref{eqn:kbsm_xvxm_Bracket_Sigma} can be written as 
\begin{equation*}
A^{-1} (\llangle R_{m-\nu_{0}} \rrangle_{\star\star} - \llangle x_{\nu_{1}}F_{m}x_{-\nu_{2}} \rrangle_{\star\star}) = A (\llangle R_{m-\nu_{0}-2} \rrangle_{\star\star} - \llangle x_{\nu_{1}}F_{m-2}x_{-\nu_{2}} \rrangle_{\star\star}).
\end{equation*}
Therefore, using induction on $m$ we see that \eqref{eqn:Bracket_Two_Fibers_xFx_R_Case_1} holds for all $m \in \mathbb{Z}$.
\end{proof}

\begin{lemma}
\label{lem:Relations_Bracket_Two_Fibers_Case_2a}
Let $\nu_{0} \leq -2$, then for any $\varepsilon \in \{0,1\}$ and $n \geq 0$,
\begin{equation}
\label{eqn:lem_Relations_Bracket_Two_Fibers_Case_1b}
\llangle (x_{\nu_{1}})^{\varepsilon}\lambda^{n}x_{-\nu_{2}-1} \rrangle_{\star\star} = -A^{3} 
\llangle (x_{\nu_{1}})^{\varepsilon}\lambda^{n}x_{-\nu_{2}} \rrangle_{\star\star}.
\end{equation}
\end{lemma}

\begin{proof}
Assume that $\varepsilon = 0$. Using part c3) in the definition of $\langle \cdot \rangle_{\star \star}$, we see that using \eqref{eqn:kbsm_xm_Bracket_Sigma} and since $F_{0} = 1$,
\begin{equation*}
\llangle x_{-\nu_{2}-1} \rrangle_{\star\star} = \llangle F_{0}x_{-\nu_{2}-1} \rrangle_{\star\star} = -A^{3}\llangle x_{\nu_{1}}F_{\nu_{0}} \rrangle_{\star\star} = -A^{3} \llangle x_{-\nu_{2}} \rrangle_{\star\star},
\end{equation*}
which proves \eqref{eqn:lem_Relations_Bracket_Two_Fibers_Case_1b} for $n = 0$. Using part c3) in the definition of $\langle \cdot \rangle_{\star\star}$, we see that 
\begin{equation*}
\llangle x_{\nu_{1}}F_{\nu_{0}-1} \rrangle_{\star\star} = -A^{-3}\llangle F_{1} x_{-\nu_{2}-1}\rrangle_{\star\star},
\end{equation*}
By \eqref{eqn:kbsm_xm_Bracket_Sigma} and \eqref{eqn:rel_xm_to_Qxk}
\begin{equation*}
\llangle x_{\nu_{1}}F_{\nu_{0}-1} \rrangle_{\star\star} = \llangle x_{-\nu_{2}+1} \rrangle_{\star\star} = A^{-1}\llangle \lambda x_{-\nu_{2}} \rrangle_{\star\star} - A^{-2}\llangle x_{-\nu_{2}-1} \rrangle_{\star\star},
\end{equation*}
on the other hand, since $F_{1} = A^{-1}\lambda + A$,
\begin{equation*}
 -A^{-3}\llangle F_{1} x_{-\nu_{2}-1}\rrangle_{\star\star} =  -A^{-4}\llangle \lambda x_{-\nu_{2}-1} \rrangle_{\star\star} - A^{-2}\llangle x_{-\nu_{2}-1} \rrangle_{\star\star},
\end{equation*}
it follows that
\begin{equation*}
-A^{-4}\llangle \lambda x_{-\nu_{2}-1} \rrangle_{\star\star} = A^{-1} \llangle \lambda x_{-\nu_{2}} \rrangle_{\star\star},
\end{equation*}
which proves \eqref{eqn:lem_Relations_Bracket_Two_Fibers_Case_1b} for $n = 1$.

\begin{figure}[H]
\centering
\includegraphics[scale=0.7]{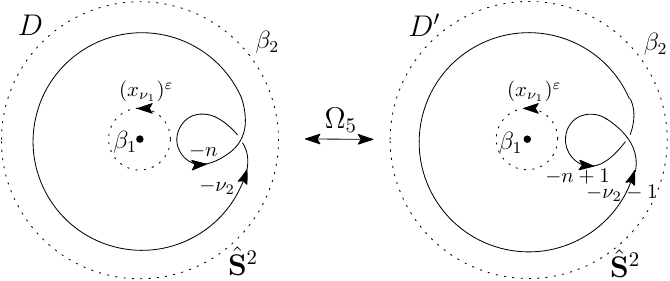}
\caption{Arrow diagrams $D$ and $D'$ related by $\Omega_{5}$-move}
\label{fig:Omega_5ForTwoFibersNegativeNu}
\end{figure}

We prove that for any $n \geq 2$,
\begin{equation*}
\llangle P_{-n}x_{-\nu_{2}} \rrangle_{\star\star} = -A^{-3} \llangle P_{-n}x_{-\nu_{2}-1} \rrangle_{\star\star}.
\end{equation*}
Since arrow diagrams $D$ and $D'$ in Figure~\ref{fig:Omega_5ForTwoFibersNegativeNu} are related by $\Omega_{5}$-move, by \eqref{eqn:Compositionh4h2h1}, $\phi_{\beta_{1}}(D) = \phi_{\beta_{1}}(D')$ or
\begin{equation*}
A\llangle P_{-n}x_{-\nu_{2}} \rrangle_{\Sigma^{\prime}_{\nu_{1}}} + A^{-1} \llangle x_{n-\nu_{2}} \rrangle_{\Sigma^{\prime}_{\nu_{1}}} = A\llangle x_{n-\nu_{2}-2} \rrangle_{\Sigma^{\prime}_{\nu_{1}}} + A^{-1} \llangle P_{-n+1}x_{-\nu_{2}-1} \rrangle_{\Sigma^{\prime}_{\nu_{1}}}.    
\end{equation*}
Therefore, by \eqref{eqn:rel_F_to_P}, \eqref{eqn:kbsm_xm_Bracket_Sigma}, and part c3) in the definition of $\langle \cdot \rangle_{\star\star}$.
\begin{eqnarray*}
\llangle P_{-n}x_{-\nu_{2}} \rrangle_{\star\star} &=& A^{-2} \llangle P_{-n+1}x_{-\nu_{2}-1} \rrangle_{\star\star} + \llangle x_{n-\nu_{2}-2} \rrangle_{\star\star} - A^{-2} \llangle x_{n-\nu_{2}} \rrangle_{\star\star}\\ 
&=& A^{-2}\llangle (-A^{-2}F_{n-1} + A^{-1}F_{n-2})x_{-\nu_{2}-1} \rrangle_{\star\star} + \llangle x_{\nu_{1}}F_{\nu_{0} -n+2} \rrangle_{\star\star} - A^{-2}\llangle x_{\nu_{1}}F_{\nu_{0} - n} \rrangle_{\star\star}\\
&=& A^{-2}\llangle (-A^{-2}F_{n-1} + A^{-1}F_{n-2})x_{-\nu_{2}-1} \rrangle_{\star\star} - A^{-3}\llangle F_{n-2} x_{-\nu_{2}-1}\rrangle_{\star\star} + A^{-5}\llangle F_{n} x_{-\nu_{2}-1}\rrangle_{\star\star}\\
&=& - A^{-3}\llangle (- A^{-2} F_{n} + A^{-1}F_{n-1}) x_{-\nu_{2}-1} \rrangle_{\star\star} = -A^{-3} \llangle P_{-n} x_{-\nu_{2}-1} \rrangle_{\star\star}.
\end{eqnarray*}
Consequently, \eqref{eqn:lem_Relations_Bracket_Two_Fibers_Case_1b} holds for $n \geq 2$ by Remark~\ref{rem:NoteOnGeneratingSetForKBSM}.

Assume $\varepsilon = 1$. Using part c2) in the definition of $\langle \cdot \rangle_{\star\star}$, we see that using \eqref{eqn:kbsm_xvxm_Bracket_Sigma} and since $F_{0} = 1$,
\begin{equation*}
-A^{-3} \llangle x_{\nu_{1}}x_{-\nu_{2}-1} \rrangle_{\star\star} = -A^{-3} \llangle x_{\nu_{1}}F_{0}x_{-\nu_{2}-1} \rrangle_{\star\star} = \llangle R_{-\nu_{0}} \rrangle_{\star\star} = \llangle x_{\nu_{1}}x_{-\nu_{2}} \rrangle_{\star\star},
\end{equation*}
which proves \eqref{eqn:lem_Relations_Bracket_Two_Fibers_Case_1a} for $n = 0$. Using part c2) in the definition of $\langle \cdot \rangle_{\star\star}$ we see that
\begin{equation*}
-A^{-3} \llangle x_{\nu_{1}}F_{1}x_{-\nu_{2}-1} \rrangle_{\star\star} = \llangle R_{1-\nu_{0}} \rrangle_{\star\star}.
\end{equation*}
Since $F_{1} = A^{-1}\lambda + A$, the left hand side of the above equation becomes
\begin{equation*}
-A^{-3} \llangle x_{\nu_{1}}F_{1}x_{-\nu_{2}-1} \rrangle_{\star\star} = -A^{-4} \llangle x_{\nu_{1}} \lambda x_{-\nu_{2}-1} \rrangle_{\star\star} - A^{-2} \llangle x_{\nu_{1}}x_{-\nu_{2}-1} \rrangle_{\star\star},
\end{equation*}
on the other hand, by \eqref{eqn:kbsm_xvxm_Bracket_Sigma} and \eqref{eqn:rel_xm_to_Qxk}
\begin{equation*}
\llangle R_{1-\nu_{0}} \rrangle_{\star\star} = \llangle x_{\nu_{1}}x_{-\nu_{2}+1} \rrangle_{\star\star} = A^{-1} \llangle x_{\nu_{1}} \lambda x_{-\nu_{2}} \rrangle_{\star\star} -A^{-2} \llangle x_{\nu_{1}}x_{-\nu_{2}-1} \rrangle_{\star\star},
\end{equation*}
it follows that $-A^{-4} \llangle x_{\nu_{1}} \lambda x_{-\nu_{2}-1} \rrangle_{\star\star} = A^{-1} \llangle x_{\nu_{1}} \lambda x_{-\nu_{2}} \rrangle_{\star\star}$, which proves the case $n = 1$ of \eqref{eqn:lem_Relations_Bracket_Two_Fibers_Case_1b}. 

Now we prove that
\begin{equation*}
\llangle x_{\nu_{1}}P_{-n}x_{-\nu_{2}} \rrangle_{\star\star} = -A^{-3} \llangle x_{\nu_{1}}P_{-n}x_{-\nu_{2}-1} \rrangle_{\star\star}
\end{equation*}
Since arrow diagrams $D$ and $D'$ in Figure~\ref{fig:Omega_5ForTwoFibers} are related by $\Omega_{5}$-move, by \eqref{eqn:Compositionh4h2h1}, $\phi_{\beta_{1}}(D) = \phi_{\beta_{1}}(D')$ or
\begin{equation*}
A \llangle x_{\nu_{1}}P_{-n}x_{-\nu_{2}} \rrangle_{\Sigma^{\prime}_{\nu_{1}}}
+ A^{-1} \llangle x_{\nu_{1}}x_{n-\nu_{2}} \rrangle_{\Sigma^{\prime}_{\nu_{1}}} = A \llangle x_{\nu_{1}}x_{n-\nu_{2}-2} \rrangle_{\Sigma^{\prime}_{\nu_{1}}} + A^{-1} \llangle x_{\nu_{1}}P_{-n+1}x_{-\nu_{2}-1} \rrangle_{\Sigma^{\prime}_{\nu_{1}}}. 
\end{equation*}
Moreover, by \eqref{eqn:kbsm_xvxm_Bracket_Sigma}, \eqref{eqn:rel_F_to_P}, and part c2) in the definition of $\langle \cdot \rangle_{\star\star}$, we see that
\begin{eqnarray*}
& & \llangle x_{\nu_{1}}P_{-n}x_{-\nu_{2}} \rrangle_{\star\star}
= A^{-2} \llangle x_{\nu_{1}}P_{-n+1}x_{-\nu_{2}-1} \rrangle_{\star\star} + \llangle R_{-\nu_{0}+n-2} \rrangle_{\star\star} - A^{-2} \llangle R_{-\nu_{0}+n} \rrangle_{\star\star} \\
&=& A^{-2} \llangle x_{\nu_{1}}(A^{-1}F_{n-2} - A^{-2}F_{n-1})x_{-\nu_{2}-1} \rrangle_{\star\star} - A^{-3} \llangle x_{\nu_{1}}F_{n-2}x_{-\nu_{2}-1} \rrangle_{\star\star} + A^{-5} \llangle x_{\nu_{1}}F_{n}x_{-\nu_{2}-1} \rrangle_{\star\star} \\
&=& -A^{-3} \llangle x_{\nu_{1}}(A^{-1}F_{n-1} - A^{-2}F_{n})x_{-\nu_{2}-1} \rrangle_{\star\star} = -A^{-3} \llangle x_{\nu_{1}}P_{-n}x_{-\nu_{2}-1} \rrangle_{\star\star}.
\end{eqnarray*}
Therefore, \eqref{eqn:lem_Relations_Bracket_Two_Fibers_Case_1b} holds for $n\geq 2$ by Remark~\ref{rem:NoteOnGeneratingSetForKBSM}.
\end{proof}

\begin{lemma}
\label{lem:Relations_Bracket_Two_Fibers_Case_2}
Let $\nu_{0} \leq -2$, then for all $m \in \mathbb{Z}$,
\begin{equation}
\label{eqn:Bracket_Two_Fibers_Fx_xF_Case_2}
-A^{-3}\llangle F_{m}x_{-\nu_{2}-1} \rrangle_{\star\star} = \llangle x_{\nu_{1}}F_{\nu_{0}-m} \rrangle_{\star\star}
\end{equation}
and
\begin{equation}      
\label{eqn:Bracket_Two_Fibers_xFx_R_Case_2}
-A^{-3}\llangle x_{\nu_{1}} F_{m}x_{-\nu_{2}-1} \rrangle_{\star\star}  = \llangle R_{m-\nu_{0}} \rrangle_{\star\star}.
\end{equation}
\end{lemma}

\begin{proof}
By the definition of $\llangle \cdot \rrangle_{\star\star}$, \eqref{eqn:Bracket_Two_Fibers_Fx_xF_Case_2} and \eqref{eqn:Bracket_Two_Fibers_xFx_R_Case_2} hold for $m \geq 0$. Since arrow diagrams $D$ and $D'$ in Figure~\ref{fig:Omega_5LemmaForTwoFibers} are related by $\Omega_{5}$-move, by \eqref{eqn:Compositionh4h2h1}, $\phi_{\beta_{1}}(D) = \phi_{\beta_{1}}(D')$ or
\begin{equation*}
A \llangle P_{-m}x_{-\nu_{2}} \rrangle_{\Sigma^{\prime}_{\nu_{1}}} + A^{-1} \llangle x_{m-\nu_{2}} \rrangle_{\Sigma^{\prime}_{\nu_{1}}} = A \llangle x_{m-\nu_{2}-2} \rrangle_{\Sigma^{\prime}_{\nu_{1}}} + A^{-1} \llangle P_{-m+1}x_{-\nu_{2}-1} \rrangle_{\Sigma^{\prime}_{\nu_{1}}}.
\end{equation*}
By \eqref{eqn:rel_F_to_P} and \eqref{eqn:lem_Relations_Bracket_Two_Fibers_Case_1b}, above equation becomes
\begin{eqnarray*}
& &-A^{-2} \llangle (A^{-1}F_{m-1} - A^{-2}F_{m})x_{-\nu_{2}-1} \rrangle_{\star\star} + A^{-1} \llangle x_{m-\nu_{2}} \rrangle_{\star\star} \\
&=& A \llangle x_{m-\nu_{2}-2} \rrangle_{\star\star} + A^{-1} \llangle (A^{-1}F_{m-2} - A^{-2}F_{m-1})x_{-\nu_{2}-1} \rrangle_{\star\star},
\end{eqnarray*}
which by \eqref{eqn:kbsm_xm} we can write as
\begin{equation*}
A^{-1} (\llangle x_{\nu_{1}}F_{\nu_{0}-m} \rrangle_{\star\star} + A^{-3} \llangle F_{m}x_{-\nu_{2}-1} \rrangle_{\star\star}) = A (\llangle x_{\nu_{1}}F_{\nu_{0}-m+2} \rrangle_{\star\star} + A^{-3} \llangle F_{m-2}x_{-\nu_{2}-1} \rrangle_{\star\star}).
\end{equation*}
Therefore, by induction on $m$, \eqref{eqn:Bracket_Two_Fibers_Fx_xF_Case_2} holds for all $m \in \mathbb{Z}$.

Since arrow diagrams $D$ and $D'$ in Figure~\ref{fig:Omega_5LemmaForTwoFibers} are related by $\Omega_{5}$-move, by \eqref{eqn:Compositionh4h2h1}, $\phi_{\beta_{1}}(D) = \phi_{\beta_{1}}(D')$ or
\begin{equation*}
A \llangle x_{\nu_{1}}P_{-m}x_{-\nu_{2}} \rrangle_{\Sigma^{\prime}_{\nu_{1}}} + A^{-1} \llangle x_{\nu_{1}}x_{m-\nu_{2}} \rrangle_{\Sigma^{\prime}_{\nu_{1}}} = A \llangle x_{\nu_{1}}x_{m-\nu_{2}-2} \rrangle_{\Sigma^{\prime}_{\nu_{1}}} + A^{-1} \llangle x_{\nu_{1}}P_{-m+1}x_{-\nu_{2}-1} \rrangle_{\Sigma^{\prime}_{\nu_{1}}}.
\end{equation*}
By \eqref{eqn:rel_F_to_P} and \eqref{eqn:lem_Relations_Bracket_Two_Fibers_Case_1b}, the above equation becomes
\begin{eqnarray*}
& & -A^{-2} \llangle x_{\nu_{1}}(A^{-1}F_{m-1} - A^{-2}F_{m})x_{-\nu_{2}-1} \rrangle_{\star\star} + A^{-1} \llangle x_{\nu_{1}}x_{m-\nu_{2}} \rrangle_{\star\star} \\
&=& A \llangle x_{\nu_{1}}x_{m-\nu_{2}-2} \rrangle_{\star\star} + A^{-1} \llangle x_{\nu_{1}}(A^{-1}F_{m-2} - A^{-2}F_{m-1})x_{-\nu_{2}-1} \rrangle_{\star\star},
\end{eqnarray*}
which by \eqref{eqn:kbsm_xvxm} can be written as 
\begin{equation*}
A^{-1} (\llangle R_{m-\nu_{0}} \rrangle_{\star\star} + A^{-3} \llangle x_{\nu_{1}}F_{m}x_{-\nu_{2}-1} \rrangle_{\star\star}) = A (\llangle R_{m-\nu_{2}-2} \rrangle_{\star\star} + A^{-3} \llangle x_{\nu_{1}}F_{m-2}x_{-\nu_{2}-1} \rrangle_{\star\star}).
\end{equation*}
Therefore, using induction on $m$, \eqref{eqn:Bracket_Two_Fibers_xFx_R_Case_2} holds for all $m \in \mathbb{Z}$.
\end{proof}

We summarize results of Lemma~\ref{lem:Relations_Bracket_Two_Fibers_Case_1a}--Lemma~\ref{lem:Relations_Bracket_Two_Fibers_Case_2} as the following corollary.

\begin{corollary}
\label{cor:SummaryOfLemmasTwoFiberCase}
For $\nu_{0} \neq -1$, $m \in \mathbb{Z}$, $\varepsilon \in \{0,1\}$, and $n \geq 0$,
\begin{equation}
\label{eqn:Bracket_Two_Fibers_Fx_xF}
\llangle F_{m}x_{-\nu_{2}} \rrangle_{\star\star} = \llangle x_{\nu_{1}}F_{\nu_{0}-m} \rrangle_{\star\star},
\end{equation}
\begin{equation}      
\label{eqn:Bracket_Two_Fibers_xFx_R}
\llangle x_{\nu_{1}} F_{m}x_{-\nu_{2}} \rrangle_{\star\star}  = \llangle R_{m-\nu_{0}} \rrangle_{\star\star},
\end{equation}
and
\begin{equation}
\label{eqn:lem_Relations_Bracket_Two_Fibers}
\llangle (x_{\nu_{1}})^{\varepsilon}\lambda^{n}x_{-\nu_{2}-1} \rrangle_{\star\star} = -A^{3} 
\llangle (x_{\nu_{1}})^{\varepsilon}\lambda^{n}x_{-\nu_{2}} \rrangle_{\star\star}.
\end{equation}
\end{corollary}

\medskip

For arrow diagrams $D$, $D'$ in Figure~\ref{fig:SBeta2MovesOnTCurveTwoFibersLemma}, we see that $D = (x_{\nu_{1}})^{\varepsilon} \lambda^{n_{1}} t_{m,n_{2}}$ and $D' = (x_{\nu_{1}})^{\varepsilon} \lambda^{n_{1}} W$. Thus, $D'_{+} = (x_{\nu_{1}})^{\varepsilon} \lambda^{n_{1}} t_{m-1,n_{2}}$ and $D'_{-} = (x_{\nu_{1}})^{\varepsilon} \lambda^{n_{1}} x_{-m-\nu_{2}} \lambda^{n_{2}} x_{-\nu_{2}-1}$ are obtained by smoothing crossing of $W$ according to positive and negative markers.

\begin{figure}[ht]
\centering
\includegraphics[scale=0.95]{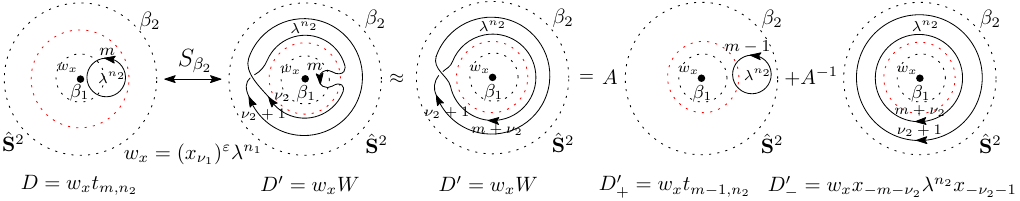}
\caption{Arrow diagrams $D$ and $D'$ related by $S_{\beta_{2}}$-move}
\label{fig:SBeta2MovesOnTCurveTwoFibersLemma}
\end{figure}

\begin{lemma}
\label{lem:Relations_Bracket_Two_Fibers_I}
Assume that $\nu_{0}\neq -1$, then for any $\varepsilon \in \{0,1\}$, $m \in \mathbb{Z}$, and $n_{1},n_{2} \geq 0$,
\begin{equation*}
\llangle (x_{\nu_{1}})^{\varepsilon} \lambda^{n_{1}} P_{m,n_{2}} - A (x_{\nu_{1}})^{\varepsilon} \lambda^{n_{1}} P_{m-1,n_{2}} - A^{-1} (x_{\nu_{1}})^{\varepsilon} \lambda^{n_{1}} x_{-m-\nu_{2}} \lambda^{n_{2}} x_{-\nu_{2}-1} \rrangle_{\star\star} = 0.
\end{equation*}
\end{lemma}

\begin{proof}
By Lemma~\ref{lem:annulus_for_any_kn_bracket_2_beta}, it suffices to show the case $n_{1} = n_{2} = 0$, i.e., we show that for all $m\in \mathbb{Z}$,
\begin{equation*}
\llangle (x_{\nu_{1}})^{\varepsilon} P_{m} \rrangle_{\star\star} = A \llangle (x_{\nu_{1}})^{\varepsilon} P_{m-1} \rrangle_{\star\star} + A^{-1} \llangle (x_{\nu_{1}})^{\varepsilon} x_{-m-\nu_{2}} x_{-\nu_{2}-1} \rrangle_{\star\star}.
\end{equation*}

By \eqref{eqn:kbsm_xm_Bracket_Sigma}, \eqref{eqn:lem_Relations_Bracket_Two_Fibers}, and \eqref{eqn:Bracket_Two_Fibers_xFx_R},
\begin{eqnarray*}
A \llangle P_{m-1} \rrangle_{\star\star} + A^{-1} \llangle x_{-m-\nu_{2}}x_{-\nu_{2}-1} \rrangle_{\star\star} 
&=& A \llangle P_{m-1} \rrangle_{\star\star} - A^{2} \llangle x_{\nu_{1}}F_{\nu_{0}+m}x_{-\nu_{2}} \rrangle_{\star\star} \\
&=& A \llangle P_{m-1} \rrangle_{\star\star} - A^{2} \llangle R_{m} \rrangle_{\star\star} = \llangle P_{m} \rrangle_{\star\star},
\end{eqnarray*}
which proves the case $\varepsilon = 0$.

By \eqref{eqn:kbsm_xvxm_Bracket_Sigma}, \eqref{eqn:lem_Relations_Bracket_Two_Fibers}, \eqref{eqn:rel_F_to_P}, and \eqref{eqn:Bracket_Two_Fibers_Fx_xF},
\begin{eqnarray*}
& & A \llangle x_{\nu_{1}}P_{m-1} \rrangle_{\star\star} + A^{-1} \llangle x_{\nu_{1}}x_{-m-\nu_{2}}x_{-\nu_{2}-1} \rrangle_{\star\star} \\
&=& A \llangle x_{\nu_{1}}P_{m-1} \rrangle_{\star\star} + A^{-1} \llangle R_{-m-\nu_{0}}x_{-\nu_{2}-1} \rrangle_{\star\star} \\
&=& A \llangle x_{\nu_{1}}P_{m-1} \rrangle_{\star\star} - A^{2} \llangle (A^{-1}P_{-m-\nu_{0}-1} - A^{-2}P_{-m-\nu_{0}})x_{-\nu_{2}} \rrangle_{\star\star} \\
&=& A \llangle x_{\nu_{1}}P_{m-1} \rrangle_{\star\star} - A^{2} \llangle (-A^{-3}F_{m+\nu_{0}+1} + A^{-2}F_{m+\nu_{0}} + A^{-4}F_{m+\nu_{0}} - A^{-3}F_{m+\nu_{0}-1})x_{-\nu_{2}} \rrangle_{\star\star} \\
&=& A \llangle x_{\nu_{1}}(-A^{-2}F_{-m+1} + A^{-1}F_{-m}) \rrangle_{\star\star} - A^{2} \llangle x_{\nu_{1}}(-A^{-3}F_{-m-1} + A^{-2}F_{-m} + A^{-4}F_{-m} - A^{-3}F_{-m+1}) \rrangle_{\star\star} \\
&=& \llangle x_{\nu_{1}}(A^{-1}F_{-m-1} - A^{-2}F_{-m}) \rrangle_{\star\star} = \llangle x_{\nu_{1}}P_{m} \rrangle_{\star\star}
\end{eqnarray*}
which proves the case $\varepsilon = 1$.
\end{proof}
For arrow diagrams $D$, $D'$ in Figure~\ref{fig:SBeta2MovesOnXCurveTwoFibersLemma}, we see that $D = (x_{\nu_{1}})^{\varepsilon} \lambda^{n_{1}} x_{m} \lambda^{n_{2}}$ and $D' = (x_{\nu_{1}})^{\varepsilon} \lambda^{n_{1}} W$. Thus, $D'_{+} = (x_{\nu_{1}})^{\varepsilon} \lambda^{n_{1}}x_{m-1}\lambda^{n_{2}}$ and $D'_{-} = (x_{\nu_{1}})^{\varepsilon} \lambda^{n_{1}} t_{-\nu_{2}-m,n_{2}}x_{-\nu_{2}-1}$ are obtained by smoothing crossing of $W$ according to positive and negative markers.
\begin{figure}[ht]
\centering
\includegraphics[scale=0.94]{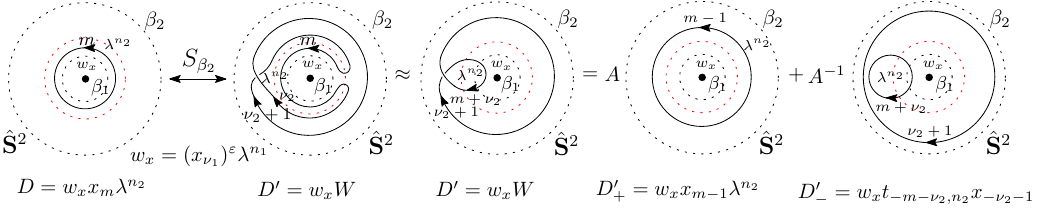}
\caption{Arrow diagrams $D$ and $D'$ related by $S_{\beta_{2}}$-move}
\label{fig:SBeta2MovesOnXCurveTwoFibersLemma}
\end{figure}

\begin{lemma}
\label{lem:Relations_Bracket_Two_Fibers_II}
Assume that $\nu_{0}\neq -1$, then for any $\varepsilon \in \{0,1\}$, $m \in \mathbb{Z}$, and $n_{1},n_{2} \geq 0$,
\begin{equation*}
\llangle (x_{\nu_{1}})^{\varepsilon} \lambda^{n_{1}} x_{m} \lambda^{n_{2}} - A (x_{\nu_{1}})^{\varepsilon} \lambda^{n_{1}} x_{m-1} \lambda^{n_{2}} - A^{-1} (x_{\nu_{1}})^{\varepsilon} \lambda^{n_{1}} P_{-m-\nu_{2},n_{2}} x_{-\nu_{2}-1} \rrangle_{\star\star} = 0.
\end{equation*}
\end{lemma}

\begin{proof}
By Lemma~\ref{lem:annulus_for_any_kn_bracket_2_beta}, it suffices to show the case $n_{1} = n_{2} = 0$, i.e., we show that for all $m\in \mathbb{Z}$,
\begin{equation*}
\llangle (x_{\nu_{1}})^{\varepsilon} x_{m} \rrangle_{\star\star} = A \llangle (x_{\nu_{1}})^{\varepsilon} x_{m-1} \rrangle_{\star\star} + A^{-1} \llangle (x_{\nu_{1}})^{\varepsilon} P_{-m-\nu_{2}} x_{-\nu_{2}-1} \rrangle_{\star\star}.
\end{equation*}

By \eqref{eqn:rel_F_to_P}, \eqref{eqn:lem_Relations_Bracket_Two_Fibers}, \eqref{eqn:kbsm_xm_Bracket_Sigma}, and \eqref{eqn:Bracket_Two_Fibers_Fx_xF},
\begin{eqnarray*}
& & A \llangle x_{m-1} \rrangle_{\star\star} + A^{-1} \llangle P_{-m-\nu_{2}}x_{-\nu_{2}-1} \rrangle_{\star\star} \\
&=& A \llangle x_{m-1} \rrangle_{\star\star} - A^{2} \llangle (A^{-1}F_{m+\nu_{2}-1} - A^{-2}F_{m+\nu_{2}})x_{-\nu_{2}} \rrangle_{\star\star} \\
&=& A \llangle x_{\nu_{1}}F_{\nu_{1}-m+1} \rrangle_{\star\star} - A^{2} \llangle x_{\nu_{1}}(A^{-1}F_{-m+\nu_{1}+1} - A^{-2}F_{-m+\nu_{1}}) \rrangle_{\star\star} \\
&=& \llangle x_{\nu_{1}}F_{\nu_{1}-m} \rrangle_{\star\star} = \llangle x_{m} \rrangle_{\star\star},
\end{eqnarray*}
which proves the case $\varepsilon = 0$.

By \eqref{eqn:rel_F_to_P}, \eqref{eqn:lem_Relations_Bracket_Two_Fibers}, \eqref{eqn:kbsm_xvxm_Bracket_Sigma}, and \eqref{eqn:Bracket_Two_Fibers_xFx_R},
\begin{eqnarray*}
& & A \llangle x_{\nu_{1}}x_{m-1} \rrangle_{\star\star} + A^{-1} \llangle x_{\nu_{1}}P_{-m-\nu_{2}}x_{-\nu_{2}-1} \rrangle_{\star\star} \\
&=& A \llangle x_{\nu_{1}}x_{m-1} \rrangle_{\star\star} - A^{2} \llangle x_{\nu_{1}}(A^{-1}F_{m+\nu_{2}-1} - A^{-2}F_{m+\nu_{2}})x_{-\nu_{2}} \rrangle_{\star\star} \\
&=& A \llangle R_{m-1-\nu_{1}} \rrangle_{\star\star} - A^{2} (A^{-1} \llangle R_{m-1-\nu_{1}} \rrangle_{\star\star} - A^{-2} \llangle R_{m-\nu_{1}} \rrangle_{\star\star}) \\
&=& \llangle R_{m-\nu_{1}} \rrangle_{\star\star} = \llangle x_{\nu_{1}}x_{m} \rrangle_{\star\star}
\end{eqnarray*}
which proves the case $\varepsilon = 1$.
\end{proof}

Let $D$ be an arrow diagram on $\hat{\bf S}^{2}$, define
\begin{equation*}
\phi_{\nu_{1},\nu_{2}}(D) = \llangle \llangle \llangle D \rrangle \rrangle_{\Gamma} \rrangle_{\star\star} = \langle \phi_{\beta_{1}}(D)\rangle_{\star\star}.  
\end{equation*}

\begin{lemma}
\label{lem:lemma_for_h6}
If $\nu_{0} \neq -1$, then
\begin{equation*}
\phi_{\nu_{1},\nu_{2}}(D-D') = 0
\end{equation*}
whenever arrow diagrams $D,D'$ in ${\bf S}^{2}$ are related by $\Omega_{1} - \Omega_{5}$, $S_{\beta_{1}}$, and $S_{\beta_{2}}$-moves, i.e., $\phi_{\nu_{1},\nu_{2}}$ is a well-defined homomorphism of free $R$-modules $R\mathcal{D}(\hat{\bf S}^{2})$ and $R\Sigma''_{\nu_{1},\nu_{2}}$.
\end{lemma}

\begin{proof}
As it was mentioned in Section~\ref{s:Summary}, for arrow diagrams $D$ and $D'$ which are related by $\Omega_{1} - \Omega_{5}$ and $S_{\beta_{1}}$-moves on $\hat{\bf S}^{2}$, 
\begin{equation*}
\phi_{\nu_{1},\nu_{2}}(D-D') = \langle \phi_{\beta_{1}}(D-D') \rangle_{\star\star} = 0.
\end{equation*}
Therefore, it suffices to show that $\phi_{\nu_{1},\nu_{2}}(D-D') = 0$ when $D,D'$ are related by $S_{\beta_{2}}$-move. Let $D$ and $D'$ be arrow diagrams in $\hat{\bf S}^{2}$ related by an $S_{\beta_{2}}$-move in a $2$-disk $\hat{\bf S}^{2}$ centered at $\beta_{2}$ (see right of Figure~\ref{fig:SBeta1Beta2MovesOnhatS2}). We denote by $\mathcal{K}(D)$ and $\mathcal{K}(D')$ their corresponding sets of Kauffman states. As shown in Figure~\ref{fig:CasesDiagramsBySBetaMoveOnHatS2} Kauffman states $s\in \mathcal{K}(D)$ are in bijection with pairs of Kauffman states $s_{+},s_{-}\in \mathcal{K}(D')$. Moreover, $s$ and $s_{+},s_{-}$ are related as follows
\begin{equation*}
p(s_{+})-n(s_{+}) = p(s)-n(s)+1 \quad \text{and} \quad p(s_{-})-n(s_{-}) = p(s)-n(s)-1,
\end{equation*}
and we denote by $D_{s}$, $D_{s_{+}}$, and $D_{s_{-}}$ the arrow diagrams corresponding $s$ and $s_{+},s_{-}$, respectively.
\begin{figure}[ht]
\centering
\includegraphics[scale=0.85]{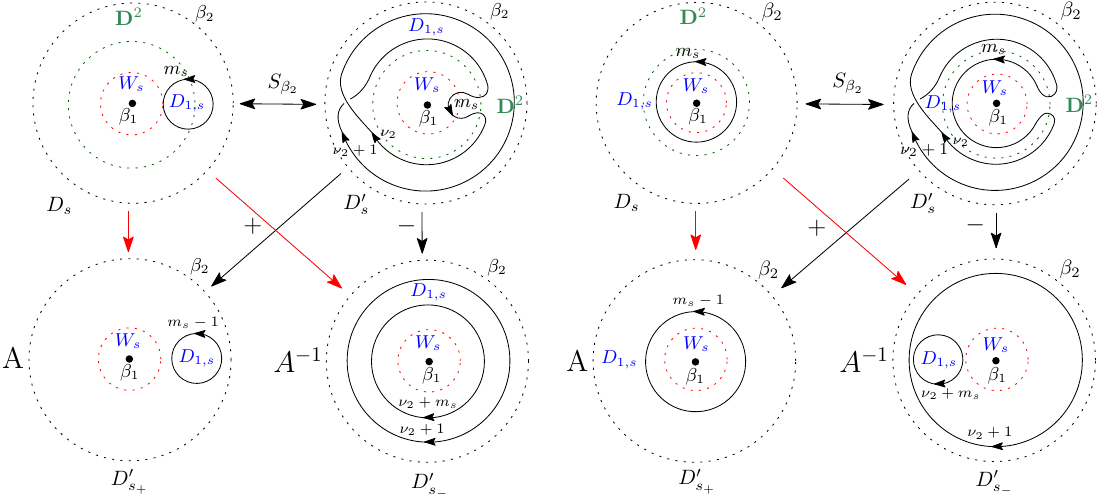}
\caption{$D_{s}$ and $D'_{s}$ related by an $S_{\beta_{2}}$-move on $\hat{\bf S}^{2}$}
\label{fig:CasesDiagramsBySBetaMoveOnHatS2}
\end{figure}
Therefore,
\begin{equation*}
\llangle D-D' \rrangle = \sum_{s\in \mathcal{K}(D)}A^{p(s)-n(s)}(\langle D_{s}\rangle - A\langle D'_{s+}\rangle -A^{-1}\langle D'_{s-}\rangle). 
\end{equation*}
For $D_{1,s}$ and $W_{s}$ in Figure~\ref{fig:CasesDiagramsBySBetaMoveOnHatS2}, let
\begin{equation*}
\langle D_{1,s} \rangle_{r} = \sum_{i=0}^{n_{s}}r_{s,i}^{(1)}\lambda^{i}\,\,\text{and}\,\,\llangle \llangle W_{s} \rrangle \rrangle_{\Gamma} =  \sum_{j = 0}^{k_{s}}r_{s,j}^{(2)} w_{j}(s).   
\end{equation*}
Thus, for the arrow diagrams on the left of Figure~\ref{fig:CasesDiagramsBySBetaMoveOnHatS2}
\begin{eqnarray*}
&&\llangle \langle D_{s}\rangle - A\langle D'_{s_{+}}\rangle-A^{-1}\langle D'_{s_{-}}\rangle \rrangle_{\Gamma}\\
&=&\sum_{i = 0}^{n_{s}}\sum_{j = 0}^{k_{s}} r_{s,i}^{(1)}r_{s,j}^{(2)} w_{j}(s)( P_{m_{s},i} - AP_{m_{s}-1,i}-A^{-1}x_{-\nu_{2}-m_{s}}\lambda^{i}x_{-\nu_{2}-1})
\end{eqnarray*}
and for the arrow diagrams on the right of Figure~\ref{fig:CasesDiagramsBySBetaMoveOnHatS2}
\begin{eqnarray*}
&&\llangle \langle D_{s}\rangle - A\langle D'_{s_{+}}\rangle-A^{-1}\langle D'_{s_{-}}\rangle \rrangle_{\Gamma}\\
&=&\sum_{i = 0}^{n_{s}}\sum_{j = 0}^{k_{s}} r_{s,i}^{(1)}r_{s,j}^{(2)} w_{j}(s)(x_{m_{s}}\lambda^{i} - Ax_{m_{s}-1}\lambda^{i}-A^{-1}P_{-\nu_{2}-m_{s},i}x_{-\nu_{2}-1}).
\end{eqnarray*}
Since for each $j = 0,1,\ldots,k_{s}$,
\begin{equation*}
\llangle w_{j}(s) \rrangle_{\Sigma^{\prime}_{\nu_{1}}} = \sum_{\varepsilon \in \{0,1\}}\sum_{k=0}^{l_{s,j}}r_{s,j,\varepsilon,k}^{(3)}(x_{\nu_{1}})^{\varepsilon}\lambda^{k}.
\end{equation*}
Therefore, for the arrow diagrams on the left of Figure~\ref{fig:CasesDiagramsBySBetaMoveOnHatS2},
\begin{eqnarray*}
&&\llangle \llangle \langle D_{s}\rangle - A\langle D'_{s_{+}}\rangle-A^{-1}\langle D'_{s_{-}}\rangle \rrangle_{\Gamma} \rrangle_{\Sigma^{\prime}_{\nu_{1}}}\\
&=& \sum_{i = 0}^{n_{s}}\sum_{j = 0}^{k_{s}}\sum_{\varepsilon \in \{0,1\}}\sum_{k=0}^{l_{s,j}} r_{s,i}^{(1)}r_{s,j}^{(2)}r_{s,j,\varepsilon,k}^{(3)}\llangle (x_{\nu_{1}})^{\varepsilon}\lambda^{k}( P_{m_{s},i} - AP_{m_{s}-1,i}-A^{-1}x_{-\nu_{2}-m_{s}}\lambda^{i}x_{-\nu_{2}-1})\rrangle_{\Sigma^{\prime}_{\nu_{1}}}\\
\end{eqnarray*}
and for the arrow diagrams on the right of Figure~\ref{fig:CasesDiagramsBySBetaMoveOnHatS2},
\begin{eqnarray*}
&&\llangle \llangle \langle D_{s}\rangle - A\langle D'_{s_{+}}\rangle-A^{-1}\langle D'_{s_{-}}\rangle \rrangle_{\Gamma} \rrangle_{\Sigma^{\prime}_{\nu_{1}}}\\
&=& \sum_{i = 0}^{n_{s}}\sum_{j = 0}^{k_{s}}\sum_{\varepsilon \in \{0,1\}}\sum_{k=0}^{l_{s,j}} r_{s,i}^{(1)}r_{s,j}^{(2)}r_{s,j,\varepsilon,k}^{(3)}\llangle (x_{\nu_{1}})^{\varepsilon}\lambda^{k}(x_{m_{s}}\lambda^{i} - Ax_{m_{s}-1}\lambda^{i}-A^{-1}P_{-\nu_{2}-m_{s},i}x_{-\nu_{2}-1})\rrangle_{\Sigma^{\prime}_{\nu_{1}}}.    
\end{eqnarray*}
Since
\begin{eqnarray*}
\phi_{\nu_{1},\nu_{2}}(D-D')=\llangle \llangle \langle D_{s}\rangle - A\langle D'_{s_{+}}\rangle-A^{-1}\langle D'_{s_{-}}\rangle \rrangle_{\Gamma} \rrangle_{\star\star} = \langle \llangle \llangle \langle D_{s}\rangle - A\langle D'_{s_{+}}\rangle-A^{-1}\langle D'_{s_{-}}\rangle\rrangle_{\Gamma}\rrangle_{\Sigma^{\prime}_{\nu_{1}}}\rangle_{\star\star},
\end{eqnarray*}
it suffices to show that
\begin{eqnarray*}
&&\llangle (x_{\nu_{1}})^{\varepsilon}\lambda^{k}( P_{m_{s},i} - AP_{m_{s}-1,i}-A^{-1}x_{-\nu_{2}-m_{s}}\lambda^{i}x_{-\nu_{2}-1})\rrangle_{\star\star} = 0 \quad \text{and} \\
&&\llangle (x_{\nu_{1}})^{\varepsilon}\lambda^{k}(x_{m_{s}}\lambda^{i} - Ax_{m_{s}-1}\lambda^{i}-A^{-1}P_{-\nu_{2}-m_{s},i}x_{-\nu_{2}-1})\rrangle_{\star\star} = 0.
\end{eqnarray*}
However, the above identities follow from Lemma~\ref{lem:Relations_Bracket_Two_Fibers_I} and Lemma~\ref{lem:Relations_Bracket_Two_Fibers_II}, respectively.
\end{proof}

We summarize our results from this subsection as Theorem~\ref{thm:KBSMForM2Beta1Beta2}.

\begin{theorem}
\label{thm:KBSMForM2Beta1Beta2}
For $\beta_{1}+\beta_{2} \neq 0$ the KBSM of $M_{2}(\beta_{1},\beta_{2})$ is a free $R$-module of rank $|\beta_{1}+\beta_{2}|+1$ and its basis consists of equivalence classes of generic framed links in $M_{2}(\beta_{1},\beta_{2})$ whose arrow diagrams are in $\Sigma''_{\nu_{1},\nu_{2}}$, i.e.,
\begin{equation*}
S_{2,\infty}(M_{2}(\beta_{1},\beta_{2});R,A)) \cong R\Sigma''_{\nu_{1},\nu_{2}}.
\end{equation*}
\end{theorem}

\begin{proof}
The statement follows by arguments analogous to those in our proof of Theorem~\ref{thm:lens_space}. Specifically, by Lemma~\ref{lem:lemma_for_h6}, the homomorphism of $R$-modules
\begin{equation*}
\phi_{\nu_{1},\nu_{2}}: R\mathcal{D}(\hat{\bf S}^{2}) \to  R\Sigma''_{\nu_{1},\nu_{2}},\,\,\phi_{\nu_{1},\nu_{2}}(D) = \llangle \llangle \llangle D \rrangle \rrangle_{\Gamma} \rrangle_{\star\star} = \langle \phi_{\beta_{1}}(D)\rangle_{\star\star}
\end{equation*}
descends to an isomorphism of free $R$-modules
\begin{equation*}
\hat{\phi}_{\nu_{1},\nu_{2}} : S\mathcal{D}_{\nu_{1},\nu_{2}} \to  R\Sigma''_{\nu_{1},\nu_{2}},\,\,\hat{\phi}_{\nu_{1},\nu_{2}}(D) = \phi_{\nu_{1},\nu_{2}}(D)
\end{equation*}
and then we apply Theorem~\ref{thm:AmbientIsotopiesInMBeta1Beta2}.
\end{proof}

\subsection{KBSM of \texorpdfstring{$M_{2}(\beta_{1},\beta_{2})$ with $\nu_{0} = -1$}{M\unichar{8322}(\unichar{0946}\unichar{8321}, \unichar{0946}\unichar{8322}) with \unichar{0957}\unichar{8320} = -1}}
\label{s:CaseWithTwoFibers_NuEqualMinus_1}

In this section, we find a new generating set for the KBSM of $L(0,1) = {\bf S}^{2}\times S^{1}$. It was proved in \cite{HP1993}  (see Theorem~4) that
\begin{equation}
\label{eqn:kbsm_S2S1}
\mathcal{S}_{2,\infty}(L(0,1);R,A) \cong R\oplus \bigoplus_{i = 1}^{\infty}\frac{R}{(1-A^{2i+4})}.
\end{equation}
A different proof of this result was given in \cite{M2011b} (see Theorem~3). Our proof of \eqref{eqn:kbsm_S2S1} differs from those in \cite{HP1993} and \cite{M2011b} since, in particular, we use $M_{2}(\beta_{1},\beta_{2})$ with $\beta_{1} + \beta_{2} = 0$ as our model for $L(0,1)$. 

\medskip

As noted in \cite{DW2025}, ambient isotopy classes of generic framed links in  $(\beta_{1},2)$-fibered torus $V(\beta_{1},2)$ are in bijection with equivalence classes $\mathcal{D}({\bf D}^{2}_{\beta_{1}})$ of arrow diagrams (including the empty diagram) on a $2$-disk ${\bf D}^{2}_{\beta_{1}}$ with marked point $\beta_{1}$, modulo $\Omega_{1}-\Omega_{5}$ and $S_{\beta_{1}}$-moves. Since an embedding
\begin{equation*}
i: V(\beta_{1},2) \to M_{2}(\beta_{1},\beta_{2}),\,i(L) = L,
\end{equation*}
induces corresponding epimorphism of $R$-modules
\begin{equation*}
i_{*}: S\mathcal{D}({\bf D}^{2}_{\beta_{1}})\to S\mathcal{D}_{\nu_{1},\nu_{2}},\,i_{*}([D]) = [\![D]\!],
\end{equation*}
it follows that
\begin{equation*}
S\mathcal{D}({\bf D}^{2}_{\beta_{1}})/\ker(i_{*})\cong S\mathcal{D}_{\nu_{1},\nu_{2}}.
\end{equation*}
As it was shown in \cite{DW2025}, $S\mathcal{D}({\bf D}^{2}_{\beta_{1}}) \cong R\Sigma^{\prime}_{\nu_{1}}$ and, using arguments as in Lemma~\ref{lem:lemma_for_h6}, we see that $\ker(i_{*})$ is generated by:
\begin{eqnarray*}
&&(x_{\nu_{1}})^{\varepsilon} \lambda^{n_{1}} P_{m,n_{2}} - A (x_{\nu_{1}})^{\varepsilon} \lambda^{n_{1}} P_{m-1,n_{2}} - A^{-1} (x_{\nu_{1}})^{\varepsilon} \lambda^{n_{1}} x_{-m-\nu_{2}} \lambda^{n_{2}} x_{-\nu_{2}-1}\quad \text{and}\notag\\
&&(x_{\nu_{1}})^{\varepsilon} \lambda^{n_{1}} x_{m} \lambda^{n_{2}} - A (x_{\nu_{1}})^{\varepsilon} \lambda^{n_{1}} x_{m-1} \lambda^{n_{2}} - A^{-1} (x_{\nu_{1}})^{\varepsilon} \lambda^{n_{1}} P_{-m-\nu_{2},n_{2}} x_{-\nu_{2}-1},
\end{eqnarray*}
where $\varepsilon \in\{0,1\}$, $n_{1},n_{2}\geq 0$, and $m\in \mathbb{Z}$. 

Let $S_{\nu_{2}}({\bf D}^{2}_{\beta_{1}})$ denote the $R$-submodule of $S\mathcal{D}({\bf D}^{2}_{\beta_{1}})$ generated by
\begin{equation*}
F_{m}x_{-\nu_{2}} - x_{\nu_{1}} F_{-1-m} \,\, \text{and} \,\,
x_{\nu_{1}}F_{m}x_{-\nu_{2}} - R_{m+1},
\end{equation*}
for $m \in \mathbb{Z}$ (see Lemma~\ref{lem:rel_F_and_FR_TwoFibers}). We start by showing that 
\begin{equation*}
\ker(i_{*}) = S_{\nu_{2}}({\bf D}^{2}_{\beta_{1}})
\end{equation*}
and then we compute $S\mathcal{D}({\bf D}^{2}_{\beta_{1}})/S_{\nu_{2}}({\bf D}^{2}_{\beta_{1}})$. 

\begin{lemma}
\label{lem:Relations_Bracket_Two_Fibers_nu2}
For any $\varepsilon \in \{0,1\}$ and $m \in \mathbb{Z}$,
\begin{equation*}
(x_{\nu_{1}})^{\varepsilon} F_{m} x_{-\nu_{2}-1} + A^{3} (x_{\nu_{1}})^{\varepsilon} F_{m} x_{-\nu_{2}} \in S_{\nu_{2}}({\bf D}^{2}_{\beta_{1}}).
\end{equation*}
In particular, for any $\varepsilon \in \{0,1\}$ and $n \geq 0$,
\begin{equation*}
(x_{\nu_{1}})^{\varepsilon} \lambda^{n} x_{-\nu_{2}-1} + A^{3} (x_{\nu_{1}})^{\varepsilon} \lambda^{n} x_{-\nu_{2}} \in S_{\nu_{2}}({\bf D}^{2}_{\beta_{1}}).
\end{equation*}
\end{lemma}

\begin{proof}
Applying Kauffman bracket skein relation to arrow diagrams in Figure~\ref{fig:Omega_5ForTwoFibersNegativeNu} we see that
\begin{equation*}
P_{-m}x_{-\nu_{2}} = A^{-2}P_{-m+1}x_{-\nu_{2}-1} + x_{m-\nu_{2}-2} - A^{-2}x_{m-\nu_{2}}.
\end{equation*}
Furthermore, using \eqref{eqn:rel_F_to_P} and \eqref{eqn:kbsm_xm}, we see that
\begin{equation*}
(A^{-1}F_{m-1}-A^{-2}F_{m})x_{-\nu_{2}} = A^{-2}(A^{-1}F_{m-2}-A^{-2}F_{m-1})x_{-\nu_{2}-1} + x_{\nu_{1}}F_{-m+1} - A^{-2}x_{\nu_{1}}F_{-m-1}
\end{equation*}
or equivalently
\begin{eqnarray*}
& & A^{-3}(F_{m-2}x_{-\nu_{2}-1} + A^{3}F_{m-2}x_{-\nu_{2}}) - A^{-4}(F_{m-1}x_{-\nu_{2}-1} + A^{3}F_{m-1}x_{-\nu_{2}}) \\
&=& (F_{m-2}x_{-\nu_{2}} - x_{\nu_{1}}F_{-m+1}) - A^{-2}(F_{m}x_{-\nu_{2}} - x_{\nu_{1}}F_{-m-1}) \in S_{\nu_{2}}({\bf D}^{2}_{\beta_{1}}).
\end{eqnarray*}
Since $\nu_{0} = -1$, $F_{0} = 1$ and $F_{-1} = -A^{3}$, one can see that
\begin{equation*}
F_{0}x_{-\nu_{2}-1} + A^{3}F_{0}x_{-\nu_{2}} = x_{\nu_{1}} + A^{3}x_{-\nu_{2}} = -(F_{-1}x_{-\nu_{2}} - x_{\nu_{1}}F_{0}) \in S_{\nu_{2}}({\bf D}^{2}_{\beta_{1}}).
\end{equation*}
Therefore, by induction on $m$, we conclude that
\begin{equation*}
F_{m}x_{-\nu_{2}-1} + A^{3}F_{m}x_{-\nu_{2}} \in S_{\nu_{2}}({\bf D}^{2}_{\beta_{1}})
\end{equation*}
for any $m \in \mathbb{Z}$, which proves the case $\varepsilon = 0$.

Applying Kauffman bracket skein relation to arrow diagrams in Figure~\ref{fig:Omega_5ForTwoFibersNegativeNu} we see that
\begin{equation*}
x_{\nu_{1}}P_{-m}x_{-\nu_{2}} = A^{-2}x_{\nu_{1}}P_{-m+1}x_{-\nu_{2}-1} + x_{\nu_{1}}x_{m-\nu_{2}-2} - A^{-2}x_{\nu_{1}}x_{m-\nu_{2}}.
\end{equation*}
Therefore, using \eqref{eqn:rel_F_to_P} and \eqref{eqn:kbsm_xvxm} we see that
\begin{equation*}
x_{\nu_{1}}(A^{-1}F_{m-1}-A^{-2}F_{m})x_{-\nu_{2}} = A^{-2}x_{\nu_{1}}(A^{-1}F_{m-2}-A^{-2}F_{m-1})x_{-\nu_{2}-1} + R_{m-1} - A^{-2}R_{m+1}
\end{equation*}
or equivalently
\begin{eqnarray*}
& & A^{-3}x_{\nu_{1}}(F_{m-2}x_{-\nu_{2}-1} + A^{3}F_{m-2}x_{-\nu_{2}}) - A^{-4}x_{\nu_{1}}(F_{m-1}x_{-\nu_{2}-1} + A^{3}F_{m-1}x_{-\nu_{2}}) \\
&=& (x_{\nu_{1}}F_{m-2}x_{-\nu_{2}} - R_{m-1}) - A^{-2}(x_{\nu_{1}}F_{m}x_{-\nu_{2}} - R_{m+1}) \in S_{\nu_{2}}({\bf D}^{2}_{\beta_{1}}).
\end{eqnarray*}
Since $\nu_{0} = -1$, $F_{0} = 1$, $F_{-1} = -A^{3}$, and $x_{\nu_{1}}x_{\nu_{1}} = R_{0}$ by \eqref{eqn:kbsm_xvxm},
one sees that
\begin{equation*}
x_{\nu_{1}}F_{0}x_{-\nu_{2}-1} + A^{3}x_{\nu_{1}}F_{0}x_{-\nu_{2}} = x_{\nu_{1}}x_{\nu_{1}} + A^{3}x_{\nu_{1}}x_{-\nu_{2}} = -(x_{\nu_{1}}F_{-1}x_{-\nu_{2}} -R_{0}) \in S_{\nu_{2}}({\bf D}^{2}_{\beta_{1}}).
\end{equation*}
Therefore, by induction on $m$, we see that
\begin{equation*}
x_{\nu_{1}}F_{m}x_{-\nu_{2}-1} + A^{3}x_{\nu_{1}}F_{m}x_{-\nu_{2}} \in S_{\nu_{2}}({\bf D}^{2}_{\beta_{1}})
\end{equation*}
for any $m \in \mathbb{Z}$, which proves the case $\varepsilon = 1$.
\end{proof}

\begin{lemma}
\label{lem:Tm}
Let $T_{m}(n_{1},n_{2})$ be a family of elements of $S\mathcal{D}({\bf D}^{2}_{\beta_{1}})$, $m\in \mathbb{Z}$, $n_{1},n_{2}\geq 0$. Assume that $T_{m}(n_{1},n_{2})$ satisfies conditions:
\begin{equation*}
T_{m}(n_{1}+1,n_{2}) = A^{-1}T_{m-1}(n_{1},n_{2}) + AT_{m+1}(n_{1},n_{2}),
\end{equation*}
\begin{equation*}
T_{m}(n_{1},n_{2}+1) = AT_{m-1}(n_{1},n_{2}) + A^{-1}T_{m+1}(n_{1},n_{2}),
\end{equation*}
and $T_{m}(0,0) \in S_{\nu_{2}}({\bf D}^{2}_{\beta_{1}})$ for all $m \in \mathbb{Z}$. Then $
T_{m}(n_{1},n_{2}) \in S_{\nu_{2}}({\bf D}^{2}_{\beta_{1}})$ for all $m \in \mathbb{Z}$ and $n_{1},n_{2} \geq 0$.
\end{lemma}

\begin{proof}
As one may show 
\begin{eqnarray*}
T_{m}(n_{1},n_{2}) 
&=& \sum_{i=0}^{n_{1}} A^{n_{1}-2i} \binom{n_{1}}{i} T_{m+n_{1}-2i}(0,n_{2}) \\
&=& \sum_{i=0}^{n_{1}} \sum_{j=0}^{n_{2}} A^{n_{1}-2i+n_{2}-2j} \binom{n_{1}}{i} \binom{n_{2}}{j} T_{m+n_{1}-2i-n_{2}+2j}(0,0).
\end{eqnarray*}    
Since $T_{m}(0,0) \in S_{\nu_{2}}({\bf D}^{2}_{\beta_{1}})$, for all $m \in \mathbb{Z}$, our statement follows.
\end{proof}

\begin{lemma}
\label{lem:Relations_Bracket_Two_Fibers_v0Minus1_I}
For any $\varepsilon \in \{0,1\}$, $m \in \mathbb{Z}$, and $n_{1},n_{2} \geq 0$,
\begin{equation*}
(x_{\nu_{1}})^{\varepsilon} \lambda^{n_{1}} P_{m,n_{2}} - A (x_{\nu_{1}})^{\varepsilon} \lambda^{n_{1}} P_{m-1,n_{2}} - A^{-1} (x_{\nu_{1}})^{\varepsilon} \lambda^{n_{1}} x_{-m-\nu_{2}} \lambda^{n_{2}} x_{-\nu_{2}-1} \in S_{\nu_{2}}({\bf D}^{2}_{\beta_{1}}).
\end{equation*}
\end{lemma}

\begin{proof}
For $\varepsilon = 0$ with $n_{1} = n_{2} = 0$:
\begin{eqnarray*}
& & P_{m} - A P_{m-1} - A^{-1} x_{-m-\nu_{2}}x_{-\nu_{2}-1} = P_{m} - A P_{m-1} - A^{-1} x_{\nu_{1}}F_{m-1}x_{-\nu_{2}-1} \\
&=& A^{2} (x_{\nu_{1}}F_{m-1}x_{-\nu_{2}} - R_{m}) - A^{-1}(x_{\nu_{1}}F_{m-1}x_{-\nu_{2}-1} + A^{3}x_{\nu_{1}}F_{m-1}x_{-\nu_{2}}) \in S_{\nu_{2}}({\bf D}^{2}_{\beta_{1}})
\end{eqnarray*}
by \eqref{eqn:kbsm_xm} and Lemma~\ref{lem:Relations_Bracket_Two_Fibers_nu2}.

For $\varepsilon = 1$  with $n_{1} = n_{2} = 0$:
\begin{eqnarray*}
& & x_{\nu_{1}}P_{m} - A x_{\nu_{1}}P_{m-1} - A^{-1} x_{\nu_{1}}x_{-m-\nu_{2}}x_{-\nu_{2}-1} \\
&=& x_{\nu_{1}}P_{m} - A x_{\nu_{1}}P_{m-1} - A^{-1} R_{-m+1}x_{-\nu_{2}-1} \\
&=& x_{\nu_{1}}P_{m} - A x_{\nu_{1}}P_{m-1} + A^{2} (A^{-1}P_{-m} - A^{-2}P_{-m+1})x_{-\nu_{2}} - A^{-1} (R_{-m+1}x_{-\nu_{2}-1} + A^{3}R_{-m+1}x_{-\nu_{2}}) \\
&=& x_{\nu_{1}}(A^{-1}F_{-m-1} - A^{-2}F_{-m}) - A x_{\nu_{1}}(-A^{-2}F_{-m+1} + A^{-1}F_{-m}) \\
&+& A^{2} (-A^{-3}F_{m} + A^{-2}F_{m-1} + A^{-4}F_{m-1} - A^{-3}F_{m-2})x_{-\nu_{2}} - A^{-1} (R_{-m+1}x_{-\nu_{2}-1} + A^{3}R_{-m+1}x_{-\nu_{2}}) \\
&=& - A^{-1}(F_{m}x_{-\nu_{2}} - x_{\nu_{1}}F_{-m-1}) - A^{-1}(F_{m-2}x_{-\nu_{2}} - x_{\nu_{1}}F_{-m+1}) + (F_{m-1}x_{-\nu_{2}} - x_{\nu_{1}}F_{-m}) \\
&+& A^{-2}(F_{m-1}x_{-\nu_{2}} - x_{\nu_{1}}F_{-m}) - A^{-1} (R_{-m+1}x_{-\nu_{2}-1} + A^{3}R_{-m+1}x_{-\nu_{2}}) \in S_{\nu_{2}}({\bf D}^{2}_{\beta_{1}})
\end{eqnarray*}
by \eqref{eqn:kbsm_xvxm}, \eqref{eqn:rel_F_to_P}, and Lemma~\ref{lem:Relations_Bracket_Two_Fibers_nu2}. Let
\begin{equation*}
T_{m}(n_{2},n_{1}) = (x_{\nu_{1}})^{\varepsilon} \lambda^{n_{1}} P_{m,n_{2}} - A (x_{\nu_{1}})^{\varepsilon} \lambda^{n_{1}} P_{m-1,n_{2}} - A^{-1} (x_{\nu_{1}})^{\varepsilon} \lambda^{n_{1}} x_{-m-\nu_{2}} \lambda^{n_{2}} x_{-\nu_{2}-1}.
\end{equation*}
Since by definition of $P_{m}$ and $P_{m,k}$, and Lemma~\ref{lem:annulus_for_any_kn_bracket_2_beta0}, 
\begin{eqnarray*}
P_{m,k} &=& A P_{m+1,k-1} + A^{-1} P_{m-1,k-1},\\
\lambda P_{m} &=& A^{-1} P_{m+1}  + A P_{m-1},\\
\lambda x_{m} &=& A^{-1} x_{m-1}  + A x_{m+1},\\
x_{m} \lambda &=& A x_{m-1}  + A^{-1} x_{m+1},
\end{eqnarray*}
as one may verify:
\begin{equation*}
T_{m}(n_{2}+1,n_{1}) = A^{-1}T_{m-1}(n_{2},n_{1}) + AT_{m+1}(n_{2},n_{1}),
\end{equation*}
\begin{equation*}
T_{m}(n_{2},n_{1}+1) = AT_{m-1}(n_{2},n_{1}) + A^{-1}T_{m+1}(n_{2},n_{1}),
\end{equation*}
and as we showed $T_{m}(0,0) \in S_{\nu_{2}}({\bf D}^{2}_{\beta_{1}})$. Therefore, statement of Lemma~\ref{lem:Relations_Bracket_Two_Fibers_v0Minus1_I} follows by Lemma~\ref{lem:Tm}.
\end{proof}

\begin{lemma}
\label{lem:Relations_Bracket_Two_Fibers_v0Minus1_II}
For any $\varepsilon \in \{0,1\}$, $m \in \mathbb{Z}$, and $n_{1},n_{2} \geq 0$,
\begin{equation*}
(x_{\nu_{1}})^{\varepsilon} \lambda^{n_{1}} x_{m} \lambda^{n_{2}} - A (x_{\nu_{1}})^{\varepsilon} \lambda^{n_{1}} x_{m-1} \lambda^{n_{2}} - A^{-1} (x_{\nu_{1}})^{\varepsilon} \lambda^{n_{1}} P_{-m-\nu_{2},n_{2}} x_{-\nu_{2}-1} \in S_{\nu_{2}}({\bf D}^{2}_{\beta_{1}}).
\end{equation*}
\end{lemma}

\begin{proof}
For $\varepsilon = 0$:
\begin{eqnarray*}
& & x_{m} - A x_{m-1} - A^{-1} P_{-m-\nu_{2}}x_{-\nu_{2}-1} \\
&=& x_{\nu_{1}}F_{\nu_{1}-m} - A x_{\nu_{1}}F_{\nu_{1}-m+1} + A^{2} (A^{-1}F_{m+\nu_{2}-1} - A^{-2}F_{m+\nu_{2}})x_{-\nu_{2}} \\
&-& A^{-1}(P_{-m-\nu_{2}}x_{-\nu_{2}-1} + A^{3}P_{-m-\nu_{2}}x_{-\nu_{2}}) \\
&=& -(F_{m+\nu_{2}}x_{-\nu_{2}} - x_{\nu_{1}}F_{\nu_{1}-m}) + A (F_{m+\nu_{2}-1}x_{-\nu_{2}} - x_{\nu_{1}}F_{\nu_{1}-m+1}) \\
&-& A^{-1}(P_{-m-\nu_{2}}x_{-\nu_{2}-1} + A^{3}P_{-m-\nu_{2}}x_{-\nu_{2}}) \in S_{\nu_{2}}({\bf D}^{2}_{\beta_{1}})
\end{eqnarray*}
by \eqref{eqn:kbsm_xm}, \eqref{eqn:rel_F_to_P}, and Lemma~\ref{lem:Relations_Bracket_Two_Fibers_nu2}.

For $\varepsilon = 1$:
\begin{eqnarray*}
& & x_{\nu_{1}}x_{m} - A x_{\nu_{1}}x_{m-1} - A^{-1} x_{\nu_{1}}P_{-m-\nu_{2}}x_{-\nu_{2}-1} \\
&=& R_{m-\nu_{1}} - A R_{m-1-\nu_{1}} + A^{2} x_{\nu_{1}}(A^{-1}F_{m+\nu_{2}-1} - A^{-2}F_{m+\nu_{2}})x_{-\nu_{2}} \\
&-& A^{-1} (x_{\nu_{1}}P_{-m-\nu_{2}}x_{-\nu_{2}-1} + A^{3}x_{\nu_{1}}P_{-m-\nu_{2}}x_{-\nu_{2}}) \\
&=& - (x_{\nu_{1}}F_{m+\nu_{2}}x_{-\nu_{2}} - R_{m-\nu_{1}}) + A (x_{\nu_{1}}F_{m+\nu_{2}-1}x_{-\nu_{2}} - R_{m-1-\nu_{1}}) \\
&-& A^{-1} (x_{\nu_{1}}P_{-m-\nu_{2}}x_{-\nu_{2}-1} + A^{3}x_{\nu_{1}}P_{-m-\nu_{2}}x_{-\nu_{2}}) \in S_{\nu_{2}}({\bf D}^{2}_{\beta_{1}})
\end{eqnarray*}
by \eqref{eqn:kbsm_xvxm}, \eqref{eqn:rel_F_to_P}, and Lemma~\ref{lem:Relations_Bracket_Two_Fibers_nu2}. Furthermore, taking
\begin{equation*}
T_{m}(n_{1},n_{2}) = (x_{\nu_{1}})^{\varepsilon} \lambda^{n_{1}} x_{m} \lambda^{n_{2}} - A (x_{\nu_{1}})^{\varepsilon} \lambda^{n_{1}} x_{m-1} \lambda^{n_{2}} - A^{-1} (x_{\nu_{1}})^{\varepsilon} \lambda^{n_{1}} P_{-m-\nu_{2},n_{2}} x_{-\nu_{2}-1},
\end{equation*}
as in our proof of Lemma~\ref{lem:Relations_Bracket_Two_Fibers_v0Minus1_I} using the definition of $P_{m}$, $P_{m,k}$, and Lemma~\ref{lem:annulus_for_any_kn_bracket_2_beta0}, one verifies that
\begin{equation*}
T_{m}(n_{1}+1,n_{2}) = A^{-1}T_{m-1}(n_{1},n_{2}) + AT_{m+1}(n_{1},n_{2}),
\end{equation*}
\begin{equation*}
T_{m}(n_{1},n_{2}+1) = AT_{m-1}(n_{1},n_{2}) + A^{-1}T_{m+1}(n_{1},n_{2}).
\end{equation*}
Furthermore, as we showed $T_{m}(0,0) \in S_{\nu_{2}}({\bf D}^{2}_{\beta_{1}})$, so the statement of Lemma~\ref{lem:Relations_Bracket_Two_Fibers_v0Minus1_II} follows by Lemma~\ref{lem:Tm}.
\end{proof}

\begin{corollary}
\label{cor:SubmodulesOfSD}
$\ker(i_{*}) = S_{\nu_{2}}({\bf D}^{2}_{\beta_{1}})$.
\end{corollary}

\begin{proof}
It follows from Lemma~\ref{lem:Relations_Bracket_Two_Fibers_v0Minus1_I} and Lemma~\ref{lem:Relations_Bracket_Two_Fibers_v0Minus1_II} that $\ker(i_{*}) \subseteq S_{\nu_{2}}({\bf D}^{2}_{\beta_{1}})$. As we showed in Lemma~\ref{lem:rel_F_and_FR_TwoFibers} that $F_{m}x_{-\nu_{2}} - x_{\nu_{1}} F_{-1-m} = 0$ and $x_{\nu_{1}}F_{m}x_{-\nu_{2}} - R_{m+1} = 0$ in $S\mathcal{D}_{\nu_{1},\nu_{2}} = S\mathcal{D}({\bf D}^{2}_{\beta_{1}})/\ker(i_{*})$, hence
\begin{equation*}
F_{m}x_{-\nu_{2}} - x_{\nu_{1}} F_{-1-m}, \,\, x_{\nu_{1}}F_{m}x_{-\nu_{2}} - R_{m+1} \in \ker(i_{*}).    
\end{equation*}
It follows that $S_{\nu_{2}}({\bf D}^{2}_{\beta_{1}})\subseteq \ker(i_{*})$.
\end{proof}

Since
\begin{equation*}
S\mathcal{D}({\bf D}^{2}_{\beta_{1}}) \cong R\Sigma^{\prime}_{\nu_{1}} \cong RX_{0}\oplus RX_{1},
\end{equation*}
where $X_{0} = \{\lambda^{n}\mid n\geq 0\}$ and $X_{1} = \{x_{\nu_{1}}\lambda^{n}\mid n\geq 0\}$, to compute $S\mathcal{D}({\bf D}^{2}_{\beta_{1}})/S_{\nu_{2}}({\bf D}^{2}_{\beta_{1}})$, we start by changing the basis of $RX_{0}\oplus RX_{1}$ and then we represent generators of $\Sigma^{\prime}_{\nu_{1}}$ in terms of this basis. 

For $m \geq 0$, let
\begin{equation*}
\varphi_{m} = Q_{m+1} - 2Q_{m} + 2Q_{m-1} - \cdots + 2(-1)^{m-1}Q_{2} + (-1)^{m}Q_{1}
\end{equation*}
and
\begin{equation*}
\psi_{m} = x_{\nu_{1}}(Q_{m+1} - Q_{m} + \cdots + (-1)^{m-1}Q_{2} + (-1)^{m}Q_{1}).
\end{equation*}
It is easy to check 
\begin{equation*}
RX_{0} = R\{\varphi_{m}\mid m\geq 0\} \quad \text{and} \quad RX_{1} = R\{ \psi_{m} \mid m\geq 0\}.    
\end{equation*}
Therefore,
\begin{equation*}
S\mathcal{D}({\bf D}^{2}_{\beta_{1}}) \cong R\Sigma^{\prime}_{\nu_{1}} \cong R\{\varphi_{m}\}_{m\geq 0}\oplus R\{\psi_{m}\}_{m\geq 0}.    
\end{equation*}

Let $q_{k} = A^{-k} - A^{k}$ and define $\{\Phi_{m}\}_{m \in \mathbb{Z}}$ and $\{\Psi_{m}\}_{m \in \mathbb{Z}}$ as follows:
\begin{equation*}
\Phi_{m} = q_{2m+2}\varphi_{m} \quad \text{and} \quad \Psi_{m} = q_{2m+1}\psi_{m-1} 
\end{equation*}
when $m \geq 1$, $\Phi_{0} = \Phi_{-1} = 0 = \Psi_{0} = \Psi_{-1}$, and
\begin{equation*}
\Phi_{m} = -\Phi_{-m-2} \quad \text{and} \quad \Psi_{m} = \Psi_{-m-1}
\end{equation*}
for $m \leq -2$. Let 
\begin{equation*}
S_{2}(\Phi\oplus\Psi) = R\{\Phi_{m}\}_{m\geq 1}\oplus R\{\Psi_{m}\}_{m\geq 1}.
\end{equation*}
be a free $R$-submodule of $R\Sigma^{\prime}_{\nu_{1}} \cong R\{\varphi_{m}\}_{m\geq 0}\oplus R\{\psi_{m}\}_{m\geq 0}$ with basis $\{\Phi_{m}\oplus\Psi_{k}\mid m,k\geq 1\}$.

\begin{lemma}
\label{lem:torsion_relations}
Suppose that $(u_{m})_{m\in \mathbb{Z}}$ is a sequence in $R$ which for all $m\in \mathbb{Z}$ satisfies the relation,
\begin{equation*}
u_{m+1} = zu_{m} - u_{m-1},
\end{equation*}
where $z = A^{-2} + A^{2}$. Let $(B_{m})_{m\in \mathbb{Z}}$ be a sequence in $S\mathcal{D}({\bf D}^{2}_{\beta_{1}})$ and for any $m > 0$, let
\begin{equation*}
S_{m} = u_{m+1}\sum_{i = 0}^{m-1}(-1)^{i}B_{m-i} 
\end{equation*}
and for $m \leq 0$, let
\begin{equation*}
S_{m} = u_{m+1}\sum_{i = 0}^{-m-1}(-1)^{i}B_{m+i+1}.
\end{equation*}
Then
\begin{equation}
\label{eqn:uB_to_S}
u_{m+1}B_{m} + u_{m-1}B_{m-1} = S_{m} + zS_{m-1} + S_{m-2}    
\end{equation}
for any $m \in \mathbb{Z}$. 
\end{lemma}

\begin{proof}
It is clear that \eqref{eqn:uB_to_S} holds for $m = 1$. For $m \geq 2$, we see that
\begin{equation*}
u_{m+1}B_{m} = S_{m} - u_{m+1}\sum_{i = 1}^{m-1}(-1)^{i}B_{m-i} = S_{m} - (zu_{m}-u_{m-1})\sum_{i = 1}^{m-1}(-1)^{i}B_{m-i}
\end{equation*}
and
\begin{equation*}
u_{m-1}B_{m-1} = u_{m-1}\sum_{i = 2}^{m-1}(-1)^{i}B_{m-i} - u_{m-1}\sum_{i = 1}^{m-1}(-1)^{i}B_{m-i}.
\end{equation*}
Therefore,
\begin{eqnarray*}
u_{m+1}B_{m} + u_{m-1}B_{m-1} &=& S_{m} + zu_{m} \sum_{i = 0}^{m-2}(-1)^{i}B_{m-1-i} + u_{m-1} \sum_{i = 0}^{m-3}(-1)^{i}B_{m-2-i} \\
&=& S_{m} + zS_{m-1} + S_{m-2}.
\end{eqnarray*}

Furthermore, for $m \leq 0$ we see that
\begin{equation*}
u_{m-1}B_{m-1} = S_{m-2} - u_{m-1}\sum_{i = 1}^{-m+1}(-1)^{i}B_{m-1+i} = S_{m-2} - (zu_{m}-u_{m+1})\sum_{i = 1}^{-m+1}(-1)^{i}B_{m-1+i}
\end{equation*}
and
\begin{equation*}
u_{m+1}B_{m} = u_{m+1}\sum_{i = 2}^{-m+1}(-1)^{i}B_{m-1+i} - u_{m+1}\sum_{i = 1}^{-m+1}(-1)^{i}B_{m-1+i}.
\end{equation*}
Therefore,
\begin{eqnarray*}
u_{m+1}B_{m} + u_{m-1}B_{m-1} &=& S_{m-2} + zu_{m} \sum_{i = 0}^{-m}(-1)^{i}B_{m+i} + u_{m+1} \sum_{i = 0}^{-m-1}(-1)^{i}B_{m+1+i} \\
&=& S_{m} + zS_{m-1} + S_{m-2}.
\end{eqnarray*}
Consequently, \eqref{eqn:uB_to_S} holds for any $m \in \mathbb{Z}$.
\end{proof}

\begin{lemma}
\label{lem:Relation_rel_xFx_R_to_Phi_TwoFibers}
In $S\mathcal{D}({\bf D}^{2}_{\beta_{1}})$, 
for all $m\in \mathbb{Z}$,
\begin{equation*}
x_{\nu_{1}} F_{m} x_{-\nu_{2}} - R_{m+1} = -A^{-m-1}(\Phi_{m} + (A^{-2}+A^{2})\Phi_{m-1} + \Phi_{m-2}).
\end{equation*}
\end{lemma}

\begin{proof}
We first show that 
\begin{equation}
\label{eqn:pf_Relation_rel_xFx_R_to_Phi_TwoFibers_0}
x_{\nu_{1}} F_{m} x_{-\nu_{2}} - R_{m+1} = -A^{-m-1}(q_{2m+2}(Q_{m+1} - Q_{m})+q_{2m-2}(Q_{m} - Q_{m-1}))
\end{equation}
for all $m \in \mathbb{Z}$. For $m = 0$, since $F_{0} = Q_{1} = 1$ and 
\begin{equation*}
x_{\nu_{1}} F_{m} x_{-\nu_{2}} = x_{\nu_{1}} F_{0} x_{-\nu_{2}} = R_{-\nu_{2}-\nu_{1}} =  R_{1},
\end{equation*}
it follows that
\begin{equation*}
x_{\nu_{1}} F_{m} x_{-\nu_{2}} - R_{m+1} = x_{\nu_{1}} F_{0} x_{-\nu_{2}} - R_{1} =  0.
\end{equation*}
Moreover, the right hand side of \eqref{eqn:pf_Relation_rel_xFx_R_to_Phi_TwoFibers_0} when $m = 0$ is
\begin{equation*}
-A^{-1}(q_{2}(Q_{1} - Q_{0})+q_{-2}(Q_{0} - Q_{-1})) = -A^{-1}(q_{2}+q_{-2}) = 0,
\end{equation*}
so \eqref{eqn:pf_Relation_rel_xFx_R_to_Phi_TwoFibers_0} holds for $m = 0$.

Assume that $m \geq 1$. Using \eqref{eqn:kbsm_xm}, \eqref{eqn:IdentityTwoFiberCase1}, and \eqref{eqn:kbsm_xvxm}, we see that
\begin{eqnarray}
x_{\nu_{1}} F_{m} x_{-\nu_{2}} &=& x_{\nu_{1}-m}x_{-\nu_{2}} = A^{-2m}x_{\nu_{1}}x_{-\nu_{2}-m} + \sum_{i=0}^{m-1}A^{-2i}(P_{-\nu_{0}+ m-2-2i}-A^{-2}P_{-\nu_{0} + m-2i}) \notag\\
&=& A^{-2m}R_{- m + 1} + \sum_{i = 0}^{m-1}A^{-2i}P_{m-1-2i}- \sum_{i = 0}^{m-1}A^{-2i-2}P_{m+1-2i}. \label{eqn:pf_Relation_rel_xFx_R_to_Phi_TwoFibers_1}
\end{eqnarray}
Since $P_{i} = -A^{i+2}Q_{i+1}+A^{i-2}Q_{i-1}$ (see \eqref{eqn:rel_Pn}), it follows that
\begin{eqnarray}
\sum_{i=0}^{m-1} A^{-2i}P_{m-1-2i} &=& - \sum_{i=0}^{m-1} A^{m+1-4i} Q_{m-2i} + \sum_{i=0}^{m-1} A^{m-3-4i} Q_{m-2-2i} \notag\\
&=& - A^{m+1} Q_{m} + A^{-3m+1} Q_{-m} \label{eqn:pf_Relation_rel_xFx_R_to_Phi_TwoFibers_2}
\end{eqnarray}
and consequently,
\begin{equation}
- \sum_{i=1}^{m} A^{-2i-2}P_{m+1-2i} = A^{m-3} Q_{m} - A^{-3m-3} Q_{-m}. \label{eqn:pf_Relation_rel_xFx_R_to_Phi_TwoFibers_3}
\end{equation}
Moreover, since by the definition $R_{j} = A^{-1}P_{j-1}-A^{-2}P_{j}$, it follows that
\begin{equation}
\label{eqn:pf_Relation_rel_xFx_R_to_Phi_TwoFibers_4}
A^{-2m}R_{-m+1} + A^{-2m-2}P_{-m+1} = A^{-2m-1}P_{-m} = -A^{-3m+1}Q_{-m+1} + A^{-3m-3}Q_{-m-1}
\end{equation}
and
\begin{equation}
\label{eqn:pf_Relation_rel_xFx_R_to_Phi_TwoFibers_5}
- R_{m+1} -A^{-2}P_{m+1} = -A^{-1}P_{m} = A^{m+1}Q_{m+1} - A^{m-3}Q_{m-1}.
\end{equation}
Therefore, by adding equations \eqref{eqn:pf_Relation_rel_xFx_R_to_Phi_TwoFibers_1}--\eqref{eqn:pf_Relation_rel_xFx_R_to_Phi_TwoFibers_5}, 
\begin{eqnarray*}
x_{\nu_{1}} F_{m} x_{-\nu_{2}} - R_{m+1} 
&=& -(A^{-3m-3} - A^{m+1}) (Q_{m+1} - Q_{m}) - (A^{-3m+1} - A^{m-3}) (Q_{m} - Q_{m-1}) \\
&=& -A^{-m-1}(q_{2m+2}(Q_{m+1} - Q_{m})+q_{2m-2}(Q_{m} - Q_{m-1})),
\end{eqnarray*}
which proves \eqref{eqn:pf_Relation_rel_xFx_R_to_Phi_TwoFibers_0} when $m \geq 1$.

Assume that $m \leq -1$.  Using \eqref{eqn:kbsm_xm}, \eqref{eqn:IdentityTwoFiberCase2}, and \eqref{eqn:kbsm_xvxm}, we see that
\begin{eqnarray}
x_{\nu_{1}}F_{m}x_{-\nu_{2}} &=& x_{\nu_{1}-m}x_{-\nu_{2}}
= A^{-2m}x_{\nu_{1}}x_{-\nu_{2}-m} + \sum_{i=0}^{-m-1}A^{2i}(P_{-\nu_{0}+m+2+2i}-A^{2}P_{-\nu_{0}+m+2i}) \notag\\
&=& A^{-2m}R_{-m+1} + \sum_{i=0}^{-m-1} A^{2i}P_{m+3+2i} - \sum_{i=0}^{-m-1} A^{2i+2}P_{m+1+2i}. \label{eqn:pf_Relation_rel_xFx_R_to_Phi_TwoFibers_6}
\end{eqnarray}
Since $P_{i} = -A^{i+2}Q_{i+1}+A^{i-2}Q_{i-1}$ (see \eqref{eqn:rel_Pn}), it follows that
\begin{eqnarray}
\sum_{i=-1}^{-m-2} A^{2i}P_{m+3+2i} &=& - \sum_{i=-1}^{-m-2} A^{m+5+4i} Q_{m+4+2i} + \sum_{i=-1}^{-m-2} A^{m+1+4i} Q_{m+2+2i} \notag\\
&=& - A^{-3m-3} Q_{-m} + A^{m-3} Q_{m} \label{eqn:pf_Relation_rel_xFx_R_to_Phi_TwoFibers_7}
\end{eqnarray}
and consequently,
\begin{equation}
\label{eqn:pf_Relation_rel_xFx_R_to_Phi_TwoFibers_8}
- \sum_{i=0}^{-m-1} A^{2i+2}P_{m+1+2i} = A^{-3m+1} Q_{-m} - A^{m+1} Q_{m}.
\end{equation}
Moreover, as it could easily be seen, \eqref{eqn:pf_Relation_rel_xFx_R_to_Phi_TwoFibers_4} and \eqref{eqn:pf_Relation_rel_xFx_R_to_Phi_TwoFibers_5} also hold for the case $m \leq -1$. Therefore, by adding equations \eqref{eqn:pf_Relation_rel_xFx_R_to_Phi_TwoFibers_4}--\eqref{eqn:pf_Relation_rel_xFx_R_to_Phi_TwoFibers_8}, 
\begin{eqnarray*}
x_{\nu_{1}} F_{m} x_{-\nu_{2}} - R_{m+1} 
&=& -(A^{-3m-3} - A^{m+1}) (Q_{m+1} - Q_{m}) - (A^{-3m+1} - A^{m-3}) (Q_{m} - Q_{m-1}) \\
&=& -A^{-m-1}(q_{2m+2}(Q_{m+1} - Q_{m})+q_{2m-2}(Q_{m} - Q_{m-1})),
\end{eqnarray*}
which proves \eqref{eqn:pf_Relation_rel_xFx_R_to_Phi_TwoFibers_0} when $m \leq -1$.

We showed that \eqref{eqn:pf_Relation_rel_xFx_R_to_Phi_TwoFibers_0} holds for all $m \in \mathbb{Z}$. Now let $u_{m} = q_{2m}$ and $B_{m} = Q_{m+1} - Q_{m}$, then one can easily check that
\begin{equation*}
u_{-m} = q_{-2m} = -q_{2m} = -u_{m}, \quad B_{-m} = Q_{-m+1} - Q_{-m} = -Q_{m-1} + Q_{m} = B_{m-1},
\end{equation*}
and
\begin{equation*}
u_{m+1} = (A^{-2}+A^{2})u_{m} - u_{m-1}.
\end{equation*}
Furthermore, $S_{m}$ defined in Lemma~\ref{lem:torsion_relations} becomes
\begin{equation*}
S_{m} = u_{m+1} \sum_{i=0}^{m-1}(-1)^{i}B_{m-i} = q_{2m+2}\varphi_{m} = \Phi_{m}
\end{equation*}
for $m \geq 1$, $S_{0} = 0 = \Phi_{0}$, $S_{-1} = u_{0}B_{0} = 0 = \Phi_{-1}$, and 
\begin{eqnarray*}
S_{m} &=& u_{m+1} \sum_{i=0}^{-m-1}(-1)^{i}B_{m+i+1} = -u_{-m-1}\sum_{i=0}^{-m-1}(-1)^{i}B_{-m-i-2} \\
&=& -S_{-m-2} - u_{-m-1}(-1)^{-m-2}(B_{0} - B_{-1}) = -S_{-m-2} = - \Phi_{-m-2} = \Phi_{m}
\end{eqnarray*}
for $m \leq -2$. It follows that $S_{m} = \Phi_{m}$ for all $m \in \mathbb{Z}$. Therefore, by \eqref{eqn:pf_Relation_rel_xFx_R_to_Phi_TwoFibers_0} and Lemma~\ref{lem:torsion_relations}
\begin{eqnarray*}
x_{\nu_{1}} F_{m} x_{-\nu_{2}} - R_{m+1} &=& -A^{-m-1}(q_{2m+2}(Q_{m+1} - Q_{m})+q_{2m-2}(Q_{m} - Q_{m-1})) \\
&=& -A^{-m-1}(u_{m+1}B_{m} + u_{m-1}B_{m-1}) \\
&=& -A^{-m-1}(\Phi_{m} + (A^{-2}+A^{2})\Phi_{m-1} + \Phi_{m-2}).   
\end{eqnarray*}
\end{proof}

\begin{lemma}
\label{lem:Relation_rel_xF_Fx_to_Psi_TwoFibers}
In $S\mathcal{D}({\bf D}^{2}_{\beta_{1}})$, 
for all $m\in \mathbb{Z}$,
\begin{equation*}
A^{m-2}(F_{m}x_{-\nu_{2}} - x_{\nu_{1}}F_{-1-m}) - A^{m-3}(F_{m-1}x_{-\nu_{2}} - x_{\nu_{1}}F_{-m}) = \Psi_{m} + (A^{-2}+A^{2})\Psi_{m-1} + \Psi_{m-2}.
\end{equation*}
\end{lemma}

\begin{proof}
We first show that
\begin{equation}
\label{eqn:pf_Relation_rel_xF_Fx_to_Psi_TwoFibers_0}
A^{m-2}(F_{m}x_{-\nu_{2}} - x_{\nu_{1}}F_{-1-m}) - A^{m-3}(F_{m-1}x_{-\nu_{2}} - x_{\nu_{1}}F_{-m}) = q_{2m+1}x_{\nu_{1}}Q_{m}+q_{2m-3}x_{\nu_{1}}Q_{m-1}
\end{equation}
for all $m \in \mathbb{Z}$. When $m = 0$, since $F_{0} = 1$ and $F_{-1} = -A^{3}$, it follows from \eqref{eqn:kbsm_xm} that
\begin{equation*}
F_{m}x_{-\nu_{2}} - x_{\nu_{1}}F_{-m-1} = F_{0}x_{-\nu_{2}} - x_{\nu_{1}}F_{-1} = x_{-\nu_{2}} + A^{3} x_{\nu_{1}} = x_{\nu_{1}+1} + A^{3} x_{\nu_{1}} = x_{\nu_{1}}F_{-1} + A^{3} x_{\nu_{1}} = 0
\end{equation*}
and 
\begin{eqnarray*}
F_{m-1}x_{-\nu_{2}} - x_{\nu_{1}}F_{-m} &=& F_{-1}x_{-\nu_{2}} - x_{\nu_{1}}F_{0} = -A^{3}x_{-\nu_{2}} - x_{\nu_{1}} = -A^{3}x_{\nu_{1}+1} - x_{\nu_{1}}  \\
&=& -A^{3}x_{\nu_{1}}F_{-1} - x_{\nu_{1}} = A^{3}q_{-3}x_{\nu_{1}},
\end{eqnarray*}
and consequently
\begin{equation*}
A^{m-2}(F_{m}x_{-\nu_{2}} - x_{\nu_{1}}F_{-1-m}) - A^{m-3}(F_{m-1}x_{-\nu_{2}} - x_{\nu_{1}}F_{-m}) 
= -q_{-3}x_{\nu_{1}},
\end{equation*}
so equation \eqref{eqn:pf_Relation_rel_xF_Fx_to_Psi_TwoFibers_0} holds for $m=0$.

Using a version of \eqref{eqn:rel_xm_to_Qxk} in $S\mathcal{D}({\bf D}^{2}_{\beta_{1}})$, we see that
\begin{equation*}
Q_{n}x_{k} = A^{-1}Q_{n-1}x_{k-1} + A^{n-1}x_{n+k-1},
\end{equation*}
for any $n,k \in \mathbb{Z}$ and by \eqref{eqn:kbsm_xm}, for $m \geq 1$,
\begin{equation*}
Q_{m}x_{-\nu_{2}} = \sum_{i = 0}^{m-1} A^{m-1-2i}x_{m-\nu_{2}-1-2i} = \sum_{i=0}^{m-1} A^{m-1-2i}x_{\nu_{1}}F_{-m+2i}.
\end{equation*}
Therefore,
\begin{eqnarray*}
F_{m}x_{-\nu_{2}} = (A^{-m}Q_{m+1} + A^{-m+2}Q_{m})x_{-\nu_{2}} = \sum_{i=0}^{m} A^{-2i}x_{\nu_{1}}F_{-m-1+2i} + \sum_{i=0}^{m-1} A^{1-2i}x_{\nu_{1}}F_{-m+2i}
\end{eqnarray*}
and consequently
\begin{eqnarray}
A^{m-2}(F_{m}x_{-\nu_{2}} - x_{\nu_{1}}F_{-1-m}) &=&
\sum_{i=1}^{m} A^{m-2-2i}x_{\nu_{1}}F_{-m-1+2i} + \sum_{i=0}^{m-1} A^{m-1-2i}x_{\nu_{1}}F_{-m+2i} \label{eqn:pf_Relation_rel_xF_Fx_to_Psi_TwoFibers_1} \\
&=& \sum_{i=1}^{m} A^{m-2-2i}x_{\nu_{1}}F_{-m-1+2i} + \sum_{i=1}^{m} A^{m+1-2i}x_{\nu_{1}}F_{-m-2+2i} \notag.
\end{eqnarray}
Replacing $m$ with $m-1$, we see that
\begin{equation}
\label{eqn:pf_Relation_rel_xF_Fx_to_Psi_TwoFibers_2}
-A^{m-3}(F_{m-1}x_{-\nu_{2}} - x_{\nu_{1}}F_{-m}) = -\sum_{i=1}^{m-1} A^{m-3-2i}x_{\nu_{1}}F_{-m+2i} - \sum_{i=1}^{m-1} A^{m-2i}x_{\nu_{1}}F_{-m-1+2i}.
\end{equation}
Notice that
\begin{equation}
\label{eqn:pf_Relation_rel_xF_Fx_to_Psi_TwoFibers_3}
\sum_{i=1}^{m} A^{m-2-2i}x_{\nu_{1}}F_{-m-1+2i} = \sum_{i=1}^{m} A^{2m-1-4i} x_{\nu_{1}}Q_{-m+2i} + \sum_{i=1}^{m} A^{2m+1-4i}x_{\nu_{1}}Q_{-m-1+2i},
\end{equation}
\begin{equation}
\label{eqn:pf_Relation_rel_xF_Fx_to_Psi_TwoFibers_4}
\sum_{i=0}^{m-1} A^{m-1-2i} x_{\nu_{1}}F_{-m+2i} = \sum_{i=0}^{m-1} A^{2m-1-4i}x_{\nu_{1}}Q_{-m+1+2i} + \sum_{i=0}^{m-1} A^{2m+1-4i}x_{\nu_{1}}Q_{-m+2i},
\end{equation}
\begin{eqnarray}
-\sum_{i=1}^{m-1} A^{m-3-2i}x_{\nu_{1}}F_{-m+2i} &=& - \sum_{i=1}^{m-1} A^{2m-3-4i}x_{\nu_{1}}Q_{-m+1+2i} - \sum_{i=1}^{m-1} A^{2m-1-4i}x_{\nu_{1}}Q_{-m+2i} \notag\\
&=& - \sum_{i=2}^{m} A^{2m+1-4i}x_{\nu_{1}}Q_{-m-1+2i} - \sum_{i=1}^{m-1} A^{2m-1-4i}x_{\nu_{1}}Q_{-m+2i} \label{eqn:pf_Relation_rel_xF_Fx_to_Psi_TwoFibers_5},
\end{eqnarray}
and
\begin{eqnarray}
-\sum_{i=1}^{m-1} A^{m-2i}x_{\nu_{1}}F_{-m-1+2i} &=& -\sum_{i=1}^{m-1} A^{2m+1-4i}x_{\nu_{1}}Q_{-m+2i}-\sum_{i=1}^{m-1} A^{2m+3-4i}x_{\nu_{1}}Q_{-m-1+2i} \notag\\
&=& -\sum_{i=1}^{m-1} A^{2m+1-4i}x_{\nu_{1}}Q_{-m+2i}-\sum_{i=0}^{m-2} A^{2m-1-4i}x_{\nu_{1}}Q_{-m+1+2i}. \label{eqn:pf_Relation_rel_xF_Fx_to_Psi_TwoFibers_6}
\end{eqnarray}
Using \eqref{eqn:pf_Relation_rel_xF_Fx_to_Psi_TwoFibers_1}--\eqref{eqn:pf_Relation_rel_xF_Fx_to_Psi_TwoFibers_6}, we see that
\begin{eqnarray*}
A^{m-2}(F_{m}x_{-\nu_{2}} - x_{\nu_{1}}F_{-1-m}) - A^{m-3}(F_{m-1}x_{-\nu_{2}} - x_{\nu_{1}}F_{-m}) = q_{2m+1}x_{\nu_{1}}Q_{m}+q_{2m-3}x_{\nu_{1}}Q_{m-1},
\end{eqnarray*}
which proves \eqref{eqn:pf_Relation_rel_xF_Fx_to_Psi_TwoFibers_0} for $m \geq 1$.

For $m \leq -1$, using a version of \eqref{eqn:rel_xm_to_Qxk} in $S\mathcal{D}({\bf D}^{2}_{\beta_{1}})$, we see that
\begin{equation*}
Q_{n}x_{k} = AQ_{n+1}x_{k+1} - A^{n+1}x_{n+k+1},
\end{equation*}
for any $n,k \in \mathbb{Z}$ and by \eqref{eqn:kbsm_xm},
\begin{equation*}
Q_{m}x_{-\nu_{2}} = - \sum_{i = 0}^{-m-1} A^{m+2i+1}x_{m-\nu_{2}+2i+1} = - \sum_{i = 0}^{-m-1} A^{m+2i+1}x_{\nu_{1}}F_{-m-2-2i}.
\end{equation*}
Therefore,
\begin{eqnarray*}
F_{m}x_{-\nu_{2}} = (A^{-m}Q_{m+1} + A^{-m+2}Q_{m})x_{-\nu_{2}} = -\sum_{i=0}^{-m-2} A^{2i+2}x_{\nu_{1}}F_{-m-3-2i} - \sum_{i=0}^{-m-1} A^{2i+3}x_{\nu_{1}}F_{-m-2-2i}
\end{eqnarray*}
and consequently
\begin{eqnarray}
A^{m-2}(F_{m}x_{-\nu_{2}} - x_{\nu_{1}}F_{-1-m}) &=&
-\sum_{i=-1}^{-m-2} A^{m+2i}x_{\nu_{1}}F_{-m-3-2i} - \sum_{i=0}^{-m-1} A^{m+2i+1}x_{\nu_{1}}F_{-m-2-2i} \notag \\
&=& -\sum_{i=0}^{-m-1} A^{m+2i-2}x_{\nu_{1}}F_{-m-1-2i} - \sum_{i=1}^{-m} A^{m+2i-1}x_{\nu_{1}}F_{-m-2i} \label{eqn:pf_Relation_rel_xF_Fx_to_Psi_TwoFibers_7}.
\end{eqnarray}
Replacing $m$ with $m-1$, we see that
\begin{eqnarray}
\label{eqn:pf_Relation_rel_xF_Fx_to_Psi_TwoFibers_8}
-A^{m-3}(F_{m-1}x_{-\nu_{2}} - x_{\nu_{1}}F_{-m}) &=& \sum_{i=0}^{-m} A^{m+2i-3}x_{\nu_{1}}F_{-m-2i} + \sum_{i=1}^{-m+1} A^{m+2i-2}x_{\nu_{1}}F_{-m+1-2i} \notag \\
&=& \sum_{i=0}^{-m} A^{m+2i-3}x_{\nu_{1}}F_{-m-2i} + \sum_{i=0}^{-m} A^{m+2i}x_{\nu_{1}}F_{-m-1-2i}.
\end{eqnarray}
Notice that
\begin{eqnarray}
-\sum_{i=0}^{-m-1} A^{m+2i-2}x_{\nu_{1}}F_{-m-1-2i} &=& -\sum_{i=0}^{-m-1} A^{2m+4i-1}x_{\nu_{1}}Q_{-m-2i} - \sum_{i=0}^{-m-1} A^{2m+4i+1}x_{\nu_{1}}Q_{-m-1-2i} \notag \\
&=& -\sum_{i=0}^{-m-1} A^{2m+4i-1}x_{\nu_{1}}Q_{-m-2i} - \sum_{i=1}^{-m} A^{2m+4i-3}x_{\nu_{1}}Q_{-m+1-2i}, \label{eqn:pf_Relation_rel_xF_Fx_to_Psi_TwoFibers_9}
\end{eqnarray}
\begin{eqnarray}
- \sum_{i=1}^{-m} A^{m+2i-1}x_{\nu_{1}}F_{-m-2i} = - \sum_{i=1}^{-m} A^{2m+4i-1}x_{\nu_{1}}Q_{-m-2i+1} - \sum_{i=1}^{-m} A^{2m+4i+1}x_{\nu_{1}}Q_{-m-2i} \label{eqn:pf_Relation_rel_xF_Fx_to_Psi_TwoFibers_10},
\end{eqnarray}
\begin{eqnarray}
\sum_{i=0}^{-m} A^{m+2i-3}x_{\nu_{1}}F_{-m-2i} = \sum_{i=0}^{-m} A^{2m+4i-3} x_{\nu_{1}}Q_{-m-2i+1} + \sum_{i=0}^{-m} A^{2m+4i-1}x_{\nu_{1}}Q_{-m-2i}  \label{eqn:pf_Relation_rel_xF_Fx_to_Psi_TwoFibers_11},
\end{eqnarray}
and
\begin{eqnarray}
\sum_{i=0}^{-m} A^{m+2i}x_{\nu_{1}}F_{-m-1-2i} &=& \sum_{i=0}^{-m} A^{2m+4i+1}x_{\nu_{1}}Q_{-m-2i} + \sum_{i=0}^{-m} A^{2m+4i+3}x_{\nu_{1}}Q_{-m-1-2i} \notag \\
&=& \sum_{i=0}^{-m} A^{2m+4i+1}x_{\nu_{1}}Q_{-m-2i} + \sum_{i=1}^{-m+1} A^{2m+4i-1}x_{\nu_{1}}Q_{-m+1-2i} \label{eqn:pf_Relation_rel_xF_Fx_to_Psi_TwoFibers_12}.
\end{eqnarray}
Using \eqref{eqn:pf_Relation_rel_xF_Fx_to_Psi_TwoFibers_7}--\eqref{eqn:pf_Relation_rel_xF_Fx_to_Psi_TwoFibers_12}, we see that
\begin{eqnarray*}
A^{m-2}(F_{m}x_{-\nu_{2}} - x_{\nu_{1}}F_{-1-m}) - A^{m-3}(F_{m-1}x_{-\nu_{2}} - x_{\nu_{1}}F_{-m}) = q_{2m+1}x_{\nu_{1}}Q_{m}+q_{2m-3}x_{\nu_{1}}Q_{m-1},
\end{eqnarray*}
which proves \eqref{eqn:pf_Relation_rel_xF_Fx_to_Psi_TwoFibers_0} for $m \leq -1$.

We showed that \eqref{eqn:pf_Relation_rel_xF_Fx_to_Psi_TwoFibers_0} holds for all $m \in \mathbb{Z}$. Now, let $u_{m} = q_{2m-1}$ and $B_{m} = x_{\nu_{1}}Q_{m}$, then one can check
\begin{equation*}
u_{-m} = q_{-2m-1} = -q_{2m+1} = -u_{m+1}, \quad B_{-m} = x_{\nu_{1}}Q_{-m} = -x_{\nu_{1}}Q_{m} = -B_{m},
\end{equation*}
and
\begin{equation*}
u_{m+1} = (A^{-2}+A^{2})u_{m} - u_{m-1}.
\end{equation*}
Furthermore, $S_{m}$ defined in Lemma~\ref{lem:torsion_relations} becomes
\begin{equation*}
S_{m} = u_{m+1} \sum_{i=0}^{m-1}(-1)^{i}B_{m-i} = q_{2m+1}\psi_{m-1} = \Psi_{m}
\end{equation*}
for $m \geq 1$, $S_{0} = 0 = \Psi_{0}$, $S_{-1} = u_{0}B_{0} = 0 = \Psi_{-1}$, and 
\begin{eqnarray*}
S_{m} &=& u_{m+1} \sum_{i=0}^{-m-1}(-1)^{i}B_{m+i+1} = u_{-m}\sum_{i=0}^{-m-1}(-1)^{i}B_{-m-i-1} \\
&=& S_{-m-1} + u_{-m}(-1)^{-m-1}B_{0} = S_{-m-1} = \Psi_{-m-1} = \Psi_{m}
\end{eqnarray*}
for $m \leq -2$. It follows that $S_{m} = \Psi_{m}$ for all $m \in \mathbb{Z}$. Therefore, by \eqref{eqn:pf_Relation_rel_xF_Fx_to_Psi_TwoFibers_0} and Lemma~\ref{lem:torsion_relations}
\begin{eqnarray*}
A^{m-2}(F_{m}x_{-\nu_{2}} - x_{\nu_{1}}F_{-m-1}) - A^{m-3}(F_{m-1}x_{-\nu_{2}} - x_{\nu_{1}}F_{-m}) &=& u_{m+1}B_{m} + u_{m-1}B_{m-1} \\
&=& \Psi_{m} + (A^{-2}+A^{2})\Psi_{m-1} + \Psi_{m-2}
\end{eqnarray*}
for any $m \in \mathbb{Z}$.
\end{proof}

\begin{corollary}
\label{cor:Submodules_1}
$S_{\nu_{2}}({\bf D}^{2}_{\beta_{1}}) = S_{2}(\Phi\oplus\Psi)$.
\end{corollary}

\begin{proof}
For any $m\in \mathbb{Z}$, by Lemma~\ref{lem:Relation_rel_xFx_R_to_Phi_TwoFibers} and the definition of $\Phi_{m}$, 
\begin{equation*}
x_{\nu_{1}} F_{m} x_{-\nu_{2}} - R_{m+1} \in S_{2}(\Phi\oplus\Psi)
\end{equation*}
and, by Lemma~\ref{lem:Relation_rel_xF_Fx_to_Psi_TwoFibers} and the definition of $\Psi_{m}$,
\begin{equation*}
A^{m-2}(F_ {m}x_{-\nu_{2}} - x_{\nu_{1}}F_{-1-m}) - A^{m-3}(F_{m-1}x_{-\nu_{2}} - x_{\nu_{1}}F_{-m}) \in S_{2}(\Phi\oplus\Psi).    
\end{equation*}
Since $F_{0}x_{-\nu_{2}} - x_{\nu_{1}}F_{-1} = 0$, it follows that
\begin{equation*}
F_{0}x_{-\nu_{2}} - x_{\nu_{1}}F_{-1} \in S_{2}(\Phi\oplus\Psi)
\end{equation*}
and consequently
\begin{equation*}
F_{m}x_{-\nu_{2}} - x_{\nu_{1}}F_{-m-1} \in S_{2}(\Phi\oplus\Psi)
\end{equation*}
for any $m \in \mathbb{Z}$. Therefore, 
\begin{equation*}
S_{\nu_{2}}({\bf D}^{2}_{\beta_{1}}) \subseteq S_{2}(\Phi\oplus\Psi).
\end{equation*}

By the definition, $\Phi_{0} = \Phi_{-1} = \Psi_{0} = \Psi_{-1} = 0$, so $\Phi_{0}, \Phi_{-1}, \Psi_{0}, \Psi_{-1} \in S_{\nu_{2}}({\bf D}^{2}_{\beta_{1}})$. So using Lemma~\ref{lem:Relation_rel_xFx_R_to_Phi_TwoFibers} and Lemma~\ref{lem:Relation_rel_xF_Fx_to_Psi_TwoFibers}, and induction on $m$, one can show that $\Phi_{m},\Psi_{m} \in S_{\nu_{2}}({\bf D}^{2}_{\beta_{1}})$ for any $m \geq 1$. Consequently, 
\begin{equation*}
S_{2}(\Phi\oplus\Psi) \subseteq S_{\nu_{2}}({\bf D}^{2}_{\beta_{1}}).    
\end{equation*}
\end{proof}

\begin{theorem}
\label{thm:final_thm_two_fibers}
For $\beta_{1}+\beta_{2} = 0$ the KBSM of $M_{2}(\beta_{1},\beta_{2}) = L(0,1)$ is generated by generic frame links with arrow diagrams in $\{\varphi_{m},\, \psi_{m} \mid m\geq 0\}$ and 
\begin{eqnarray*}
\mathcal{S}_{2,\infty}(L(0,1);R,A) &\cong&  R\{\varphi_{0}\} \oplus \bigoplus_{i = 1}^{\infty}\frac{R\{\varphi_{i}\}}{R\{q_{2i+2}\varphi_{i}\}}\oplus \bigoplus_{i = 1}^{\infty}\frac{R\{\psi_{i-1}\}}{R\{q_{2i+1}\psi_{i-1}\}}\\
&\cong&     
R\oplus \bigoplus_{i = 1}^{\infty}\frac{R}{(1-A^{2i+4})}.
\end{eqnarray*}
\end{theorem}

\begin{proof}
As we noted before,
\begin{equation*}
S\mathcal{D}({\bf D}^{2}_{\beta_{1}}) \cong R\Sigma^{\prime}_{\nu_{1}} \cong R\{\varphi_{m}\}_{m\geq 0}\oplus R\{\psi_{m}\}_{m\geq 0}.
\end{equation*}
Since
\begin{equation*}
S\mathcal{D}_{\nu_{1},\nu_{2}}\cong S\mathcal{D}({\bf D}^{2}_{\beta_{1}})/\ker(i_{*}),
\end{equation*}
and by Corollary~\ref{cor:SubmodulesOfSD} and Corollary~\ref{cor:Submodules_1}, 
\begin{equation*}
\ker(i_{*}) = S_{\nu_{2}}({\bf D}^{2}_{\beta_{1}}) = S_{2}(\Phi\oplus\Psi),    
\end{equation*}
it follows that
\begin{eqnarray*}
S\mathcal{D}_{\nu_{1},\nu_{2}} &\cong& (R\{\varphi_{m}\}_{m\geq 0}\oplus R\{\psi_{m}\}_{m\geq 0})/S_{2}(\Phi\oplus\Psi) \\
&=& (R\{\varphi_{m}\}_{m\geq 0}\oplus R\{\psi_{m}\}_{m\geq 0})/(R\{\Phi_{m}\}_{m\geq 1}\oplus R\{\Psi_{m}\}_{m\geq 1}).  
\end{eqnarray*}
Furthermore, $\Phi_{m} = q_{2m+2}\varphi_{m} = A^{-2m-2}(1-A^{4m+4})\varphi_{m}$ and $\Psi_{m} = q_{2m+1}\psi_{m-1} = A^{-2m-1}(1-A^{4m+2})\psi_{m-1}$, thus 
\begin{equation*}
S\mathcal{D}_{\nu_{1},\nu_{2}} \cong R\{\varphi_{0}\} \oplus \bigoplus_{i = 1}^{\infty}\frac{R\{\varphi_{i}\}}{R\{q_{2i+2}\varphi_{i}\}}\oplus \bigoplus_{i = 1}^{\infty}\frac{R\{\psi_{i-1}\}}{R\{q_{2i+1}\psi_{i-1}\}} \cong R\oplus \bigoplus_{i = 1}^{\infty}\frac{R}{(1-A^{2i+4})}.   
\end{equation*}
\end{proof}

\section*{Acknowledgement} Authors would like to thank Professor J\'{o}zef H. Przytycki for all valuable discussions and suggestions.

\bibliography{mybibfile}

\begin{bibdiv}
\begin{biblist}

\bib{DW2025}{unpublished}{
      author={Dabkowski, Mieczyslaw~K.},
      author={Wu, Cheyu},
       title={Basis for {KBSM} of fibered torus with multiplicity two
  exceptional fiber},
        date={2025},
        note={preprint,
  \href{https://arxiv.org/abs/2502.00912}{arXiv:2502.00912} [math.GT]},
}

\bib{HP1993}{article}{
      author={Hoste, Jim},
      author={Przytycki, J\'{o}zef~H.},
       title={The $(2,\infty)$-skein module of lens spaces; {A} generalization
  of the {J}ones polynomial},
        date={1993},
     journal={J. Knot Theory Ramifications},
      volume={2},
      number={3},
       pages={321\ndash 333},
}

\bib{HP1995}{article}{
      author={Hoste, Jim},
      author={Przytycki, J\'{o}zef~H.},
       title={The skein module of genus one {W}hitehead type manifolds},
        date={1995},
     journal={J. Knot Theory Ramifications},
      volume={4},
      number={3},
       pages={411\ndash 427},
}

\bib{Hud1969}{book}{
      author={Hudson, John F.~P.},
       title={Piecewise {L}inear {T}opology},
   publisher={W.A. Benjamin, Inc., New York-Amsterdam},
        date={1969},
}

\bib{JN1983}{book}{
      author={Jankins, Mark},
      author={Neumann, Walter~D.},
       title={{L}ectures on {S}eifert manifolds},
      series={Brandeis Lecture Notes},
   publisher={Brandeis University},
     address={Waltham, MA},
        date={1983},
      number={2},
  note={\href{https://mathscinet.ams.org/mathscinet-getitem?mr=741334}{MR:741334}},
}

\bib{KLH1987}{article}{
      author={Kauffman, Louis~H.},
       title={State models and the {J}ones polynomial},
        date={1987},
     journal={Topology},
      volume={26},
      number={3},
       pages={395\ndash 407},
}

\bib{M2011b}{article}{
      author={Mroczkowski, Maciej},
       title={Kauffman bracket skein module of the connected sum of two
  projective spaces},
        date={2011},
     journal={J. Knot Theory Ramifications},
      volume={20},
      number={5},
       pages={651\ndash 675},
        note={E-print: \href{https://arxiv.org/abs/1008.1007}{arXiv:1008.1007}
  [math.GT]},
}

\bib{MD2009}{article}{
      author={Mroczkowski, Maciej},
      author={Dabkowski, Mieczyslaw~K.},
       title={{KBSM} of the product of a disk with two holes and {$S^{1}$}},
        date={2011},
     journal={Topol. Appl.},
      volume={156},
      number={10},
       pages={1831\ndash 1849},
        note={E-print: \href{https://arxiv.org/abs/0808.3782}{arXiv:0808.3782}
  [math.GT]},
}

\bib{Prz1991}{article}{
      author={Przytycki, J\'{o}zef~H.},
       title={Skein modules of $3$-manifolds},
        date={1991},
     journal={Bull. Polish Acad. Sci. Math.},
      volume={39},
      number={1-2},
       pages={91\ndash 100},
        note={E-print:
  \href{https://arxiv.org/abs/math/0611797}{arXiv:math/0611797} [math.GT]},
}

\bib{Tur1990}{article}{
      author={Turaev, Vladimir~G.},
       title={The {C}onway and {K}auffman modules of a solid torus},
        date={1990},
     journal={J. Soviet Math.},
      volume={52},
      number={1},
       pages={2799\ndash 2805},
}

\end{biblist}
\end{bibdiv}

\end{document}